\documentclass[10pt]{amsart}
\usepackage{amsmath, amssymb, amsfonts, amsthm, array, enumerate, stmaryrd, enumitem}
\usepackage[all]{xy}
\usepackage[margin=1in]{geometry}

\makeatletter
\@namedef{subjclassname@2020}{\textup{2020} Mathematics Subject Classification}
\makeatother

\usepackage{tikz}
\usetikzlibrary{positioning}

\usepackage{xcolor}

\usepackage{multirow}
\usepackage{hyperref}
\hypersetup{
    colorlinks=true,
    linkcolor=blue,
    urlcolor =blue, 
    citecolor = magenta}

\numberwithin{equation}{subsection}

\usepackage[capitalize, nameinlink]{cleveref}

\theoremstyle{plain}
\newtheorem{mythm}{Theorem}[section]
\newtheorem*{mythm*}{Theorem}%[section]
\newtheorem{myprop}[mythm]{Proposition}
\newtheorem{mylemma}[mythm]{Lemma}
\newtheorem{mycor}[mythm]{Corollary}

\newtheorem{thma}{Theorem}
\theoremstyle{definition}

\theoremstyle{remark}
\newtheorem{remarkx}[mythm]{Remark}
\newenvironment{remark}
  {\pushQED{\qed}\remarkx}
  {\popQED\endremarkx}

\Crefname{subsection}{subsection}{subsections}
\Crefname{section}{section}{section}
\crefname{myprop}{Proposition}{Propositions}
\crefname{mylemma}{Lemma}{Lemmas}
\crefname{thma}{Theorem}{Theorems}

\def\R{\mathbb{Q}}

\def\Q2{\mathbb Q_2}

\def\Ga1{\Gamma_1}

\def\tr{\operatorname{tr}}
\DeclareFontFamily{U}{wncy}{}
\DeclareFontShape{U}{wncy}{m}{n}{<->wncyr10}{}
 \DeclareSymbolFont{mcy}{U}{wncy}{m}{n}

\newcommand{\toiso}{\xrightarrow[\raisebox{0.25em}{\smash{\ensuremath{\sim}}}]{}}
\renewcommand{\phi}{\varphi}
\DeclareMathOperator{\Hom}{Hom}
\newcommand{\ZZ}{\mathbb Z}
\newcommand{\QQ}{\mathbb Q}
\newcommand{\FF}{\mathbb F}
\newcommand{\CC}{\mathbb C}
\newcommand{\CCC}{{\mathcal C}}
\renewcommand{\sl}{\mathop{\rm SL}}

\newcommand{\gl}{\mathop{\rm GL}}
\newcommand{\lb}{\llbracket}
\newcommand{\rb}{\rrbracket}
\DeclareMathOperator{\im}{\rm im}
\newcommand{\onto}{\twoheadrightarrow}

\newcommand{\into}{\hookrightarrow}
\mathchardef\myhyphen="2D
\DeclareMathOperator{\End}{\rm End}
\DeclareMathOperator{\frob}{{\rm Frob}}
\DeclareMathOperator{\Frob}{{\rm Frob}}
\DeclareMathOperator{\Gal}{\rm Gal}
\DeclareMathOperator{\gal}{\rm Gal}
\DeclareMathOperator{\lcm}{{\rm lcm}}
\newcommand{\mm}{{\mathfrak m}}
\newcommand{\nn}{{\mathfrak n}}

\newcommand{\dd}{\Delta}
\newcommand{\old}{{\rm old}}
\newcommand{\univ}{{\rm univ}}
\newcommand{\new}{{\rm new}}

\newcommand{\hecke}{{\rm mod}}
\let\tan\relax
\DeclareMathOperator{\tan}{{\rm Tan}}

\DeclareMathOperator{\tandim}{\dim {\rm Tan}}

\newcommand{\rhobar}{{\bar\rho}}
\newcommand{\rhob}{\one}

\newcommand{\red}{{\rm red}}
\newcommand{\pf}{{\rm pf}}

\newcommand{\eps}{\varepsilon}

\newcommand{\charpoly}{\mathcal P}

\newcommand{\RRR}{{\mathcal R}}
\newcommand{\PPP}{{\mathcal P}}
\newcommand{\one}{{1 \oplus 1}}

\newcommand{\levn}{{}}
\newcommand{\protwo}{{\rm pro\myhyphen2}}

\newcommand{\Gprotwo}{G^\protwo}
\newcommand{\pp}{\mathfrak p}

\newcommand{\OO}{\mathcal O}
\newcommand{\bellaiche}{Bella\"iche}
\newcommand{\jb}{\bellaiche}

\newcommand{\thm}{Th\'eor\`eme}
\newcommand{\frat}{\Pi}

\newcommand{\barr}[1]{\mkern 1.5mu\overline{\mkern-1.5mu#1\mkern-1.5mu}\mkern 1.5mu}

\newcommand{\elleng}{{\rm vnew}}
\newcommand{\verynew}{very\ new}

\newenvironment{psmallmatrix}
	{\left( \begin{smallmatrix}}
	{\end{smallmatrix} \right)}
\newcommand{\cmod}[1]{\ \mathrm{mod}\ #1}

\begin{document}
\title{Mod-$2$ Hecke algebras of level $3$ and $5$}
\author{Shaunak V. Deo} 
\email{shaunakdeo@iisc.ac.in}
\address{Department of Mathematics, Indian Institute of Science, Bangalore 560012, India}
\author{Anna Medvedovsky}
\email{medvedov@post.harvard.edu \qquad {\rm \emph{Website:}} \url{https://people.mpim-bonn.mpg.de/medved/}}
\address{Max-Planck-Institut für Mathematik, Vivatsgasse 7, 53113 Bonn, Germany} 
\date{\today}


\dedicatory{Dedicated to the memory of Jo\"{e}l Bella\"{i}che.}

\begin{abstract}
We use deformation theory to study the big Hecke algebra acting on mod-2 modular forms of prime level $N$ and all weights, especially its local component at the 
trivial representation. For $N = 3, 5$, we prove that the {maximal reduced quotient} of this big Hecke algebra is isomorphic to the maximal reduced quotient of the corresponding universal deformation ring. Then we completely determine the structure of this big Hecke algebra.
We also describe a natural grading on mod-$p$ Hecke algebras. 
\end{abstract}

%%%%%%%%%%%%%%%%%%%%%%%%%%%%%%%%%%%%%%%%%%%%
%%%%%%%%%%%%%%%%%%%%%%%%%%%%%%%%%%%%%%%%%%%%
%%%%%%%%%%%%%%%%%%%%%%%%%%%%%%%%%%%%%%%%%%%%s

\maketitle

\setcounter{tocdepth}{1} % This sets the "depth" of the TOC. 
\begingroup
\makeatletter
\parskip\z@skip
\def\l@subsection{\@tocline{2}{0pt}{2.5pc}{5pc}{}}
\tableofcontents
\endgroup

%%%%%%%%%%%%%%%%%%%%%%%%%%%%%%%%%%%%%%%%%%%%
\section{Introduction}
%%%%%%%%%%%%%%%%%%%%%%%%%%%%%%%%%%%%%%%%%%%%
%%%%%%%%%%%%%%%%%%%%%%%%%%%%%%%%%%%%%%%%%%%%

We study the local component at the trivial representation of \emph{big} mod-$2$ Hecke algebras acting on all mod-$2$ modular forms of a fixed level $N$, focusing on $N = 3$ and $N = 5$. To describe the history of the study of mod-$p$ Hecke algebras and our contributions, we introduce some notation, informally at first. 

For an integer $N$ and a prime $p$ not dividing $N$, write $M(N, \FF_p)$ for the space of modular forms modulo~$p$ of level $\Gamma_0(N)$ and any weight in the sense of Serre and Swinnerton-Dyer \cite{SerreBBK, SwDy}. Let $A(N, \FF_p)$ be the shallow Hecke algebra acting on $M(N, \FF_p)$: 
this is the closed 
subalgebra of $\End_{\FF_p}\!\big(M(N, \FF_p)\big)$ topologically generated by the Hecke operators $T_n$ for $n$ not dividing $Np$. 
Then $A(N, \FF_p)$ is a semilocal noetherian ring, which splits as a product of its completions at maximal ideals,
corresponding,
up to $\gal(\bar\FF_p/\FF_p)$-conjugacy, to semisimple Galois representations $\rhobar: G_{\QQ} \to \gl_2(\bar\FF_p)$ arising from eigenforms appearing in $M(N, \bar\FF_p)$.
Passing to a large enough finite extension $\FF/\FF_p$ to resolve the Galois conjugacy, let $A(N, \FF)_\rhobar$ denote the local component of $A(N, \FF_p) \otimes_{\FF_p} \FF$ corresponding to the Galois representation $\rhobar$.

The main object of study here is the complete noetherian local $\FF$-algebra $A(N, \FF)_\rhobar$ in the special case that~$p = 2$, $\rhobar = 1 \oplus 1$, and $N$ is a prime eventually specializing to $3$ or $5$.

\subsection{Historical context} 
The structure of $A(N, \FF)_\rhobar$ was first studied in the late 70s by Jochnowitz, who proved that $A(N, \FF)_\rhobar$ is an infinite-dimensional $\FF$-vector space \cite{J}. In the late 90s, Khare observed that deformation theory {implies that} 
$A(N, \FF)_\rhobar$ is noetherian, 
so that Jochnowitz's result may be reinterpreted in terms of Krull dimension:  $\dim A(N, \FF)_\rhobar \geq 1$
 \cite{K}. 
 There were no further developments until Nicolas and Serre revitalized the 
field of study of mod-$p$ Hecke algebras in~2012;
we survey the subsequent progress.
\begin{itemize}[topsep = -3pt, itemsep = 4pt]
\item[2012] Nicolas and Serre use the Hecke recurrence in characteristic $2$ (see proof of \cref{nilpy} for an example thereof) to show that $A(1, \FF_2)$ is a regular local $\FF_2$-algebra of dimension $2$. More precisely, they prove that $A(1, \FF_2) \cong \FF_2\lb T_3, T_5 \rb$.
\item[2015] Bella\"{i}che and Khare use very different methods --- careful comparison with the characteristic-zero Hecke algebra, known to be big by the Gouv\^ea-Mazur infinite fern {(\cite[Theorem 1]{GM}; see also \cite[Corollary 2.28]{Em})}\label{fin} ---  to show that that for $N = 1$ and $p \geq 5$ the Krull dimension of $A(N, \FF)_\rhobar$ is always at least 2, and that $A(N, \FF)_\rhobar \cong \FF\lb x, y \rb$ whenever $\rhobar$ is deformation-theoretically \emph{unobstructed}~\cite{BK}. 
\item[2017] Deo (first-named author here) generalizes the \jb-Khare result \cite{BK} to all $N \geq 1$, still {under the assumption} $p \geq 5$ \cite{D}. 
\item[2018] Medvedovsky (second-named author) determines the structure of $A(1, \FF_3)$,  developing the \emph{nilpotence method} (\cref{nilpthm} or see \cite{M, Mpaper}), which gives the bound $\dim A(N, \FF)_\rhobar \geq 2$  so long as the genus of $X_0(Np)$ is zero (forthcoming) by comparing the weight filtration on $M(N, \FF)$ with the nilpotence filtration. This method builds on ideas of Bella\"{i}che \cite[appendix]{BK}; like the Nicolas-Serre method it relies on the Hecke recurrence. 

\item[$\geq 2012$] Monsky studies  $A(N, \FF_p)_\rhobar$ and its subquotients for various small $N$ and $p$ by considering recurrences and power series in characteristic $p$ in the spirit of Nicolas-Serre; see \url{https://mathoverflow.net/users/6214/paul-monsky} for many conjectures. In particular, he determines the structure of
$A(3, \FF_2)$ and $A(5, \FF_2)$: see \cite{Mo3alg, Mo5alg, Mo35}. We recover these structure theorems here (see \cref{B}) using completely different methods. 
\end{itemize}

All this progress has left unresolved 
most of the cases for $p = 2, 3$, where the best Krull dimension bound for the mod-$p$ Hecke algebra is still the Jochnowitz-Khare bound $\dim A(N, \FF)_\rhobar \geq 1$,
as well as the many cases where $\rhobar$ is obstructed 
and 
the precise structure of $A(N, \FF)_\rhobar$ is not known. 

The aim of the present paper is to explore how much information one can get about the structure of~$A(N, \FF_2)_{1 \oplus 1}$ 
for $N$ prime by using deformation theory of 
Chenevier pseudorepresentations
combined with the nilpotence method. 
In particular, we prove that the maximal reduced quotients of $A(3,\FF_2)$ and $A(5,\FF_2)$ are isomorphic to those of suitable deformation rings, whose structure we determine explicitly.
We also determine the structure of $A(3, \FF_2)$ and $A(5, \FF_2)$ completely, recovering the structure results of Monsky \cite{Mo35} obtained by entirely different methods. 
In our proof, we use a purely local deformation condition that we call \emph{level-$N$ shape} to restrict {our} attention to {deformations} that look like those coming from modular forms of level $N$ (vs.~a power of $N$). The level-$N$ shape condition may be interpreted as a coarser version of Wake and Wang-Erickson's 
Steinberg condition if it were extended to $p = 2$ and beyond~$k = 2$~\cite{WWE}.
 Finally, we identify a natural grading on a mod-$p$ Hecke algebra by the 2-Frattini quotient of the Galois group, compatible with the grading on the universal constant-determinant deformation ring described, for~$p = 2$ and $\rhobar = \one$, by~Bella\"iche. 

\subsection{Main results} \label{structsec}
We prove a number of structure theorems, of the Hecke algebras $A(3, \FF_2)$ and $A(5, \FF_2)$, and of various related rings, including deformation rings, which we now introduce informally. For precise definitions, see \cref{defnote}. 

Let $N$ be an odd integer and $G_{\QQ, 2N}$  the Galois group of the maximal extension of $\QQ$ unramified outside $2N$. We consider lifts of the trivial representation $\one: G_{\QQ, 2N} \to \gl_2(\FF_2)$ (viewed as a Chenevier \emph{pseudorepresentation}, \cref{psrep}) to local pro-$2$ $\FF_2$-algebras with residue field $\FF_2$ subject to two conditions: the lifts must be \emph{odd} (here: trace of complex conjugation is $0$) and their determinant must be \emph{constant} (here: determinant is $1$). Such lifts are parametrized by a noetherian universal deformation ring $\hat \RRR(N, \FF_2)_\one$.

The $2$-adic Galois representations attached to classical modular forms of level $N$ glue together to form the modular pseudorepresentation of $G_{\QQ, 2N}$ taking values in $A(N, \FF_2)_\one$, so that the universal property of~$\hat \RRR(N, \FF_2)_\one$ gives us a unique local $\ZZ_2$-algebra morphism, a surjection,
\begin{equation}\label{hatphintro} 
\hat \varphi: \hat \RRR(N, \FF_2)_\one \onto A(N, \FF_2)_\one
\end{equation}
inducing the modular pseudorepresentation from the universal one (\cref{pdonA}). 
One does not expect this map to be an isomorphism for general $N$, and certainly not for $N$ prime (see \cref{levelnshapeintro}). 
 Our first result is that for $N = 3, 5$, this map is in fact an isomorphism on reduced quotients. Note that for $N = 3, 5$ the trivial mod-$2$ representation is the only (semisimple) modular one, so that $A(N, \FF_2) = A(N, \FF_2)_\one$ (\cref{loc}). 

\begin{thma}[see \cref{isomcor} and \cref{mainthm}]\label{a1}
For $N = 3$, $5$, the surjection $\hat\RRR(N, \FF_2)^\levn_\one \onto A(N, \FF_2)$ induces an isomorphism $$\hat\RRR(N, \FF_2)_\one^{\red} \toiso A(N, \FF_2)^\red.$$ 
Moreover, there are explicitly describable Hecke operators $X, Y, Z \in A(N, \FF_2)$ with images $\bar X, \bar Y, \bar Z$ in $A(N, \FF_2)^\red$, respectively, and power series $f$ and $g$ in two variables over $\FF_2$ so that the map 
$$\frac{\FF_2\lb x, y, z\rb}{\big(z - f(x, y)\big)\big(x - g(y, z)\big)} 
\longrightarrow
A(N, \FF_2)^\red$$
defined by $x \mapsto \bar X$,  $y \mapsto \bar Y$, and $z \mapsto \bar Z$ is an isomorphism. 
\end{thma}

\cref{a1} gives us, in completely explicit terms, the structure of reduced deformation rings ${\hat\RRR(3, \FF_2)_{\one}^\red}$ and ${\hat\RRR(5, \FF_2)_{\one}^\red}$, not previously known. We know of no other way of determining the structure of these rings.
Moreover, \cref{a1} together with the fact that the tangent dimension of $A(N, \FF_2)$ is $4$ (see \cref{gen3,gen5}) immediately implies that $A(N, \FF_2)$ is not reduced, the first explicit examples of nonreduced mod-$p$ Hecke algebras for $\Gamma_0(N)$.\footnote{Examples of nonreduced Hecke algebras for $\Gamma_1(N)$ were previously found by the first-named author in \cite{D}, but the nilpotent elements came from diamond operators.} 
We also give the following refinement of \cref{a1} on the Hecke algebra side.

\begin{thma}[see \cref{Astructurethm}]\label{B}
For $N = 3$ and $5$, there exist Hecke operators $X, Y, Z, W$ in $A(N, \FF_2)$ such that the map 
$$ \frac{\FF_2\lb x, y, z, w\rb}{(xz, xw, (z + w)^2)} 
\longrightarrow
A(N, \FF_2)$$
sending $x \mapsto X$, $y \mapsto Y$, $z \mapsto Z$, and $w \mapsto W$ is an isomorphism. 
\end{thma}

In \cref{B} we have the first explicit structure of a mod-$p$ Hecke algebra that is not a regular local ring. In particular, it is clear that $A(3, \FF_2)$ and $A(5,\FF_2)$ are not Gorenstein.

Additionally, we determine the structure of $A(N, \FF_2)^\new$, the Hecke algebra acting on the space of mod-$2$ newforms in the sense of \cite{deomed}, and the structure of $A(N, \FF_2)^\pf$, the \emph{partially full} Hecke algebra topologically generated by the action of Hecke operator $U_N$ as well as the $T_m$ for $m \nmid 2N$. 
See \cref{pfprop} and \cref{anewstructure}. 

On the deformation side, although we do not prove a structure theorem for $\hat\RRR(N, \FF_2)_\one$ or its level-$N$ quotient described in \cref{levelnshapeintro}, we do prove an $R = \mathbb T$ theorem for $A(N, \FF_2)^\pf$. See \cref{rtpfsec} for the relevant definitions and \cref{rt} for the exact statement.

\subsubsection{The grading on mod-$p$ Hecke algebras}\label{gradingintro}
The full statement of our Hecke algebra structure result (\cref{Astructurethm}) describes a natural $(\ZZ/8\ZZ)^\times$-grading on $A(N, \FF_2)_\one$, one that generalizes to mod-$p$ Hecke algebras and to local components at dihedral $\rhobar$: that is, $\rhobar$ for which $\rhobar \simeq \rhobar \otimes \omega_p^{(p-1)/2}$, where $\omega_p$ is the mod-$p$ cyclotomic character. Let $\Pi_{pN}$ be the $2$-Frattini quotient of $G_{\QQ, Np}$.
\begin{thma}[see \cref{gradingthm}] \label{c} 
The Hecke algebra $A(N, \FF)$ has a $\frat_{p}$-grading $A(N, \FF)= \bigoplus_{i \in \frat_p} A(N, \FF)^i$, natural in the sense that for $m \nmid Np$ we have $T_m \in A(N, \FF)^m$.

\vspace{-3pt}
If $p = 2$ or if $\rhobar: G_{\QQ, Np} \to \gl_2(\FF)$ satisfies $\rhobar \simeq \rhobar \otimes \omega_p^{(p-1)/2}$, then the same is true for $A(N, \FF)_\rhobar$. 
\end{thma}

In the statement of \cref{c} we've implicitly identified $\Pi_p$, the $2$-Frattini quotient of $G_{\QQ, p}$, with $\ZZ_p^\times/(\ZZ_p^\times)^2$. \cref{gradingthm} itself says more: there is a corresponding $\Pi_{p}$-grading on the space $K(N, \FF)$ of forms killed by $U_p$ making $K(N, \FF)$ into a  graded $A(N, \FF)$-module, and the grading is compatible with the modular pseudorepresentation map~$G_{\QQ, Np} \to A(N, \FF)$, in the sense that the fibers of $G_{\QQ, Np} \onto G_{\QQ, p} \onto \frat_p$ map to the corresponding graded components of $A(N, \FF)$. 
These statements restrict to $A(N, \FF)_\rhobar$ and the $\rhobar$-generalized eigenform subspace $K(N, \FF)_\rhobar$ of $K(N, \FF)$ under the assumptions on $\rhobar$ as above. 

The grading in \cref{c} had been previously discovered in two special cases: $(p, N) = (2, 1)$~\cite{NS1} and~$(p, N) = (3, 1)$~\cite{medved3}. \cref{c} was inspired by a question of Serre, as well as a partial answer from \jb, which we take the opportunity to present below. 

\begin{mythm*}[\jb; see \cref{bellgrade}]\label{bellgrade} For any odd $N$, the universal deformation ring $\hat \RRR(N, \FF_2)_\one$ has a natural $\frat_{2N}$-grading compatible with the universal pseudorepresentation map $G_{\QQ, 2N} \to \hat \RRR(N, \FF_2)_\one$.
\end{mythm*}

\subsection{Overview of proofs} 

We give a rough outline of the proofs of \cref{a1,B}. The proof of \cref{a1} has a spiral nature that may of be independent interest --- we move several times between the deformation side and the Hecke side to achieve our results. The proof of \cref{c} is more straightforward; see \cref{gradingsec}. 

For the proof of \cref{a1}, we first construct the universal deformation ring $\RRR(N, \FF_2)_\one$ subject to our local level-$N$-shape condition. For more on this condition, see \cref{levelnshapeintro} below or \cref{steinberg}. 
We use a computation of Chenevier to prove a result about the dimension of the tangent {space} of this restricted deformation ring $\RRR(N, \FF_2)_\one$ (\cref{tandim}). 
{Using true representations}, we get a bound on the dimension of the tangent space of quotients of $\RRR(N, \FF_2)_\one$ by prime ideals (\cref{tanlemma}).
Then we use this information to prove results about quotients of the Hecke algebra.
 In particular, we use the nilpotence method to determine the structure of $A(N, \FF_2)^\elleng$, the Hecke algebra acting on the space of forms killed by~$U_N + 1$, which we call the \emph{very new} modular forms because the newforms are those killed by $U_N^2 - 1 = (U_N + 1)^2$ (\cref{Avnewstruct}). Because $A(N, \FF_2)^\elleng$ has dimension $2$, we are able to conclude 
 that the minimal prime ideals of $\RRR(N, \FF_2)_\one$ are preimages of minimal primes of $A(N, \FF_2)$, so that the two rings have the same reduced quotient (\cref{minprimes}). Separately, we determine the structure of $A(N, \FF_2)^\red$ (\cref{mainthm}), completing the proof of \cref{a1}; in particular, the two integral-domain quotients of $A(N, \FF_2)^\red$, visible in the statement of \cref{a1}, are the old Hecke algebra $A(1, \FF_2)$ and very new Hecke algebra $A(N, \FF_2)^\elleng$.

The proof of \cref{B} is quite involved in its own way. We first determine the structure of the {partially full} Hecke algebra (\cref{pfprop}).
 Then we use the partially full Hecke algebra, the two integral domain quotients found in the proof of \cref{a1}, and the precise description of the cotangent space $A(N, \FF_2)$ to find the structure of $A(N, \FF_2)$ (\cref{Astructurethm}).

\subsection{The level-$N$-shape deformation condition}\label{levelnshapeintro}
A key tool in the proof of \cref{a1} is our purely local \emph{level-$N$-shape} deformation condition, which may be defined for any prime $p$ and prime level $N$. The condition is simple to describe: for a pseudorepresentation lifting $\rhobar:G_{\QQ, Np} \to \gl_2(\FF)$ we ask that the restriction to the decomposition group at $N$ contains the inertia subgroup in the kernel (\cref{steinberg}).  This captures the property of a representation with Artin conductor dividing~$N$, such as those contributing to~$A(N, \FF)_\rhobar$, rather than those whose Artin conductor is a power of~$N$, which may appear in $A(N^2, \FF)_\rhobar$.

In our setting, the {level-$N$-shape} deformation condition defines the universal level-$N$-shape deformation ring~$\RRR(N, \FF_2)_\one$, a nontrivial quotient of $\hat \RRR(N, \FF_2)_\one$ by a nilpotent ideal (\cref{redlemma}). The universality induces the surjection
$$\phi: \RRR(N, \FF_2)_\one \onto A(N, \FF_2)_\one$$
factoring $\hat \phi$ from \eqref{hatphintro}. It is this map rather than $\hat \phi$ that we use in the proof of \cref{a1}. 

We now compare our level-$N$-shape condition to both the \emph{unramified-or-$\pm$-Steinberg} at-$N$ condition appearing in the work of Wake and Wang-Erickson \cite{WWE} and the \emph{ordinary-at-$N$} condition from the work of Calegari--Specter \cite[source file (!) on arxiv]{CalegariSpecter}. For this it is helpful to recall that a constant-determinant pseudorepresentation lifting $\rhobar: G_{\QQ, Np} \to \gl_2(\FF)$ to a pro-$p$ local $\FF$-algebra $B$ with residue field $\FF$ is described by a function $t: G_{\QQ, Np} \to B$ lifting $\tr \rhobar$. {In particular, $t$ is \emph{central}: $t(gh) = t(hg)$ for all $g, h \in G_{\QQ, Np}$.} For more on $t$, 
see \cref{psrep}. Also let $D_N \subseteq G_{\QQ, Np}$ be {a decomposition subgroup at $N$}, $\frob_N \in D_N$ any Frobenius element, and $I_N \subseteq D_N$ the {inertia-at-$N$ subgroup of~$D_N$}. The pseudorepresentation $t$ satisfies the level-$N$-shape condition if 
\begin{equation}
\label{levelNintro} 
t(di) = t(d)\qquad \mbox{for $i \in I_N$, $d \in D_N$}.
\end{equation}
The Wake--Wang-Erickson condition is stated for residually multiplicity-free $p$-adic pseudorepresentations with $p \geq 5$ in fixed weight $2$, and expressed in terms of \emph{generalized matrix algebras} (GMAs), introduced earlier by Bellaïche and Chenevier \cite[Chapter 1]{BC}. It is possible to obtain an equivalent formulation of the Wake--Wang-Erickson GMA relation as a pseudorepresentation statement involving all the elements of~$G_{\QQ,Np}$ by using the notion of the kernel of a pseudorepresentation (of an algebra; see \cite[\S2.1.2]{BellaicheImages}). It is also straightforward to generalize their GMA relation to weight $k$.  With these adjustments, we might expect that Wake and Wang-Erickson 
would 
call 
a weight-$k$ pseudorepresentation $t: G_{\QQ,Np} \to B$ \emph{unramified-or-$\pm$Steinberg at $N$}
if the following conditions hold for all~$g \in G_{\QQ, Np}$ and $i \in I_N$:
\begin{align}
\label{wwe} t(g \Frob_N i) - t(g \frob_N) &= \mp N^{\frac{k}{2}} \big(t(gi) -  t(g) \big)\qquad \mbox{and}\\
\label{wwe2} t(g i\Frob_N ) - t(g \frob_N) &= \mp N^{\frac{k-2}{2}} \big(t(gi) -  t(g) \big).
\end{align}
Compare to \cite[Definitions~3.4.1 and 3.8.1]{WWE}. 
If $t$ is an unramified-or-$\pm$Steinberg at $N$ pseudorepresentation of weight~$k$, then \cite[Lemma~3.4.4]{WWE} implies that~$t(i)=2$ for all $i \in I_N$.
In particular, taking $g=1$ in~\eqref{wwe} above recovers our local level-$N$-shape condition.
Conversely, the level-$N$ shape condition implies \eqref{wwe} and \eqref{wwe2} for~$g$ in the local Galois group~$D_N$, but not more generally for~$g$ in $G_{\QQ, Np}$. This captures a key distinction between ours and the Wake--Wang-Erickson condition: our condition \eqref{levelNintro} is entirely \emph{local}, whereas the Wake--Wang-Erickson condition \cite[Definition 3.8.1]{WWE} is \emph{global} --- as reflected, necessarily, in our translation-cum-generalization in \eqref{wwe}--\eqref{wwe2}. Let us briefly dwell on this confusing point: Definition~3.8.1 in~\cite{WWE}, which relies on the purely local Definition 3.4.1 of loc.~cit., requires elements in the group algebra of a local Galois group to map to zero in a ``Cayley-Hamilton" (close to ``faithful") GMA carrying a representation of a  global Galois group. It is therefore fundamentally global in nature.

The Calegari-Specter ordinary-at-$N$ condition from \cite{CalegariSpecter} is similar to Wake and Wang-Erickson's. Calegari and Specter enhance the deformation ring with a new variable $U$ meant to capture the behavior of $U_N$. Their condition holds if $U$, in addition to satisfying the characteristic polynomial of $\frob_N$ (see \eqref{charpoly}), satisfies the following: for every $g \in G_{\QQ, Np}$ and every $i \in I_N$, 
\begin{equation}\label{CScond}
{t(g i \frob_N) - t(g \frob_N) =  \big(t(gi) - t(g) \big) U = 0.}
\end{equation}
As one can see, the difference between {\eqref{wwe2}} and \eqref{CScond} is the substitution of $U$ for {$\pm N^{k/2-1}$},
as one expects from the action of $U_N$ on newforms. Depending on the context, one may also expect to add a second condition here analogous to \eqref{wwe}.
We use an identity analogous to \eqref{CScond} for $g = c$ (see \eqref{funcp})
to describe the precise structure of $A(N, \FF_2)_\one$ for $N = 3, 5$ (\cref{Astructurethm}). And we use \eqref{CScond} and the related condition on the other side to obtain an $R = \mathbb T$ theorem for $A(N, \FF_2)^\pf$ (\cref{rt}).

\subsection{Comparison with Monsky's results} \label{monskysec}
Finally we briefly describe the recent results of Paul Monsky, which inspired and catalyzed both  \cite{deomed} and the present work, and relate them to ours. 
In \cite{Mo3alg}, motivated by recent work of Nicolas and Serre on the mod-$2$ level-$1$ Hecke algebra \cite{NS1,NS2}, 
Monsky was able to determined that the Hecke algebra acting on a certain subquotient of the space $M(3, \FF_2)$ of mod-$2$ level-$3$ modular forms, which he argued ought to be viewed as the newforms in this setting, is isomorphic to $\FF_2 \lb T_{7}, T_{13}, \eps\rb/(\eps^2)$. In~\cite{Mo5alg} he proved similar results in level $5$. Monsky's structure results were obtained by explicit computations with mod-$2$ power series and Nicolas--Serre--style Hecke recurrences; though beautiful and satisfying, they did not appear to be amenable to generalization. Our curiosity piqued, we sought to reinterpret Monsky's results in a more conceptual way. This required several steps. First, in \cite{deomed} we defined a space $M(N, \FF_p)^\new$ of newforms mod $p$ for any prime-to-$p$ level $N$, and proved that Monsky's level-$3$ Hecke algebra from \cite{Mo3alg} is isomorphic to the Hecke algebra acting on $M(3, \FF_2)^\new$. Next, we used deformation theory, commutative algebra, and a dash of the nilpotence method to arrive at \cref{a1}, a structure theorem for $A(N, \FF_2)^\red$. We also conjectured the statement of \cref{B}. After mutually beneficial discussions with Monsky, both he and we were able to sharpen our results using our independent methods: Monsky's in \cite{Mo35} and ours as \cref{B} here. 

\subsection*{Acknowledgements} The authors thank Jo\"{e}l Bella\"{i}che for introducing them to this satisfyingly explicit subject. They also thank Paul Monsky, who initially obtained closely related results on the Hecke algebra side without  deformation theory; his investigations motivated the present one. Subsequent communication in both directions allowed both Monsky and the present authors to sharpen results, still using completely different methods --- to the full \cref{B}. It's a rare and wonderful thing to have immediate independent verification of one's theorems! 
Finally, we thank the anonymous referee for a careful reading of the article.

Shaunak Deo was partially supported in this work by the Fonds National de Recherche Luxembourg \mbox{INTER} /ANR/18/12589973 in the project ``Galois representations, automorphic forms and their $L$-functions~(GALF)"; 
by a Young Investigator Award from the Infosys Foundation, Bangalore; and by the DST FIST program 2021 (TPN-700661). 
Anna Medvedovsky was partially supported by an NSF postdoctoral fellowship (DMS\nobreakdash-1703834). Both authors thank the Max Planck Institute for Mathematics and University of Luxembourg for hosting them during part of this collaboration and supporting opportunities for the authors to visit each other. Anna Medvedovsky also thanks the A Room of One's Own initiative\footnote{\url{https://services.math.duke.edu/~pierce/AROOO_2020.shtml}} for support for focused research time, as well as her husband and her nanny for months of childcare support. Finally, the second author is thankful to the first author for his considerable patience: 
much of the technical work on this project was completed by 2016, but life and other projects intervened, as they do, delaying the manuscript.

%%%%%%%%%%%%%%%%%%%%%%%%%%%%%%%%%%%%%%%%%%%%
\section{Definitions and notation}\label{defnote}
%%%%%%%%%%%%%%%%%%%%%%%%%%%%%%%%%%%%%%%%%%%%

In this section we review definitions, with references where appropriate, and introduce notation used in the rest of the article. The expert reader should check in with our notation for the universal pseudodeformation ring in \cref{rdef}, glance at \cref{nilpsec,fricke}, and otherwise skip to \cref{gradingsec}. 

\subsection{Preliminaries}
\subsubsection{Finite fields} We work with a prime $p$, and will write $\FF$ for a finite extension of $\FF_p$. 

\subsubsection{Rings} All rings are assumed to be commutative with identity. For a ring $B$, let $B^\red$ be its maximal reduced quotient, the quotient of $B$ by its nilradical, the intersection of its prime ideals. 

If $B$ is a local pro-$p$ ring with maximal ideal $\mm$ and residue field $\FF$, its (reduced) tangent space is the $\FF$-vector space $\tan B := \Hom(\mm/(\mm^2 + p), \FF)$. 
If $B$ is noetherian, then its local topology agrees with its profinite topology \cite[Proposition~2.4]{deSmitLenstra}, so it's better known as a complete local noetherian ring with finite residue field of characteristic $p$. 
Moreover, in this case $\tan B$ is finite, and its dimension $d$ as an $\FF$-vector space is the same as the dimension of the dual (reduced) cotangent space $\mm/(\mm^2 +p)$, so that
$B$ is a quotient of $W(\FF) \lb x_1, \ldots, x_d\rb$, where $W(\FF) = \ZZ_{p^n}$ is the ring of Witt vectors of $\FF = \FF_{p^n}$. Finally, $\dim \tan B +1 \geq \dim B$, where $\dim B$ is the Krull dimension of~$B$. 

\subsection{Galois groups}\label{frattini} 
For any number field $K$, we write $G_K = \gal(\bar\QQ/K)$ for the absolute Galois group of~$K$. If $N$ is any integer, write $\QQ_{\{N\}}$ for the maximal extension of $\QQ$ unramified outside the primes dividing~$N$ and $\infty$. Let $G_{\QQ, N} = \gal(\QQ_{\{N\}} / \QQ)$. 

Let $\chi_p: G_\QQ \to \ZZ_p^\times$ denote the $p$-adic cyclotomic character, normalized so that $\chi_p(\frob_\ell) = \ell$ for prime $\ell \neq p$. Let $\omega_p: G_\QQ \to \ZZ_p^\times$ be the mod-$p$ cyclotomic character. Both factor through $G_{\QQ, Np}$ for any integer $N$.

If $G$ is any quotient of $G_\QQ$, we write $c$ for any complex conjugation in $G$. For any prime $\ell$, let $D_{\QQ_\ell} \subset G_\QQ$ be a decomposition group at $\ell$ and $I_{\QQ_\ell}$ be the inertia subgroup of $D_{\QQ_\ell}$.
Then write $D_\ell(G)$, or simply $D_\ell$ if~$G$ is clear, for the image of $D_{\QQ_{\ell}}$ in $G$. Note that both $D_{\QQ_\ell}$ and $D_\ell$ are only well defined up to conjugacy.  Once $D_\ell(G)$ is fixed, we write $I_\ell(G)$, or simply $I_\ell$, for the inertia subgroup of $D_\ell(G)$: it is the image of $I_{\QQ_\ell}$ in $G$, a closed normal subgroup of $D_\ell(G)$ with procyclic abelian quotient. 
We write $\frob_\ell$ for any Frobenius element of $D_\ell(G)$. If $I_\ell(G)$ is trivial, then $\frob_\ell$ is well defined up to conjugacy; otherwise, it is an arbitrary element in an $I_\ell(G)$-coset of $D_\ell(G)$, with the whole setup only defined up to conjugacy. 

If $G$ is any quotient of $G_\QQ$, or more generally any profinite group, write $G^{\protwo}$ for the maximal continuous pro-$2$ quotient of $G$. Also write $G^2$ for the closed subgroup of $G$ generated by the squares of elements of~$G$. This is a closed normal subgroup containing the commutators $[G, G]$ (see, for example, \cite[footnote before Lemma 5.3]{C}), and the quotient $G/G^2 = G^{\protwo}/(G^\protwo)^2$ is the \emph{$2$-Frattini quotient} of $G$, its maximal continuous elementary $2$-group quotient. The basic theorem of Frattini theory for $p = 2$ is that generators of $G/G^2$ lift to generators of $G^\protwo$.\label{frat}

\subsection{The space of modular forms of level $N$ and all weights}\label{modformspace}
Fix a level $N$. For an even weight $k \geq 0$, let $M_k(N, \ZZ)$ be the space of modular forms of level $\Gamma_0(N)$ and weight $k$ whose Fourier expansion at the cusp at infinity has rational integer coefficients. We view $M_k(N, \ZZ)$ as a subspace of $\ZZ\lb q \rb$ using the $q$-expansion principle. For any 
ring
$B$, 
let $M_k(N, B) := M_k(N, \ZZ) \otimes_\ZZ B \subset B\lb q \rb$, 
{the space of modular forms of level $\Gamma_0(N)$ and weight $k$ defined over $B$}. Note that for $B \subseteq \CC$, this definition coincides with the usual notion of modular forms with 
Fourier coefficients in $B$ 
viewed as $q$-expansions. Also define $M_{\leq k}(N, B) := \sum_{k' \leq k} M_{k'}(N, B)$ and $M(N, B) : = \sum_{k \geq 0} M_k(N, B) \subset B\lb q \rb$: this is the space of all modular forms of level $\Gamma_0(N)$ defined over~$B$. For a form $f \in M(N, B)$, write $a_n(f)$ for the $n^{\rm th}$ Fourier coefficient of~$f$, so that~$f = \sum_{n \geq 0} a_n(f) q^n$. 

From now on we assume that \fbox{$p$ does not divide $N$}. For $B = \FF_p$, we have defined $M(N, \FF_p) \subset \FF_p \lb q \rb$, the space of all mod-$p$ modular forms of level $\Gamma_0(N)$. This is the space of mod-$p$ modular forms as studied by Serre \cite{SerreBBK} and Swinnerton-Dyer \cite{SwDy} in level $1$.

%%%%%%%%%%%%%%%%%%%%%%%%%%%%%%%%%%%%%%%%%%%%
%%%%%%%%%%%%%%%%%%%%%%%%%%%%%%%%%%%%%%%%%%%%
\subsection{The weight filtration on mod-$p$ modular forms}\label{weightsec}
%%%%%%%%%%%%%%%%%%%%%%%%%%%%%%%%%%%%%%%%%%%%
%%%%%%%%%%%%%%%%%%%%%%%%%%%%%%%%%%%%%%%%%%%%
In characteristic zero, spaces of $q$-expansions of modular forms are graded by their weight: if $B$ is a subring of $\CC$, then $$\qquad\qquad M(N, B) = \bigoplus_{k \geq 0} M_k(N,B) \qquad\qquad \mbox{\cite[Lemma 2.1.1]{miyake}}.$$ 
In characteristic $p$, this is not the case, and the weight grading is replaced by the weight filtration. For $p \geq 5$, let $E_{p-1}$ be the level-$1$ weight-$(p-1)$ Eisenstein series 
with $q$-expansion in $1 + p \ZZ_{(p)}\lb q \rb$. Multiplication by~$E_{p-1}$ induces Hecke-equivariant embeddings $M_k(N, \FF_p) \into M_{k + p-1}(N, \FF_p)$ for each even $k \geq 0$. In fact this is the only kind of weight ambiguity: if we define, for $i \in 2\ZZ/(p-1)\ZZ$, 
\begin{equation}\label{weightgrade}M(N, \FF_p)^i := \bigcup_{k \equiv i \cmod{p-1}} 
M(N, \FF_p), \quad \mbox{ then } \quad M(N, \FF_p)=\bigoplus_{i \in 2\ZZ/(p-1)\ZZ} M(N, \FF_p)^i.
\end{equation}
See \cite[Theorem 2.2]{katzhigher} or \cite[Theorem 2(iv)]{SwDy} for $N = 1$.
To resolve the weight ambiguity we define for~$f \in M(N, \FF_p)^i$ its \emph{weight filtration} $$w(f) := \min\{k : f \in M_k(N, \FF_p)\}.$$ 
An important property of the weight filtration for our purposes is its compatibility with powers: for any~$n \geq 0$, we have $w(f^n) = n w(f)$ \cite[proof of Fact 1.7]{JochCong}.

For $p = 2, 3$, the story is a little more subtle. We may still define the na\"ive weight filtration in the same way, i.e. $w(f) := \min\{k : f \in M_k(N, \FF_p)\},$ but with this definition properties the compatibility with powers need not be satisfied. For example, the form $f_3 \in M(3, \FF_2)$ defined in \cref{lemma:polyalg} has $w(f_3) = 4$ but $w(f_3^2) = 6$. More dramatically, the form $f_5 \in M(5, \FF_2)$ \emph{loc.~cit.} has $w(f_5) = w(f_5^2) = 4$. The difficulty arises because the Hasse invariant, a geometric mod-$p$ modular form that controls the weight filtration, does not always lift to a characteristic-zero $\Gamma_0(N)$ form in weight $p-1$ for $p = 2, 3$. The Hasse invariant has $q$-expansion $1$, and simple zeros at the supersingular points of $X_0(N)_{\FF_p}$ and nowhere else; $E_{p-1}$ is a lift for $p \geq 5$. A form in $M_k(N, \FF_p)$ with a zero of minimal order $s$ at each supersingular point is divisible by $s$ copies of Hasse, and so comes from a form of lower filtration $k - s(p-1)$. {See \cite[\S 1.6--1.8]{Calegari} for more details.}

However, the Hasse invariant does always lift to a form of weight $p - 1$ for $\Gamma_1(N)$ for $N > 1$. 
For $p = 3$, see~\cite[\S 2.1]{katzpadic} for $N > 2$; for $N = 2$, take the  $\Gamma_0(N)$ weight-$2$ Eisenstein series $E_{2, N}$. For $p = 2$, 
see \href{https://mathoverflow.net/a/228596}{user Electric Penguin's answer} to \href{https://mathoverflow.net/questions/228497}{MathOverflow question 228497} or \cite[Appendix B]{meier}.  Therefore we may resolve all our difficulties with the weight filtration by replacing $w(f)$ with the weight filtration coming from $\Gamma_1(N)$. For any  $f \in M_k(N, \FF)$ set
$$w_1(f) := \min\big\{k' : \mbox{$f$ is the reduction of a $q$-expansion of a form in $M_{k'}(\Gamma_1(N), \bar\ZZ_p)$}\big\}.$$
{Then $w_1(f)$ is an integer satisfying $0 \leq w_1(f) \leq k$ and $w_1(f) \equiv k \pmod{p-1}$.}
Because $w_1(f)$ is defined geometrically
 --- $\frac{k - w_1(f)}{p-1}$ is the minimal order of a zero of $f$ at any supersingular point of $X_0(N)_{\FF_p}$ --- 
 this definition {resolves the problems} with the na\"ive filtration $w(f)$.  In particular, 
\begin{align}
\label{weightfact} &w_1(f^n) = n w(f) && \hspace{-2cm} 
\mbox{for all $n\geq 0$;}\\
\label{filt} & {w_1(f) \leq w(f) \leq w_1(f) + 3} &&  \hspace{-2cm}
\mbox{(and $w_1(f) = w(f)$ if $p \geq 5$).}
\end{align}
Here \eqref{filt} holds because for $p = 2, 3$ the weight-$4$ and level-$1$ Eisenstein form $E_4$, normalized so $a_0(E_4) = 1$, has mod-$p$ $q$-expansion $1$; by the $q$-expansion principle it is the fourth power (for $p = 2$) or the square (for~$p = 3$) of the Hasse invariant.

In our bad examples from \cref{lemma:polyalg}, $f_3$ and $f_5$ are reductions of semicuspidal\footnote{A \emph{semicuspidal} form vanishes at infinity but is nonzero at at least one other cusp, so its $q$-expansion at $\infty$ ``looks" cuspidal without it being a cuspform.} forms appearing with nontrivial quadratic nebentype in weight~$3$ and weight $2$, respectively, so that their squares appear in their ``true" $\Gamma_0(N)$ weight while they themselves do not. In other words, for $N = 3$ we have~$w_1(f_3) = 3$ and $w_1(f_3^2) = w(f_3^2) = 6$; for $N = 5$, $w_1(f_5)=2$ and $w_1(f_5^2) = w(f_5^2) = 4$. 

\subsection{Hecke operators on mod-$p$ modular forms}\label{heckeopssec}
The spaces $M_k(N, \FF)$ carry actions of Hecke operators inherited from the action on $M_k(N, \ZZ)$. More precisely, for prime $\ell \nmid Np$, the action of the Hecke operator~$T_\ell$ is defined on the $q$-expansion of a form $f \in M_k(N, \FF)$ by 
$$a_m(T_\ell f) = a_{m\ell}(f) + \ell^{k-1} a_{m/\ell}(f),$$
where $a_{m/\ell}(f)$ is understood to be $0$ if $\ell \nmid n$. We extend this to prime powers $\ell^r$ via the recurrence $T_{\ell^{r}} = T_{\ell} T_{\ell^{r-1}} - \ell^{k-1} T_{\ell^{r-1}}$ for $r \geq 2$, and to general $n$ with $\gcd(n, Np) = 1$ multiplicatively via $T_{nn'} = T_{n} T_{n'}$ provided $\gcd(n,n') = 1$. Since $\ell^{k-1}$ is well defined in $\FF$ for $k$ in a $(p-1)$-congruence class, the action of $T_n$ for $n \nmid Np$ on $M_k(N, \FF)$ extends to an action on all of $M(N, \FF)$ by \eqref{weightgrade}.

Also inherited from characteristic-zero is the action on $M_k(N, \FF)$ and $M(N, \FF)$ of the Atkin-Lehner operators~$U_n$ for any $n \mid N$, defined on $q$-expansions by $a_m(U_n f) = a_{mn}(f)$.

Finally, the reduction of the operator $T_p$ on $M(N, \ZZ)$ coincides modulo $p$   
with the Atkin-Lehner operator~$U_p$, at least for $k \geq 2$. From now on we use $U_p$ in place of $T_p$ (including for $k = 0$) for the at-$p$ Hecke action on~$M(N, \FF)$.
 The kernel $K(N, \FF)$ of $U_p$
\begin{equation}\label{knf} 
K(N, \FF) = \{f \in M(N, \FF): a_n(f) = 0 \mbox{ if $p \mid n$}\}
\end{equation}
 is a key subspace of $M(N, \FF_p)$ in the sequel. Write $K_k(N, \FF)$ for $M_k(N, \FF) \cap K(N, \FF)$. 
 
All of the Hecke operators --- $T_n$ for $(n, Np) = 1$, $U_n$ for $n \mid N$, and $U_p$ --- commute.

\subsection{The $\theta$ operator and Verschiebung} \label{verthetasec}
The operator $U_p$ has a right inverse, the \emph{Verschiebung} operator~$V_p$, sending $f = \sum_n a_n q^n \in M(N, \FF)$ to $V_p f = \sum_n a_{n} q^{np}$. If $\FF = \FF_p$, then Verschiebung coincides with the $p^{\rm th}$ power map, so that $V_p f = f^p$. More generally, $V_p f$ is a $\Gal(\FF/\FF_p)$-conjugate of $f^p$. It follows from~\eqref{weightfact} that 
\begin{equation}\label{ver}
w_1(V_p f) = p\, w_1(f). 
\end{equation}

An important operator on $M(N, \FF)$ is the derivation $\theta = q \frac{d}{dq}$, constructed for $p \geq 5$ in \cite{SwDy} and for $p \geq 2$ in~\cite{katz}. The operator $\theta$ takes $f = \sum_n {a_n q^n}$
 to
  $\theta(f) = \sum_n n a_n q^n$.
 One may verify the following facts (see~\mbox{\cite[\S II Theorem(2),(3); Corollary(6)]{katz}} for \eqref{kertheta} and \eqref{wtheta} below):
  \begin{align}
  \nonumber \theta^{p-1}&=1 - V_p U_p \mbox{ is a projector onto $K(N, \FF)$;}\\
 \label{imtheta} \im \theta &= K(N, \FF);\\
\label{kertheta} \ker \theta &= \im V_p = \im \mbox{($p^{\rm th}$ power map)};\\
 \label{wtheta} w_1(\theta f) &\leq w_1(f) + p + 1\mbox{, with equality if and only if $p \nmid w_1(f)$.}
  \end{align}    

\subsection{The Hecke algebra on mod-$p$ modular forms}\label{heckealgmodpsec}
We denote by $A_{\leq k}(N, \FF_p)$ the $\FF_p$-subalgebra of $\End_{\FF_p}\big(M_{\leq k}(N, \FF_p)\big)$ generated by the the action of the 
Hecke operators~$T_n$ with $\gcd(n, Np) = 1$ as in \cref{heckeopssec}. Since the actions are compatible with restriction maps, we set $$A(N, \FF_p):=\varprojlim_k A_{\leq k}(N, \FF_p):$$ this is the (\emph{shallow}) Hecke algebra acting on {the space of} modular forms of level $\Gamma_0(N)$ modulo~$p$. Equivalently, considering $M(N,\FF_p)$ with the discrete topology, and $\End_{\FF_p}(M(N, \FF_p))$ with the induced compact-open topology, the shallow Hecke algebra $A(N, \FF_p)$ is the closed subalgebra of $\End_{\FF_p}(M(N, \FF_p))$ generated by the 
$T_n$ with $\gcd(n, Np) = 1$
\cite[Proposition 2.4]{M}.

From this setup it follows that $A(N, \FF_p)$ is a pro-$p$ semilocal ring, so that it factors as a product of its localization at its maximal ideals. Moreover, the maximal ideals of $A(N, \FF_p)$ are in {one-to-one} correspondence with certain Galois representations, which we now describe. 

To every normalized Hecke eigenform $f$ in $M_k(N, \bar \QQ_p)$ a construction of {Eichler-Shimura and} Deligne attaches a continuous Galois representation $\rho_f: G_\QQ \to \gl_2(\bar \QQ_p)$ with the properties that $\rho_f$ is unramified at primes~$\ell \nmid Np$ with $\tr \rho_f(\frob_\ell) = a_\ell(f)$ and that $\det \rho_f = \chi_p^{k-1}$. Reducing any $G_\QQ$-invariant $\bar\ZZ_p$-lattice of~$\rho_f$ modulo the maximal ideal and semisimplifying gives us $\rhobar_f: G_\QQ \to \gl_2(\bar \FF_p)$. Note that $\rhobar_f$ is independent of the chosen lattice. Galois representations of the form $\rho_f$ or $\rhobar_f$ will be called \emph{$\Gamma_0(N)$-modular}.\footnote{We will not use this below, but Serre reciprocity (formerly Serre's conjecture) implies that a representation $\rhobar: G_{\QQ, Np} \to \gl_2(\bar \FF_p)$ is $\Gamma_0(N)$-modular if and only if its determinant is an odd power of $\omega_p$ and it has prime-to-$p$ Artin conductor dividing $N$. See \cref{artindef} for more on the Artin conductor for prime $N$.} 

The maximal ideals of $A(N, \FF_p)$, then, are in correspondence with $\gal(\barr\FF_p/\FF_p)$-orbits of Hecke eigenforms in {$M(N, \barr\FF_p)$}, which, by the Deligne-Serre lifting lemma correspond to $\gal(\barr\FF_p/\FF_p)$-orbits of $\Gamma_0(N)$-modular representations $\rhobar: G_\QQ \to \gl_2(\barr\FF_p)$. 
Passing to an extension $\FF/\FF_p$ that contains all {the Hecke eigenvalues} of all the mod-$p$ level-$\Gamma_0(N)$ Hecke eigenforms, we resolve the $\gal(\bar\FF_p/\FF_p)$-conjugacy. 
{Write $A(N, \FF)$ for $A(N,\FF_p) \otimes_{\FF_p} \FF$, and let}
$A(N, \FF)_\rhobar$ be the localization of $A(N, \FF)$ at the maximal ideal corresponding to $\rhobar$. Then $A(N, \FF)_\rhobar$ is a profinite local ring with residue field $\FF$, and we have a decomposition of the Hecke algebra
\begin{equation}\label{Asplit} 
A(N, \FF) = \prod_{\rhobar\ \mbox{\tiny $\Gamma_0(N)$-modular}} A(N, \FF)_\rhobar
\end{equation}
and a corresponding decomposition of the ring of modular forms into $\rhobar$-eigencomponents
\begin{equation}\label{Msplit} 
M(N, \FF) = \bigoplus_{\rhobar\ \mbox{\tiny $\Gamma_0(N)$-modular}} M(N, \FF)_\rhobar
\end{equation}
{refining the decomposition in \eqref{weightgrade},}
with each $A(N, \FF)_\rhobar$ acting faithfully on $M(N, \FF)_\rhobar$. (See {\cite[I.5.1]{joel:eigenbook}} for this kind of statement for the finite-dimensional quotients/subs and take limits.) 
Since Hecke operators~$T_n$ are multiplicative and the prime power Hecke operators $T_{\ell^r}$ satisfy an order-$2$ linear recurrence in $r$ with coefficients $T_\ell$ and $\det \rhobar(\Frob_\ell) \in \FF$, the local Hecke algebra $A(N, \FF)_\rhobar$ is topologically generated as an $\FF$-algebra by the operators $T_\ell$ for $\ell$ prime not dividing $Np$. 

Write $\mm(N, \FF)_\rhobar$ for the maximal ideal of $A(N, \FF)_\rhobar$. The modified Hecke operators $T'_\ell := T_\ell - \tr \rhobar(\frob_\ell)$ for prime $\ell$ not dividing $Np$ are all in $\mm(N, \FF)_\rhobar$; indeed, they topologically generate it.

\subsection{The partially full mod-$p$ Hecke algebra}\label{Apfsec} 
Finally, we define the \emph{partially full} Hecke algebra $A(N, \FF_p)^\pf$, 
the closed subalgebra of $\End_{\FF_p}(M(N, \FF_p))$ topologically generated by 
both 
all the Hecke operators $T_n$ with~$(n, Np) = 1$ and by all the Hecke operators $U_\ell$ for $\ell \mid N$.
As in \cref{heckealgmodpsec} we can alternatively define~$A(N, \FF_p)^\pf$ as an inverse limit of finite-level partial Hecke algebras.

One can check that $A(N, \FF_p)$ acts faithfully on $K(N, \FF_p)$, defined in \eqref{knf}, and  the pairing
\begin{equation} \label{pairing}\begin{split}
A(N, \FF_p)^\pf&\times K(N, \FF_p) \longrightarrow \FF_p\\
\langle T&, f \rangle \mapsto a_1(Tf)
\end{split}\end{equation}
is (continuously) perfect, inducing $A(N, \FF_p)^\pf$-module duality isomorphisms 
\begin{equation}\label{dual}
K(N, \FF_p) \toiso \Hom_{{\rm cont}}\big(A(N, \FF_p)^\pf, \FF_p\big) \quad\mbox{ and }\quad
A(N, \FF_p)^\pf \toiso \Hom\big(K(N, \FF_p), \FF_p\big).
\end{equation}
Here the continuity is with respect to the profinite (equivalently, the local) topology on $A(N, \FF_p)^\pf$ and the discrete topology on $K(N, \FF_p)$. For similar constructions, see \cite[Lemma 6.5]{J}, \cite[Th\'eor\`eme 5.1]{NS2}, \cite[Lemma 23(iii)]{BK}, \cite[Propositions 2.23 and 2.35]{M}, and \cite[section 5]{Mpaper}. 

In particular, for any closed ideal $J$ of $A(N, \FF_p)^\pf$, the duality in \eqref{dual} restricts to a duality between between the $J$-torsion in $K(N, \FF_p)$ and the quotient of $A(N, \FF_p)^\pf$ by $J$:  
\begin{equation}\label{dualJ}
 \mbox{$K(N, \FF_p)[J]$ and $A(N, \FF_p)^\pf/J$ are in (continuous) duality as $A(N, \FF_p)^\pf$-modules}
\end{equation}

If $\FF/\FF_p$ is large enough to contain all Hecke eigenvalues {of all mod-$p$ eigenforms of level $\Gamma_0(N)$}, and $\rhobar: G_{\QQ, Np} \to \gl_2(\FF)$ is a $\Gamma_0(N)$-modular representation, then we can define $A(N, \FF)^\pf_\rhobar$ as the quotient of $A(N, \FF)^\pf$ acting faithfully on $M(N, \FF)_\rhobar$. Like the shallow Hecke algebra $A(N, \FF)$, the  partially full $A(N, \FF)^\pf$ will also break up into a product of the $\rhobar$-components
$$A(N, \FF)^\pf = \prod_{\rhobar\ \mbox{\tiny $\Gamma_0(N)$-modular}} A(N, \FF)^\pf_\rhobar,$$
with $A(N, \FF)^\pf_\rhobar$ the quotient of $A(N, \FF)^\pf$ acting faithfully on $M(N, \FF)_\rhobar,$ 
and, as in \eqref{dualJ}, in duality with $K(N, \FF)_\rhobar : = K(N, \FF) \cap M(N, \FF)_\rhobar$.
However, $A(N, \FF)^\pf_\rhobar$ will not be local in general: its maximal ideals are in bijection with systems of eigenvalues $\{\alpha_\ell: \ell\ \mbox{prime dividing $N$}\}$ of the Hecke operators $U_{\ell}$ with 
$\ell$ prime dividing $N$
appearing in mod-$p$ modular forms in $M(N, \FF)_\rhobar$. 

The natural inclusion map $A(N, \FF_p) \into A(N, \FF_p)^\pf$ sending $T_\ell$ to $T_\ell$ for $\ell$ prime not dividing $Np$ is finite (see \cref{apffinite} below for the case of prime $N$ or \cite[proof~of~Theorem~3,~p.~23]{D} in general) and induces finite inclusions $A(N, \FF)_\rhobar \into A(N, \FF)_\rhobar^\pf$ for $\Gamma_0(N)$-modular $\FF$-valued $\rhobar$. 

\subsection{Pseudorepresentations}\label{psrep} We recall the definition of a {(dimension-2)} pseudorepresentation:
 for a (topological) group $G$ and a (topological) commutative ring $B$, a \emph{(continuous) pseudorepresentation of $G$ on $B$ (of dimension~$2$}) is a pair of (continuous) functions $(t,d):G \to B$ satisfying the following properties:
\begin{enumerate}[topsep = -3pt]
\item $t(1) = 2$;
\item $d: G \to B^\times$ is a group homomorphism;
\item $t$ is central: for all $g, h \in G$, we have $t(gh) =t(hg)$;
\item \emph{trace-determinant identity}: for all $g, h \in G$, 
\begin{equation} \label{tracedet} d(g)t(g^{-1}h)+t(gh)= t(g)t(h).
\end{equation}
\end{enumerate}
Pseudorepresentations, 
introduced in the form above 
by Chenevier in \cite{C} (where they are called \emph{determinants}), generalize the earlier notion introduced by Wiles~\cite{wiles} and Taylor~\cite{taylor} and further studied by Rouquier~\cite{rouq} and Nyssen \cite{nyssen} (now called \emph{pseudocharacters} following \cite{rouq}), to arbitrary characteristic. The idea is that $(t, d)$ generalizes  the data of pairs $(\tr \rho, \det \rho)$ for true representations $\rho: G \to \gl_2(B)$:  
that is,
if $\rho : G \to \gl_2(B)$ is a representation, then $(\tr \rho, \det \rho)$ is a pseudorepresentation of $G$ on $B$.
The converse is not true over for arbitrary rings $B$, but
if $B$ is an algebraically closed field then every pseudorepresentation on $B$ comes from a true representation \cite[Theorem 2.12]{C}.

If {$2 \in B^\times$} and $(t, d)$ is a 
pseudorepresentation {of $G$} on $B$, then $d$ is determined by $t$ (the trace-determinant identity for $h = g$ gives $d(g) = \frac{t(g)^2 - t(g^2)}{2}$); and $t$ is a {pseudocharacter} in the earlier notion of Taylor et al.

The \emph{kernel} of a 
pseudorepresentation $(t, d): G \to B$ of a group $G$ on a ring $B$ is 
\begin{equation}\label{kerdef} \ker (t, d) := \{g \in G: d(g) = 1,\ t(gh) = t(h)\ \mbox{for all $h \in G$}\} \subset G.\end{equation}
One can check that $\ker (t, d)$ is a (closed) normal subgroup of~$G$. Both $t$ and $d$ factor through the quotient $G/\ker (t, d)$, descending to a 
pseudorepresentation $(t, d): G/\ker (t, d) \to B$ with trivial kernel. If $\rho: G\to \gl_2(B)$ is a representation, then the kernel of $\rho$ is contained in the kernel of the associated pseudorepresentation: $\ker \rho \subseteq \ker (\tr \rho, \det \rho)$. {An easy calculation implies that the reverse containment also holds if $B$ is a field and $\rho$ is absolutely irreducible}. 

We will work with two-dimensional pseudorepresentations of profinite groups on profinite rings, and from now on will 
tacitly assume that all relevant maps are continuous.

For a topological $\ZZ_p$-algebra $B$, we say that a pseudorepresentation $(t, d)$ of $G_\QQ$ on $B$ is \emph{unramified} at some prime $\ell$ if $I_\ell(G_{\QQ})$ is contained in $\ker(t, d)$. Moreover, {we say that} $(t, d)$ is \emph{$\Gamma_0(N)$-modular} if $(t, d)$ is the base change of $(\tr \rho, \det \rho)$ for some $\Gamma_0(N)$-modular representation $\rho$.

\subsection{Pseudodeformations of mod-$p$ Galois representations}\label{rdef}
Fix a representation $$\rhobar: G_{\QQ, Np} \to \textstyle \gl_2(\FF),$$ which we assume to be semisimple and \emph{odd}: that is, $\tr \rhobar(c) = 0$.  Let $\mathcal C$ be the category of profinite local $\FF$-algebras with residue field $\FF$. 
We will call the objects of $\CCC$ ($\FF$-)\emph{coefficient algebras}. 

Let $\hat {\mathcal D}_{\rhobar}$ be the functor from $\mathcal C$ to the category of sets sending a coefficient algebra $(B, \mm)$ to the set of pseudorepresentations $(t, d)$ of $G_{\QQ, Np}$ on $B$ that reduce to the pseudorepresentation $(\tr \rhobar, \det \rhobar)$ modulo $\mm$ subject to the additional conditions that $t(c) = 0$ (oddness) and that $d = \det \rhobar$ (\emph{constant determinant}). A pseudorepresentation $(t, d)$ lifting $(\tr \rhobar, \det \rhobar)$ with constant determinant is obviously determined by $t$, so by abuse of notation we call {such a} $t$ a \emph{pseudodeformation} of $\rhobar$.

The functor $\hat{\mathcal D}_\rhobar$ is representable, represented by noetherian universal deformation ring $\big(\hat \RRR(N, \FF)_\rhobar, \hat \nn(N, \FF)_\rhobar\big)$ in $\mathcal C$ equipped with a universal pseudodeformation $\hat \tau^\univ := \hat \tau^\univ_{\rhobar, N}: G_{\QQ,Np} \to \hat \RRR(N, \FF)_\rhobar$ of $\rhobar$, in the sense that, for a coefficient algebra $B$, any pseudodeformation $t: G_{\QQ, Np} \to B$ of $\rhobar$ comes from a unique 
morphism $\hat \RRR(N, \FF)_\rhobar \to B$ 
{in $\mathcal{C}$ of coefficient algebras}:
\begin{equation} \label{psdef}
\begin{tikzpicture}[node distance=2cm]
\node(B)       {$B$};
\node(G)         [above left of = B]                   {$G_{\QQ, Np}$};
\node(R)       [above right  of = B] {$\hat \RRR(N, \FF)_\rhobar$};
\draw(G)[->] -- (R) node[midway, above] {$\hat \tau^\univ$};
\draw(G)[->] --(B) node[midway, below left] {$t$};
\draw(R)[dashed, ->]  -- (B) node[midway, below right] {$\exists !$};
\end{tikzpicture}
\end{equation}
See \cite[Propositions 3.3, 3.7]{C} for details. Note that $\hat \tau^\univ = \tr \rhobar + \hat \beta^\univ$, where $\hat \beta^\univ$ maps $G_{\QQ, Np}$ to $\hat \nn(N, \FF)_\rhobar$. Following \cite[section A.2]{BK}, we define, for $\ell$ prime not dividing $Np$, elements $\hat t_\ell := \hat \tau^\univ(\frob_\ell)$ in $\hat \RRR(N, \FF)_\rhobar$ and $\hat t'_\ell := \hat \beta^\univ(\frob_\ell)$ in $\hat \nn(N, \FF)_\rhobar$. 

By universality, the $\hat t_\ell'$  topologically generate $\hat \RRR(N, \FF)_\rhobar$ as an $\FF$-algebra; hence they also generate $\hat \nn(N, \FF)_\rhobar$ as an ideal {of $\hat \RRR(N, \FF)_\rhobar$}. If $B$ is the \emph{trace algebra} of $t$, that is, if $B$ is topologically generated as an $\FF$-algebra by $t(G_{\QQ,Np})$ (equivalently, by the $t(\frob_\ell)$ for $\ell \nmid Np$), then the unique map $\hat \RRR(N, \FF)_\rhobar \to B$, guaranteed by universality as in \eqref{psdef}, is surjective.   

Using $\hat \beta^\univ$ we also obtain an isomorphism of $\FF$-vector spaces (standard in deformation theory; see, for example, \cite[Lemma 2.6]{gouvea}): if $\FF[\eps]$ are the dual numbers, with $\eps^2 =0$, then
\begin{equation}\label{taneq}
\tan \hat \RRR(N, \FF)_\rhobar = \Hom\big (\hat \nn(N, \FF)_\rhobar /\hat \nn(N, \FF)_\rhobar^2,\ \FF\big) \cong \hat{\mathcal D}_\rhobar (\FF[\eps]).
\end{equation}
This isomorphism identifies a linear functional $h: \hat \nn(N, \FF)_\rhobar  \to \FF$ factoring through $\hat \nn(N, \FF)_\rhobar^2$ with the pseudodeformation $g \mapsto \tr \rhobar(g) + \eps h \big(\hat \beta^\univ (g)\big)$ of $\rhobar$. 

If $\rhobar: G_{\QQ, Np} \to \gl_2(\FF)$ is a semisimple odd representation that factors through $G_{\QQ, Mp}$ for some divisor $M$ of $N$, {then the surjective map $G_{\QQ,Np} \onto G_{\QQ,Mp}$ induces a natural surjection 
\begin{equation}\label{surjR}
\hat \psi_{N, M}: \hat \RRR(N, \FF)_\rhobar \onto \hat \RRR(M, \FF)_\rhobar
\end{equation} and its kernel is} the closed ideal $\hat J_{N, M}$ generated by elements of the form 
$\hat \tau^\univ(gi)- \hat \tau^\univ(g)$
 for $g \in G_{\QQ, Np}$ and $i \in I_\ell(G_{\QQ, Np})$ with $\ell$ running over the primes that divide $N$ but not $M$.

%%%%%%%%%%%%%%%%%%%%%%%%%%%%%%%%%%%%%%%%%%%%
%%%%%%%%%%%%%%%%%%%%%%%%%%%%%%%%%%%%%%%%%%%%
\subsection{The pseudodeformation of $\rhobar$ carried by $A(N, \FF)_\rhobar$}\label{pdonA}
%%%%%%%%%%%%%%%%%%%%%%%%%%%%%%%%%%%%%%%%%%%%
%%%%%%%%%%%%%%%%%%%%%%%%%%%%%%%%%%%%%%%%%%%%
By gluing together all the $\Gamma_0(N)$-modular pseudorepresentations of $G_{\QQ, Np}$ that {Eichler-Shimura and} Deligne's construction attaches to characteristic-zero modular eigenforms of level $N$ and reducing modulo $p$ (see, for example, \cite[Step~$1$ of the proof of Theorem $1$]{B} for a detailed construction for $p = 2$, $N = 1$)
one obtains an odd pseudorepresentation 
\begin{equation}\label{tauheckeall}
 \tau^\hecke_{p, N} : G_{\QQ, Np} \to  A(N, \FF_p)
\end{equation}
satisfying $\tau^\hecke_{p, N}(\frob_\ell) = T_\ell$ for any prime $\ell \nmid pN$. If we fix a $\Gamma_0(N)$-modular $\rhobar: G_{\QQ, Np} \to \gl_2(\FF)$, then by extending scalars in \eqref{tauheckeall} and composing with the map $A(N, \FF) \onto A(N, \FF)_\rhobar$ from the decomposition in \eqref{Asplit} (or alternatively, by gluing together the $\Gamma_0(N)$-modular pseudodeformations of $\rhobar$ and reducing mod~$p$), one obtains a constant-determinant odd pseudodeformation of~$\rhobar$
\begin{equation}\label{tauhecke}
\tau^\hecke_{\rhobar} {:= \tau^\hecke_{\rhobar, N}} : G_{\QQ, Np} \to  A(N, \FF)_\rhobar
\end{equation}
again with  $\tau^\hecke_{\rhobar, N}(\frob_\ell) = T_\ell$ for  $\ell \nmid pN$. 
\begin{equation}\label{RtoA}
\hat \phi: \hat \RRR(N, \FF)_\rhobar \onto A(N, \FF)_\rhobar
\end{equation}
sending {$\hat t_\ell$} to $T_\ell$ for primes $\ell \nmid Np$. Since $A(N, \FF)_\rhobar$ is topologically generated by the $T_\ell$, the morphism $\hat \phi$ is surjective, which surjectivity 
tells us that $A(N, \FF)_\rhobar$ is noetherian, so that its profinite topology coincides with its local topology \cite[Proposition 2.4]{deSmitLenstra}.

%%%%%%%%%%%%%%%%%%%%%%%%%%%%%%%%%%%%%%%%%%%%%
%%%%%%%%%%%%%%%%%%%%%%%%%%%%%%%%%%%%%%%%%%%%
\subsection{The nilpotence method for lower bounds on $\dim A(N, \FF)_\rhobar$}\label{nilpsec}
%%%%%%%%%%%%%%%%%%%%%%%%%%%%%%%%%%%%%%%%%%%%
%%%%%%%%%%%%%%%%%%%%%%%%%%%%%%%%%%%%%%%%%%%%
We summarize the method described in \cite{M, Mpaper} for obtaining a lower bound on the Krull dimension of a local piece of the mod-$p$ Hecke algebra acting on a subspace of a polynomial algebra of forms. In \cite{M} this method is applied for $A = A(1, \FF_p)_\rhobar$ with $p = 2, 3, 5, 7, 13$; more generally the method may be applied to $A = A(N, \FF_p)_\rhobar^\pf$ so long as the genus of~$X_0(Np)$ is zero.

\begin{mythm}[{Nilpotence method \cite{M, Mpaper}}]\label{nilpthm}
Suppose the following conditions are satisfied. 
\begin{enumerate}[topsep = -3pt, itemsep = 1pt] 
\item \label{power} $M(N, \FF) = \FF[f]$ for some form $f \in M$. 
\item \label{two} $A$ is a continuous local quotient of $A(N, \FF)^\pf$ acting faithfully on a subspace $K \subseteq K(N, \FF)$.
\item\label{max} The maximal ideal $\mm$ of $A$ is generated by Hecke operators $S_1, \ldots, S_d$ so that, for each $i$, 
\begin{enumerate}
\item $S_i$ is in every maximal ideal of $A(N, \FF)^\pf$; and 
\item \label{three}  the sequence $\{S_i(f^n)\}_n$ satisfies an $M$-linear recurrence of some order $d_i$ whose characteristic polynomial $X^{d_i} + a_{i, 1} X^{d_i - 1} + \cdots + a_{i, d_i - 1} X + a_{i, d_i} \in M[X]$ satisfies both $\deg_f a_{i, j} \leq j$ for all~$j$ and $\deg_f a_{i, d_i} = d_i$.
\end{enumerate}
\item \label{four} There exists a sequence of linearly independent forms $\{g_n\}_n$ in $K$ with $\deg_f g_n$ depending at most linearly on $n$. (In other words, $\deg_f g_n$ is $O(n)$.)
\end{enumerate}
Then $\dim A \geq 2$. 
\end{mythm}
Condition \eqref{power} is crucial to the method, as it relies on the Nilpotence Growth Theorem \cite{Mpaper}.
If $K = K(N, \FF)$, then generators $S_1, \ldots, S_d$ of $\mm$ satisfying the conditions in \eqref{max} are  known to exist: \mbox{see \cite[4.3.3 and 6.3]{M}} for the case $N = 1$; the general case is similar. 

The idea of the proof of \cref{nilpthm} is as follows: 
by the main theorem of \cite{Mpaper} and the condition on the sequence $\{g_n\}$, the function $h(n)$ with the property that $\mm^{h(n)}$ annihilates $\{g_0, \ldots, g_n\}$ grows slower than linearly in $n$. By duality between $A$ and $K$ coming from the duality between $A(N, \FF)^\pf$ and $K(N, \FF)$ \eqref{dual}, the Hilbert-Samuel function 
$k \mapsto \dim_{\FF}{A}/{\mm^k}$
grows faster than linearly in~$k$. Therefore the Hilbert-Samuel degree of $A$ is strictly greater than $1$, hence at least $2$. But this degree is equal to the Krull dimension of~\mbox{$A$~\cite[Theorem 11.14]{AM}}.

In practice, in condition \eqref{four} one may replace the $f$-degree of a form $g$ in $M(N, \FF)$ with its weight filtration, as described in \cref{weightsec}.

\subsection{Oldforms and newforms mod $p$}\label{fricke}
In this subsection, we assume $N$ is prime. We briefly summarize the perspective of
\cite{deomed}. 
\subsubsection{The Fricke automorphism mod $p$} The characteristic-zero Fricke involution $w_N$ on $M_k(N, \CC)$ sending~$f(z)$ to $$w_N f = \left. f \right|_k \begin{psmallmatrix} 0 & -1 \\ N & 0 \end{psmallmatrix} = N^{k/2} (Nz)^{-k} f(\textstyle \frac{-1}{Nz})$$ is defined over $\ZZ[\frac{1}{N}]$, and therefore descends to $M_k(N, \FF_p)$. However, $w_N$ is not in general well defined as an algebra involution on all of $M(N, \FF_p)$: see \cite[section 3]{deomed} for a details. Since this is inconvenient, we replace~$w_N$ on $M_k(N, \FF_p)$ by $W_N := N^{k/2} w_N $. This renormalized Fricke operator is an algebra automorphism of order dividing $p-1$ for~$p$ odd \cite[Proposition 3.13(3)]{deomed}; for $p = 2$ the operator $W_N$ coincides with $w_N$ and is hence an algebra involution. In all cases $W_N$ commutes with all $T_n$ for $n$ prime to~$Np$. 

\subsubsection{Old and new forms mod $p$} \label{oldformsnewforms}
{We can now define the mod-$p$ ``oldforms" in level $N$ in the following way: let $M(N, \QQ)^\old := M(1, \QQ) + W_N M(1, \QQ) \subset M(N, \QQ),$ as usual. 
Set $$M(N,\ZZ)^\old := M(N,\QQ)^\old \cap \ZZ\lb q\rb \subset M(N,\ZZ);$$ let $M(N,\FF_p)^\old$ be the subspace of $M(N,\FF_p)$ obtained by reducing $M(N,\ZZ)^\old$ modulo $p$, and finally set $M(N, \FF)^\old : = M(N, \FF_p) \otimes_{\FF_p} \FF$ (see \cite[Section 5]{deomed} for more details).}

We also define the ``newforms": let $M(N, \FF)^\new := \ker (U_N^2 - N^{-2} {\mathcal S}_N)$, where ${\mathcal S}_N$ is a weight-separating operator that scales $M_k(N, \FF)$ by $N^k$ \cite[Section 6]{deomed}. Observe that replacing $\FF$ by $\CC$ recovers the usual classical notion of newforms.
Note that $M(N, \FF)^\old$ and $ M(N, \FF)^\new$ need not be disjoint \cite[Corollary 7.1]{deomed}.
However, we record the following fact.

\begin{mylemma}\label{s2killsnew}
The operator $(U_N^2 - N^{-2} {\mathcal S}_N)$ maps $M(N, \FF)$ to $M(N, \FF)^\old$.
\end{mylemma}
\begin{proof}
If $f \in M_k(N,\ZZ)$, then $f=g+h$ with $g \in M_k(N,\QQ)^\new$ and $h \in M_k(N,\QQ)^\old$.
Since the operator $(U_N^2-N^{-2}{\mathcal S}_N$) kills newforms and preserves integrality, we see that $(U_N^2-N^{-2}{\mathcal S}_N)f \in M_k(N,\ZZ)^\old$.
\end{proof}

\subsubsection{New and old Hecke algebra quotients}\label{newoldhecke}

We also consider {$A(N, \FF)^{\old}$} (respectively, {$A(N, \FF)^{\new}$}), the largest quotient of {$A(N, \FF)$} acting faithfully on {$M(N, \FF)^{\old}$} (respectively, {$M(N, \FF)^{\new}$}). For the shallow Hecke algebras, we have {$A(N, \FF)^\old \cong A(1, \FF)$}.
 
Since both $W_N$ and $U_N$, the operators whose actions define oldforms and newforms, commute with all the Hecke operators away from $Np$, both the spaces {$M(N, \FF)^\old$} and {$M(N, \FF)^\new$} and the Hecke algebras {$A(N, \FF)^\old$} and {$A(N, \FF)^\new$} split into local $\rhobar$-components. 

In particular, if $\rhobar: G_{\QQ, Np} \to \gl_2(\FF)$ is $\Gamma_0(N)$-modular, then $A(N, \FF)_\rhobar$ has two quotients
\begin{equation}\label{Aoldnew}\begin{split}
\begin{tikzpicture}[node distance=2cm]
\node(A) {$A(N, \FF)_\rhobar$};
\node(Anew) [below right = 10pt and 1.5cm of  A]{$A(N, \FF)_\rhobar^\new$.};
\node(Aold) [below left = 10pt and 1.5cm of A]{$A(N, \FF)_\rhobar^\old$};
\node(A1) [left = -4pt of Aold]{$A(1, \FF)_\rhobar \cong$};
\draw(A)[->>]--(Aold) node[near start, above left]{$\scriptstyle \pi^\old$};
\draw(A)[->>]--(Anew) node[near start, above right]{$\scriptstyle \pi^\new$};
\end{tikzpicture}\end{split}
\end{equation}
If $\rhobar$ is unramified at $N$, then $M(N, \FF)_\rhobar^\old$ is nonzero, so that the left quotient $\pi^\old: A(N, \FF)_\rhobar \onto A(N, \FF)^\old_\rhobar$ is nontrivial. If $\rhobar$ further satisfies the \emph{level-raising condition at $N$} --- namely, $\tr \rhobar(\frob_N) = \pm (N + 1) N^\frac{k-2}{2}$, where 
$\det \rhobar = \omega_p^{k-1}$
 --- then $M(N, \FF)^\new_\rhobar$ is nonzero, so that the right quotient $\pi^\new: A(N, \FF)_\rhobar \onto A(N, \FF)^\new_\rhobar$ is nontrivial.  See \cite[section 7]{deomed} for more details. 

\subsubsection{New and old quotients of the partially full Hecke algebra}\label{pfoldnew}
Finally we note that $M(N, \FF)^\old$ and $M(N, \FF)^\new$, as well as local pieces $M(N, \FF)_\rhobar^\old$ and $M(N, \FF)_\rhobar^\new$, are all stable by $U_N$. Indeed, suppose~$\rhobar$ is $\Gamma_0(1)$-modular and $f$ is in $M(1, \FF)_\rhobar$. Since we can ignore the action of $T_N$, or indeed of any finite set of Hecke operators, in defining the $\rhobar$-generalized eigenspace, and since $W_N$ commutes with all  Hecke operators prime to $Np$, we find that $W_N f \in M(N, \FF)_\rhobar$. Moreover, the $\FF$-span of $\{f, W_N f\}$ is a $U_N$-stable subspace of $M(N, \FF)_\rhobar^\old$. 
And if 
$\rhobar$ is $\Gamma_0(N)$-modular, then $f \in M(N, \FF)_{\rhobar}$ is in $M(N, \FF)_{\rhobar}^\new$ if and only if $f$ is in a 
$U_N^2$-eigenspace. Since~$U_N$ and $U_N^2$ commute, this property is preserved under the {action} 
of~$U_N$, 
so that $M(N, \FF)^\new$, and hence $M(N, \FF)^\new_\rhobar$, is $U_N$-stable. 
Therefore we can define $A(N, \FF)_\rhobar^{\pf, \old}$ and $A(N, \FF)_\rhobar^{\pf, \new}$, the faithful quotients of $A(N, \FF)_\rhobar^\pf$ acting on $M(N, \FF)_\rhobar^\old$ and $M(N, \FF)_\rhobar^\new$, respectively. See \cref{pfanal}.

%%%%%%%%%%%%%%%%%%%%%%%%%%%%%%%%%%%%%%%%%%%%
%%%%%%%%%%%%%%%%%%%%%%%%%%%%%%%%%%%%%%%%%%%%
%%%%%%%%%%%%%%%%%%%%%%%%%%%%%%%%%%%%%%%%%%%%
\section{The grading on the mod-$p$ Hecke algebra} \label{gradingsec}
%%%%%%%%%%%%%%%%%%%%%%%%%%%%%%%%%%%%%%%%%%%%
%%%%%%%%%%%%%%%%%%%%%%%%%%%%%%%%%%%%%%%%%%%%
%%%%%%%%%%%%%%%%%%%%%%%%%%%%%%%%%%%%%%%%%%%%

In this short section we exhibit natural compatible gradings on $K(N, \FF_p)$, $A(N, \FF_p)$, and $\tau_{p, N}^\hecke$ by the $2$-Frattini quotient of $\ZZ_p^\times$, which restrict to gradings on $A(N, \FF)_\rhobar$ if $p = 2$ or if $p$ is odd and $\rhobar \simeq \rhobar \otimes \omega_p^{(p-1)/{2}}.$ This result (\cref{gradingthm}) generalizes the $(\ZZ/8\ZZ)^\times$-grading in the $(p, N) = (2, 1)$ setting described by Nicolas and Serre~\cite{NS1} and the $(\ZZ/3\ZZ)^\times$-grading in the $(p, N) = (3, 1)$ setting in the forthcoming treatment of mod-$3$ modular forms by the second-named author \cite{medved3}. 
It owes an additional debt of inspiration to \jb's universal grading result for mod-$2$ constant-determinant
pseudodeformations of $\rhobar = \one$ (\cref{bellgrade}).  

Suppose that $B$ is a profinite ring and $Q$ a finite abelian group, written multiplicatively. Recall that $B$ is \emph{$Q$-graded} if $B$ splits as a direct sum $B = \bigoplus_{i \in Q} B^i$ of closed additive subgroups $B^i$ with $B^1 \subseteq B$ a subring and $B^{i} B^j \subseteq B^{ij}$ for every $i,j \in Q$. If $B$ is $Q$-graded, then a $B$-module $M$ is \emph{$Q$-graded} if $M = \bigoplus_{i \in Q} M^i$ with $B^i M^j \subseteq M^{ij}$. 

Now further suppose that $(t, d): G \to B$ is a (continuous) pseudorepresentation of a profinite group $G$ on~$B$, and that $Q$ is equipped with a continuous quotient map $\pi: G \onto Q$. For $i \in Q$ let $G^i := \pi^{-1}(i)$ be the {corresponding coset of $G$}. We say that $(t, d)$ is \emph{$Q$-graded} if $B = \bigoplus_{i \in Q} B^i$ is a $Q$-graded ring and for every $i \in Q$ we have $t(G^i) \subset B^i$ and $d(G^i) \subset B^{i^2}$. Note that all the pseudorepresentation relations from \cref{psrep} are homogeneous with respect to such a grading. 

We begin with two lemmas; the grading theorem is \cref{gradingthm}. These establish that $q$-expansions of mod-$p$ forms can be separated into coefficients whose indices are picked out by a character of conductor $p$. Note that \cref{charsplit} is stated in a more general form than strictly required for \cref{gradingthm} (for which $M = 1$ in \cref{charsplit} suffices); this generality necessitates \cref{upvp}. Both are well known. 

\begin{mylemma}\label{charsplit}
If $f = \sum_n a_n q^n \in M(N, \FF)$, and $\chi$ is a quadratic Dirichlet character of modulus $p^r M$, where~$r \geq 0$ is arbitrary and $M^2 \mid N$, then the following are also forms in $M(N, \FF)$: 
$$(1)\ f_\chi := \sum_n \chi(n) a_n q^n, \quad 
(2)\ f_{\chi, +} := \sum\limits_{n: \chi(n) = 1} a_n q^n, \quad
(3)\ f_{\chi, -} :=  \sum\limits_{n: \chi(n) = -1} a_n q^n.
$$
\end{mylemma}
\begin{proof}
We follow Jeremy Rouse's answer to \href{{https://mathoverflow.net/questions/202449}}{MathOverflow question \#202449}. Let $\tilde f \in M(N, \OO)$ be a lift of $f$ for some ring of integers $\OO$ with residue field $\FF$ in {a finite extension $L$ of $\QQ_p$}; we first show that the analogously defined characteristic-zero {forms} $\tilde f_\chi$, $\tilde f_{\chi, +}$, and $\tilde f_{\chi, -}$ are in {$M(Np^{2r}, \OO)$}. Indeed, since $\chi$ has order~$2$ and the square of its modulus divides $Np^{2r}$, the statement about $\tilde f_\chi$ follows from \cite[Theorem 7.4]{Iwaniec}. Let $S$ be the squarefree product of primes dividing $p^r M$; write $U_S$ as usual for the operator $\sum_n a_n q^n \mapsto \sum a_{nS} q^n$ and~$V_S$ for the operator $\sum_n a_n q^n \to \sum a_n q^{nS}$. By \cref{upvp} below, $\tilde f_{\chi, 0} : = V_S U_S \tilde f$ is in {$M(N p^{2r} , \OO)$}, so that $\tilde f_{\chi, +} = \frac{1}{2}(\tilde f  - \tilde f_{\chi, 0} + \tilde f_\chi)$ and $\tilde f_{\chi, -} = \frac{1}{2} (\tilde f - \tilde f_{\chi, 0} - \tilde f_\chi)$ are both in {$M(N p^{2r}, \OO)$}. 
Finally, $\barr\FF_p$-reductions of forms of level $N$ and of level $Np^{2r}$ coincide:  indeed, Hatada \cite[Theorem 1]{hatada}, generalizing earlier work of Serre for $p \geq 3$~\cite{serre:padic}, shows that every modular form of level $Np^{{2r}}$ is a $p$-adic limit of forms of level $N$.  
\end{proof}
\begin{mylemma}\label{upvp}
If $f$ is in $M(N, \ZZ)$ and $\ell$ is a prime, then 
$V_\ell U_\ell f \in M\big( \lcm(N, \ell^2), \ZZ\big)$. 
\end{mylemma}

\begin{proof}
Integrality of coefficients is clearly preserved, so it suffices to establish the level. 
The level of $U_\ell f$ is a priori $N\ell$ if $\ell\nmid N$ and $N$ if $\ell \mid N$. But in fact if $\ell^2 \mid N$, then $U_\ell f$ is of level $N / \ell$. Indeed, if $\ell^2 \mid N$, then any form in $M(N, \QQ)$ is a linear combination of $\ell$-new forms (which $U_\ell$ kills), forms from level $N/\ell$ (which~$U_\ell$ keeps at level $N/\ell$), and forms in the image of $V_\ell$ coming from level $N/\ell$ (which~$U_\ell$ sends back to level $N/\ell$). Finally, $V_\ell$ raises the level from $N \ell$ or $N/\ell$ by a factor of $\ell$.
\end{proof}

We now return to our setting of a fixed prime $p$ and a level $N$ prime to $p$. Let $\frat_p$ be the $2$-Frattini quotient of $G_{\QQ, p}$: that is, $\Pi_p = \gal(L_p/\QQ)$, where $L_2 = \QQ(\mu_8)$ and for $p$ odd $L_p$ is the quadratic subfield of~$\QQ(\mu_p)$. We also identify $\frat_p$ with the $2$-Frattini quotient of $\Gal(\QQ(\mu_{p^\infty})/\QQ) \simeq \ZZ_p^\times$: explicitly, $\frat_2 = (\ZZ/8\ZZ)^\times$ and for $p$ odd $\frat_p = \FF_p^\times/(\FF_p^\times)^2$. In this way, we may think of any $n \in \ZZ$ prime to $p$ as having a value in $\frat_p$.
Restriction to $L_p$ gives a quotient map $G_{\QQ, Np}  \onto \frat_p$ with the property that $\frob_\ell$ maps to 
$\ell$ 
for $\ell \nmid Np$ prime. For $i \in \frat_p$ and $\FF/\FF_p$, let $$K(N, \FF)^i := \{f \in K(N, \FF): a_n(f) \neq 0 \implies 
n= i \mbox{ in } \Pi_p\} \subset K(N, \FF).$$ 
The next lemma shows that Hecke operators act compatibly with this $\frat_p$-indexing:
 
\begin{mylemma}[cf. {\cite[(6) and ff.]{NS1}}] \label{heckegrade} For all $m$ prime to $p$ and $i \in \Pi_p$, the Hecke operator at $m$ maps 
{$K(N, \FF)^i$} into {$K(N, \FF)^{m i}$}. This includes $T_m$ if $m \nmid Np$ and $U_m$ if $m \mid N$. 
\end{mylemma}

\begin{proof} Since the Hecke operators are multiplicative at relatively prime indices, it suffices to show this for prime-power-index Hecke operators. First let $\ell \nmid Np$ be a prime. 
For a form $f$ coming from weight $k$
the Fourier coefficients of $T_\ell f$ are well known: $$a_n(T_\ell f) = a_{\ell n}(f) + \ell^{k-1} a_{n/\ell}(f),$$ where $a_{n/\ell}(f) = 0$ if $\ell \nmid n$. The claim for $T_\ell$ follows {since $\Pi_p$ is an elementary $2$-group, so that 
$\ell \equiv \ell^{-1}$} in~$\Pi_p$.
For $r \geq 2$ the claim for $T_{\ell^r}$ follows by induction, since $T_{\ell^r} = T_\ell T_{\ell^{r -1}} - \ell^{k-1} T_{\ell^{r-2}}$. On the other hand, if~$\ell$ is a prime with $\ell^r \mid N$, then $a_n(U_{\ell^r} f) = a_{n\ell^r} (f)$ so that the claim for $U_{\ell^r}$ follows. 
\end{proof}  

Let $\rhobar$ be a $\Gamma_0(N)$-modular representation defined over $\FF$, and set $K(N, \FF)_\rhobar^i := K(N, \FF)^i \cap K(N, \FF)_\rhobar.$

\begin{mythm}\label{gradingthm}
\label{cm} Suppose that $p = 2$ or $p$ is odd and $\rhobar \cong \rhobar \otimes \omega^{\frac{p-1}{2}}$. 
\begin{enumerate}[itemsep = 3pt]
\item\label{formgrade} The space of forms has a natural $\frat_p$-grading: $K(N, \FF)_\rhobar= \bigoplus_{i \in \frat_p} K(N, \FF)_\rhobar^i$.
\item\label{alggrade} The Hecke algebra $A(N, \FF)_\rhobar$ has a natural $\frat_p$-grading $A(N, \FF)_\rhobar = \bigoplus_{i \in \frat_p} A(N, \FF)_\rhobar^i$, where for $m \nmid Np$ we have $T_m \in A(N, \FF)_\rhobar^m$. The decomposition from \eqref{formgrade} endows $K(N, \FF)_\rhobar$ with the structure of a $\frat_p$-graded $A(N, \FF)_\rhobar$-module.
\item\label{psrepgrade} The decomposition from \eqref{alggrade} and the quotient map $G_{\QQ, Np} \onto \frat_p$ gives a $\frat_p$-grading to the pseudorepresentation $\tau^\hecke_{\rhobar}$. 
\end{enumerate}
\label{sum} Moreover, \eqref{formgrade}--\eqref{psrepgrade} hold for  $A(N, \FF)_\rhobar \times A(N, \FF)_{\rhobar \otimes \omega^{(p-1)/{2}}}$ acting on $K(N, \FF)_\rhobar \oplus K(N, \FF)_{\rhobar \otimes \omega^{({p-1})/{2}}}$ and 
\label{whole}for $A(N, \FF_p)$ acting on $K(N, \FF_p)$.  
\end{mythm} 

\begin{proof}
\begin{enumerate}[itemsep = 3pt, topsep = -5pt]
\item Fix a quadratic Dirichlet character $\chi$ on $\frat_p$, so that {$\chi = \omega_p^{\frac{p-1}{2}}$} if $p$ is odd.
For $\eps = \pm 1$, let $K(N, \FF)^{\chi, \eps}_\rhobar = \{g \in K(N, \FF)_\rhobar : a_n(g) \neq 0 \mbox{ only if } \chi(n) = \eps \}$. For $f \in K(N, \FF)_\rhobar$ we have, by the assumption that $\rhobar \cong \rhobar \otimes \chi$ and in the notation of \cref{charsplit}, $f_{\chi, \eps}$ in $K(N, \FF)_\rhobar^{\chi, \eps}$, so that $K(N, \FF)_{\rhobar} = K(N, \FF)_{\rhobar}^{\chi, +} \oplus K(N, \FF)_{\rhobar}^{\chi, -}$, completing the proof for $p$ odd. For $p = 2$, decompose each {$K(N, \FF)_{\rhobar}^{\chi, \eps}$} further into a direct sum of two pieces corresponding to the values of a second quadratic Dirichlet character in {$\Pi_p$}.
\item Follows formally from \eqref{formgrade} and \cref{heckegrade} as follows. Let $B = A(N, \FF)_\rhobar$. For $i \in \frat_p$, let 
$$B^i := \big\{ T \in A(N, \FF)_\rhobar \mid \mbox{ for all $j \in \frat_p$, } T K(N, \FF)_\rhobar^j \subseteq K(N, \FF)^{ij}_\rhobar\big\}.$$
By considering finite weight and taking limits we see that each $B^i$ is a closed $\FF$-submodule of $B$.
Moreover, $1$ is in $B^1$ and $B^i B^j \subseteq B^{ij}$; by considering $q$-expansions and using \eqref{formgrade}, it's clear that the sum of the $B^i$ inside $B$ is direct, so that $B' := \bigoplus_{i \in \frat_p} B^i$ is a $\Pi_p$-graded $\FF$-algebra. Finally by \cref{heckegrade} for each $m \nmid Np$ we have $T_m \in {B^{m}}$. These operators generate $B$, so that $B = B'$. 
\item For every $i \in \frat_p$, the coset $G_{\QQ,Np}^i$ is closed in $G_{\QQ,Np}$.
By the Chebotarev density theorem, 
the $\frob_\ell$-conjugacy classes for those primes $\ell \nmid Np$ with $\ell \equiv i$ in $\frat_p$
are dense in  $G_{\QQ, Np}^i$. 
Since $\tau^\hecke_{\rhobar}(\frob_\ell) = T_\ell$ is in $A(N, \FF)_\rhobar^{{\ell}}$ by \eqref{alggrade}, and $\tau^\hecke_{\rhobar}$ is continuous with each $A(N, \FF)_\rhobar^i$ closed, the claim follows. 
\end{enumerate}
The proofs of the second and third statement are analogous. \qedhere
\end{proof}

\begin{mycor}\label{apfanewgrading} Under the conditions of \cref{gradingthm} we additionally have a $\frat_p$-grading on $A(N, \FF)_\rhobar^\pf$, and, if $N$ is prime, on $A(N, \FF)_\rhobar^\new$. These gradings are compatible with their structures as $A(N, \FF)_\rhobar$-algebras and their action on $M(N,\FF)_{\rhobar}$.
\end{mycor} 

\begin{proof} For $A(N, \FF)_\rhobar^\pf$, mimic the argument in \cref{gradingthm}\eqref{alggrade}, noting that \cref{heckegrade} covers all Hecke operators topologically generating $A(N, \FF)_\rhobar^\pf$. For $A(N, \FF)_\rhobar^\new$, first refine \cref{gradingthm}\eqref{formgrade} to give a grading on $K(N, \FF)_\rhobar^\new := K(N, \FF)_\rhobar \cap M(N, \FF)^\new$, on which $A(N, \FF)_\rhobar^\new$ acts faithfully. Namely, if $f \in K(N, \FF)_\rhobar^\new$ decomposes as $f = \sum_{i \in \frat_p} f_i$ with $f_i \in K(N, \FF)_\rhobar^i$, then in fact each $f_i$ is in $K(N, \FF)_\rhobar^\new$. Indeed, by \cref{heckegrade}, the operator $U_N^2 - N^{-2} S_N$ is $1$-graded and maps each $f_i$ to $K(N, \FF)_\rhobar^i$, so if it annihilates $f$, then it must annihilate each $f_i$. The grading on $A(N, \FF)_\rhobar^\new$ then follows as in \cref{gradingthm}\eqref{alggrade}.
\end{proof} 

%%%%%%%%%%%%%%%%%%%%%%%%%%%%%%%%%%%%%%%%%%%%
%%%%%%%%%%%%%%%%%%%%%%%%%%%%%%%%%%%%%%%%%%%%
%%%%%%%%%%%%%%%%%%%%%%%%%%%%%%%%%%%%%%%%%%%%
\section{The level-$N$ shape deformation condition}\label{steinberg}
%%%%%%%%%%%%%%%%%%%%%%%%%%%%%%%%%%%%%%%%%%%%
%%%%%%%%%%%%%%%%%%%%%%%%%%%%%%%%%%%%%%%%%%%%
%%%%%%%%%%%%%%%%%%%%%%%%%%%%%%%%%%%%%%%%%%%%

Recall that we are assuming that $N$ is a prime different from $p$. In level $1$, we expect $\hat\RRR(1, \FF)_\rhobar$ and $A(1, \FF)_\rhobar$ to be isomorphic reasonably often (details in \cite{BK}; in fact we do not know of any counterexamples). But already in prime level $N$ this is an unreasonable expectation: on the modular forms side, the prime-to-$p$ Artin conductor of a characteristic-zero $\Gamma_0(N)$-modular representation divides $N$ \cite{Car}.
But on the deformation side, the ramification at $N$ is unrestricted. 
In other words, by design and definition, $\hat \RRR(N, \FF)_\rhobar = \hat \RRR(N^j, \FF)_\rhobar$ for any $j \geq 1$, whereas a priori one expects $A(N^2, \FF)_\rhobar$ to surject onto $A(N, \FF)_\rhobar$ with nontrivial kernel.\footnote{In fact, one knows that $A(N^j, \FF)_\rhobar$ stabilizes for $j \gg 0$. For example, 
if $\rhobar$ is modular of level $1$, then it follows from \cite[Proposition 2]{Ca} that $A(N^j, \FF)_\rhobar = A(N^2, \FF)_\rhobar$ for every $j \geq 2$.\label{an2}}

In short, to compare a Hecke algebra of prime level to a deformation ring, we will have to impose a deformation condition. We write $G = G_{\QQ, Np}$, $D_N = D_N(G)$ and $I_N = I_N(G)$ for brevity.

\subsection{Pseudorepresentations of level-$N$ shape}\label{artindef}

Let $B$ be a $\ZZ_p$-algebra, and $(t, d): G \to B$ a pseudorepresentation with $d$ a power of $\chi_p$. We will say that $(t, d)$ has {\bf level-$N$ shape} if 
\begin{equation}
\fbox{for every $d \in D_N$ and $i \in I_N$, we have $t(di) = t(d)$.} 
\end{equation}
Equivalently, $(t, d)$ has level-$N$ shape if the kernel of $\left.(t, d)\right|_{D_N}$ contains $I_N$ (see  \eqref{kerdef}).

The level-$N$-shape condition is meant to capture the notion of a representation having Artin conductor dividing $N$ for pseudorepresentations. Recall that Artin conductor of a Galois representation is defined by measuring dimensions of subspaces of invariants by filtrations of inertia groups. It is not clear how to extend this notion to general Galois pseudorepresentations, as there is no underlying space on which the Galois group is acting.
But in fact, we will show that {a representation has prime-to-$p$ Artin conductor dividing $N$ if and only if its associated pseudorepresentation has level-$N$ shape.}

Let $K$ be a field that {is also a $\ZZ_p$-algebra}, and $\rho: G \to \gl_K(V)$ a two-dimensional representation with $\det \rho$ a power of $\chi_p$. Recall that $\rho$ has \emph{prime-to-$p$ Artin conductor $N$} (respectively, $1$) if the {inertial} invariants $V^{I_N}$ form a one-dimensional (respectively, two-dimensional) subspace of $V$. 

\begin{myprop}\label{fieldnshape} Let $K$, $(\rho, V)$ be as above; let $(t,d) = (\tr \rho, \det \rho)$. Then the following are equivalent. 
\begin{enumerate}[topsep = -5pt]
\item \label{artinn} $\rho$ has prime-to-$p$ Artin conductor dividing $N.$
\item \label{unipn} $\left.\rho\right|_{I_N}$ is unipotent. 
\item \label{twon} $\left.(t, d)\right|_{I_N} = (2,1)$. 
\item \label{decompn} $\left.(t,d)\right|_{D_N}$ splits over $\barr K$ as a sum of two unramified characters.
\item \label{shapen} $t$ has level-$N$ shape.
\end{enumerate}
\end{myprop}

\begin{proof} We show \eqref{artinn} $\iff$ \eqref{unipn} $\implies$  \eqref{decompn} $\implies$ \eqref{twon} $\implies$ \eqref{unipn} and \eqref{decompn} $\implies$ \eqref{shapen} $\implies$ \eqref{twon}.
If $\rho$ has Artin conductor dividing $N$, then $\left. \rho \right|_{I_N}$ is reducible: if it is not trivial, then {$\left. \rho \right|_{I_N}$} has a one-dimensional invariant line {$L \subset V$} and {$I_N$} acts through a character on the quotient $V/L$; since $\det \rho$ is unramified at $N$, this character is trivial. In either case, $\left. \rho \right|_{I_N}$ is unipotent. The converse also holds, so \mbox{\eqref{artinn} $\iff$~\eqref{unipn}}. If $\left. \rho \right|_{I_N}$ is unipotent, then since $I_N$ is normal in $D_N$ with abelian quotient, $\left.\rho \right|_{D_N}$ is upper-triangularizable, possibly after a quadratic extension. Indeed, the normality guarantees that $D_N$ preserves the $1$-eigenspace of $I_N$, on which it acts through its abelian quotient $D_N/I_N$, so that there's a common eigenvector; extending $K$ may be necessary if $I_N$ acts trivially. So \eqref{unipn} implies~\eqref{decompn}. The implication \eqref{decompn} $\implies$ \eqref{twon} is clear. If $\left.(t, d)\right|_{I_N} = (2, 1)$, then by the Brauer-Nesbitt theorem the semisimplification of $\left.\rho\right|_{I_N}$ is trivial, so that $\left.\rho\right|_{I_N}$ is unipotent, so that \eqref{twon}~$\implies$~\eqref{unipn}.   If $\left.(t,d) \right|_{D_N}$ is a sum of unramified characters, then the semisimplification of $\left.\rho\right|_{D_N}$ contains $I_N$ in its kernel. Since semisimplifying $\rho$ does not change its pseudorepresentation, we conclude that $I_N$ is contained in $\ker \left.(t,d) \right|_{D_N}$, so that \eqref{decompn} implies \eqref{shapen}. Finally, if $I_N$ is in the kernel of {$(t,d)|_{D_N}$}, then~$t(i\cdot 1) = t(1) = 2$ for all $i \in I_N$. Therefore \eqref{shapen} implies \eqref{twon}.
\end{proof}

\subsection{Level-$N$ shape as a deformation condition}\label{levelnshapej}

Suppose $\rhobar: G_{\QQ, Np} \to \gl_2(\FF)$ is a semisimple representation with prime-to-$p$ Artin conductor {dividing} $N$. Let $\mathcal D_\rhobar$ be the functor from $\mathcal C$ to sets sending a local $\FF$-algebra $B$ in $\mathcal C$ to the set of odd, constant-determinant pseudodeformations of $\rhobar$ having level-$N$ shape. 

Then $\mathcal D_\rhobar$ is representable by a complete noetherian $\FF$-algebra $\big(\RRR(N, \FF)_\rhobar,\ \nn(N, \FF)_\rhobar\big)$, 
 the quotient of $\hat \RRR(N, \FF)_\rhobar$ by the closed ideal $\hat J_N$ generated by the set $$\{\hat\tau^\univ(di)-\hat\tau^\univ(d): d \in D_N(G_{\QQ, Np}), i \in I_N(G_{\QQ, Np}) \}.$$ This gives us a universal pseudodeformation of $\rhobar$ 
\begin{equation}\label{tauuniv}
\tau^\univ: G_{\QQ, Np} \to \RRR(N, \FF)_\rhobar
\end{equation}
factoring through $\hat \RRR(N, \FF)_\rhobar$ 
If $\ell \nmid Np$ is prime, then set $t_\ell : = \tau^\univ(\frob_\ell)$ and $t'_\ell := t_\ell - \tr \rhobar(\frob_\ell)$. 

As in equation \eqref{taneq}, we can identify the tangent space of this modified universal deformation ring with deformations to the dual numbers: 
\begin{equation}\label{tanlevneq}
\tan \RRR(N, \FF)_\rhobar = \Hom\big (\nn(N, \FF)_\rhobar /\nn(N, \FF)_\rhobar^2,\ \FF\big) \cong {\mathcal D}_\rhobar (\FF[\eps]).
\end{equation}

For $\Gamma_0(1)$-modular $\rhobar$, {write} $\RRR(1, \FF)_\rhobar$ for $\hat\RRR(1, \FF)_\rhobar$, $\phi$ for $\hat \phi$ 
{from~\eqref{RtoA}}, 
and $\tau^\univ$ for $\hat \tau^\univ$ 
{from~\eqref{psdef}}. 
Note that the map $\hat \psi_{N, 1}: \hat \RRR(N, \FF)_\rhobar \onto \RRR(1, \FF)_\rhobar$ described in \eqref{surjR} factors through $\RRR(N, \FF)_\rhobar$:
\begin{equation}\label{psiN1}
\begin{tikzpicture}
\node(hatR) {$\hat \RRR(N, \FF)_\rhobar $};
\node(R)[below = 0.5 cm of  hatR] {$\RRR(N, \FF)_\rhobar$};
\draw(hatR)[->>] -- (R);
\node(R1)[right = 2 cm of R]  {$\RRR(1, \FF)_\rhobar$,};
\draw(hatR)[->>] -- (R1) node[midway, above] {$\scriptstyle \hat\psi_{N,1}$};
\draw(R)[->>] -- (R1) node[midway, below] {$\scriptstyle \psi_{N, 1}$};
\end{tikzpicture}
\end{equation}
because $\hat J_N$, described above, is visibly contained in $\hat J_{N, 1} = \ker \hat\psi_{N, 1},$ described after \eqref{surjR}.

\subsection{Level-$N$ shape and $\Gamma_0(N)$-modular pseudorepresentations}  In this subsection we establish that all $\Gamma_0(N)$-modular pseudorepresentations have level-$N$ shape and record consequences for the relationship between the Hecke algebra and the level-$N$ deformation ring. Recall our notation in this section: $N$ is prime, $G = G_{\QQ, Np}$, $D_N = D_N(G)$ and $I_N = I_N(G)$.

\begin{mythm}[Atkin-Lehner, Carayol]\label{levelelldecomp}
Let $f$ be in $S_k(N, \bar\QQ_p)^\new$. Then $a_N(f) = \pm N^{(k-2)/2}$. Moreover, if $\rho_f: G \to \gl_2(\bar \QQ_p)$ is the attached $p$-adic Galois representation, then $\left.\rho_f \right|_{D_N} \sim 
\begin{pmatrix} \chi_p \, \psi &  \ast \\ 0 & \psi \end{pmatrix},$
where $\psi$ is the unramified character sending $\frob_N$ to $a_N(f)$, and the extension $\ast$ is ramified at $N$.  
\end{mythm}

\begin{proof} The first statement is due to Atkin and Lehner \cite[Theorem 3]{AL}. The second statement is implied by local-global compatibility {established by} Carayol \cite{Car}; see {also} \cite[Section 3]{We2} and \cite[Lemma 2.6.1]{EPW} for a statement in this context. Note that $\psi = \varepsilon \,\chi_p^{(k-2)/2}$ with {$\varepsilon^2=1$}.  
 \end{proof}

\begin{mycor} \label{RNtoA}\leavevmode
\begin{enumerate}[topsep = -5pt]
\item \label{mfnshape} Any $\Gamma_0(N)$-modular pseudorepresentation has level-$N$ shape.

\item \label{Anshape} If $\rhobar: G \to \gl_2(\FF)$ is $\Gamma_0(N)$-modular, then the  pseudorepresentation 
${\tau_{\rhobar}^\hecke}: G \to A(N, \FF)_\rhobar$ 
constructed in \eqref{tauhecke} has level-$N$ shape. 

\item \label{Rnshape} The surjection $\hat \phi: \hat \RRR(N, \FF)_\rhobar \onto A(N, \FF)_\rhobar$ from \eqref{RtoA} factors through $\RRR(N, \FF)_\rhobar$, inducing a continuous surjective map 
$\phi: \RRR(N, \FF)_\rhobar \onto A(N, \FF)_\rhobar$ sending $t_\ell$ to $T_\ell$. 
\end{enumerate}
\end{mycor}

\begin{proof} Any $\Gamma_0(N)$-modular pseudorepresentation comes from a $p$-adic representation $\rho_f$ attached to a $\Gamma_0(N)$-modular eigenform $f$; by \cref{fieldnshape}, it suffices to show that $\rho_f$ satisfies any of the equivalent conditions listed there. If $f$ is a cuspidal newform, then \cref{levelelldecomp} implies that $\rho_f$ visibly satisfies condition \eqref{decompn}. Otherwise $f$ comes from a level-$1$ form, so that $\rho_f$ is unramified at $N$, and satisfies, for example, condition~\eqref{twon}. 
Parts \eqref{Anshape} and \eqref{Rnshape} follow from \eqref{mfnshape} since $\tau_\rhobar^\hecke$ is obtained by gluing characteristic-zero $\Gamma_0(N)$-modular pseudorepresentations and then reducing modulo $p$.
\end{proof}

%%%%%%%%%%%%%%%%%%%%%%%%%%%%%%%%%%%%%%%%%%%%
%%%%%%%%%%%%%%%%%%%%%%%%%%%%%%%%%%%%%%%%%%%%
\section{$U_N$ and the partially full Hecke algebra} \label{pfanal}
%%%%%%%%%%%%%%%%%%%%%%%%%%%%%%%%%%%%%%%%%%%%
%%%%%%%%%%%%%%%%%%%%%%%%%%%%%%%%%%%%%%%%%%%%
We continue the notation of \cref{steinberg}; recall that $N \neq p$ is prime. 
Here we track several consequences of \cref{levelelldecomp,RNtoA}, in particular, the connection between the Atkin-Lehner operator $U_N$ and the trace of Frobenius-at-$N$ elements in the partially full Hecke algebra. 

\subsection{The polynomial satisfied by $U_N$}

Although $\frob_N$ is not well-defined, even up to conjugacy, as an element of $G_{\QQ, Np}$, it does determine a coset of $I_N(G_{\QQ, Np})$ inside $D_N(G_{\QQ, Np})$. Therefore, the level-$N$ shape of 
${\tau_{\rhobar}^\hecke}$ 
guarantees that 
\begin{equation}\label{FN}
F_N := 
{\tau_{\rhobar}^\hecke}
(\frob_N)\ \mbox{is a well-defined element of $A(N, \FF)_\rhobar$}.
\end{equation}
Now fix a $\Gamma_0(N)$-modular $\rhobar$.
In the case that $\rhobar$ is unramified at $N$, the surjection $\pi^\old: A(N, \FF)_\rhobar \onto A(1, \FF)_\rhobar$ from \eqref{Aoldnew} maps $F_N$ to $T_N$ and $a_N(\rhobar)$ to $\tr \rhobar(\frob_N)$. 

Now let $\rhobar$ be an arbitrary $\Gamma_0(N)$-modular representation appearing in weight {$k_\rhobar$}
and let 
\begin{align}\begin{split}\label{charpoly} \charpoly_N(X) & := X^2 - 
{\tau_{\rhobar}^\hecke}
(\frob_N) X + \det \rhobar (\frob_N)  \\ &= X^2 - F_N X + N^{{k_\rhobar}-1} \qquad\qquad\qquad\qquad \in A(N, \FF)_\rhobar[X],
\end{split} \end{align}
be the characteristic polynomial of any $\frob_N$ under $
{\tau_{\rhobar}^\hecke}
$. 

\begin{myprop} \label{ulemma} \label{uellpoly}
If $\rhobar$ is $\Gamma_0(N)$-modular, then $\charpoly_N(U_N) = 0$ in $A(N, \FF)_\rhobar^\pf$. 
\end{myprop}

\begin{proof}
Since $A(N, \FF)_\rhobar^\pf$ acts faithfully on $M(N, \FF)_\rhobar$, and since the action of all the Hecke operators comes from their action on the characteristic-zero space $M(N, {\bar \QQ_p})$, which has a basis of eigenforms, it suffices to show the following: for any $k \geq 0$ even and any normalized Hecke eigenform $f \in M_k(N, {\bar \QQ_p})$, its $U_N$-eigenvalue~$a_N(f)$ is annihilated by $\charpoly_{N, f}(X) := X^2 - \tr \rho_f( \frob_N) X + \det \rho_f(\frob_N)$. Here $\rho_f: G_{\QQ, Np} \to \gl_2({\bar \QQ_p})$ is the $p$-adic Galois representation attached to $f$. 

If $f$ is a newform, then \cref{levelelldecomp} explicitly shows that $a_N(f)$ is an eigenvalue of $\rho_f$ evaluated at any $\frob_N$ (it suffices to consider the semisimplification of the matrix loc.~cit.). Thus $a_N(f)$ is a root of $\charpoly_{N, f}(X)$. On the other hand, if $f$ is an $N$-stabilization of a level-$1$ eigenform $g$, then $\rho_f = \rho_g$, and $a_N(f)$ is an eigenvalue of the matrix $\begin{pmatrix} a_N(g) & 1 \\ -N^{k-1} & 0  \end{pmatrix}$, which gives the action of $U_N$ on the two-dimensional Hecke-stable subspace with basis $g$ and $g(N z)$. Since $a_N(g) = \tr \rho_g(\frob_N)$, the characteristic polynomial of this matrix coincides with $\charpoly_{N, f}(X)$, and the claim follows. 
\end{proof}

\begin{mycor}\label{apffinite}
If $\rhobar$ is $\Gamma_0(N)$-modular, then the map sending $X$ to $U_N$ gives a surjection  
$$A(N, \FF)_\rhobar[X]/ \charpoly_N(X) \onto A(N, \FF)_\rhobar^\pf$$
compatible with the natural inclusion $A(N, \FF)_\rhobar \into A(N, \FF)^\pf_\rhobar$. 
\end{mycor}

\begin{proof}
Follows from \cref{uellpoly}. See also \cite[section 8]{D}. 
\end{proof}

\subsection{A relation between $U_N$ and the image of 
${\tau_{\rhobar}^\hecke}$
} 
Write $\tau$ for $\tau^\hecke_\rhobar$.  
\begin{myprop}
\label{fourlemma} With $i$ any element of $I_N(G_{\QQ, Np})$, any $\frob_N$ in $D_N(G_{\QQ, Np})$, and $g \in G_{\QQ, Np}$, we have 
$$U_N \big( \tau(g i) - \tau(g)\big)= \tau(g i \frob_N) - {\tau}(g \frob_N).$$
In particular, for $g = c$ any complex conjugation,
\begin{equation}\label{funcp} U_N \, {\tau}(c\,i) = \tau(c\, i \frob_N) - \tau(c \frob_N).
\end{equation}
\end{myprop}

\begin{proof} As in the proof of \cref{uellpoly} it suffices to show that for every Hecke eigenform $f \in M(N,\bar \QQ_p)$, we have $a_N(f) \big(\tr \rho_f(g i) - \tr \rho_f(g) \big) = \tr \rho_f (g i \frob_N) - \tr \rho_f(g\frob_N)$, where $\rho_f: G_{\QQ, Np} \to \gl_2(\bar \QQ_p)$ is the $p$-adic Galois representation attached to $f$. If $f$ is old, then $i$ is in the kernel of $\rho_f $, so that on the left-hand side $\tr \rho_f(gi) = \tr \rho_f(g)$ and on the right hand side $ \tr \rho_f (g i \frob_N) =  \tr \rho_f (g \frob_N)$: both sides reduce to $0$. So it suffices to consider $f$ new. In this case, \cref{levelelldecomp} implies that there is a basis for $\rho_f$ so that 
$$\rho_f(\frob_N) = \begin{pmatrix} N a_N(f) & \ast \\ 0 & a_N(f) \end{pmatrix} \quad\mbox{and}\quad \rho_f(i) = \begin{pmatrix} 1 & \ast \\ 0 & 1\end{pmatrix}.$$ The main statement then follows from part \eqref{tot} of \cref{matrixfact} below by setting $T = \rho_f(\frob_N)$, $P = \rho_f(i)$, $M = \rho_f(g)$. For the second statement, take $g = c$ and note that $\tr \rho_f(c) = 0$. 
\end{proof}

\begin{mylemma}\label{matrixfact}
Let $B$ be a ring, $M = \begin{psmallmatrix} a_M & b_M \\ c_M & d_M\end{psmallmatrix} \in M_2(B)$, upper-triangular $T = \begin{psmallmatrix} a_T & b_T \\ 0 & d_T\end{psmallmatrix} \in M_2(B)$, and upper-triangular unipotent $P = \begin{psmallmatrix} 1 & b_P \\ 0 & 1\end{psmallmatrix} \in M_2(B)$. 
Then
\begin{enumerate}[topsep = -5pt]
\item \label{MU} $\tr MP - \tr M = c_M b_P$
\item \label{TZU}\label{tot} $\tr TMP - \tr TM =  d_T c_M b_P = d_T (\tr MP - \tr M )$
\end{enumerate} \end{mylemma}
Part \eqref{MU} is an easy computation; \eqref{TZU} follows from \eqref{MU} by taking $TM$ for $M$; 
\begin{remark}\label{wwerk} By taking $MT$ for $M$ in \cref{matrixfact}, we obtain the similar  $\tr MTP - \tr MT =  a_T (\tr MP - \tr M )$, which leads to $N U_N \big( \tau(g i) - \tau(g)\big) = \tau( g \frob_N i) - \tau(g \frob_N).$ This and \eqref{funcp} are the two Calegari--Specter--style conditions that we use in \cref{rt}. 
\end{remark}

%%%%%%%%%%%%%%%%%%%%%%%%%%%%%%%%%%%%%%%%%%%%
%%%%%%%%%%%%%%%%%%%%%%%%%%%%%%%%%%%%%%%%%%%%
%%%%%%%%%%%%%%%%%%%%%%%%%%%%%%%%%%%%%%%%%%%%
\section{The trivial $\rhobar$ mod $2$}
%%%%%%%%%%%%%%%%%%%%%%%%%%%%%%%%%%%%%%%%%%%%
%%%%%%%%%%%%%%%%%%%%%%%%%%%%%%%%%%%%%%%%%%%%
%%%%%%%%%%%%%%%%%%%%%%%%%%%%%%%%%%%%%%%%%%%%

For the rest of this article, we specialize to $p = 2$ and $\rhobar = \one$, so that we can take $\FF = \FF_2$ and suppress $\FF_2$ from notation. 
Recall that $N$ is an odd prime and set $G := G_{\QQ, 2N}$; we're viewing $\one$ as a representation of $G$. Also let $D_N := D_N(G)$ and $I_N := I_N(G)$. Write $\tau: = \tau^\hecke_{\one}$.

In this setting, the Hecke algebra and the deformation rings all have compatible gradings, which we describe in \cref{rgradesec}. Moreover, the fact that the two Atkin-Lehner eigenvalues are glued together has several ramifications that we explore in \cref{pf2sec,verynewsec}. We review what's known in level one in \cref{level1sec}. 

\subsection{Galois group notation}\label{galoissplit}
We fix additional notation in use for the remainder of this document. 
Recall that the $2$-Frattini quotient of $G$ is
\begin{align}\begin{split}\label{frattinieq}
\frat_{2N} = G/G^2 = \gal\big(\QQ(i, \sqrt{2}, \sqrt{N})/\QQ\big) &\simeq \gal(\QQ(\zeta_8)/\QQ) \times \gal\big(\QQ(\sqrt{N})/\QQ\big)\\ &= (\ZZ/8\ZZ)^\times \times \{\pm 1\}.
\end{split} \end{align}
Let $\eta: G \onto  (\ZZ/8\ZZ)^\times \times \{\pm 1\}$  be the quotient map, 
with components $\eta_2: G \onto (\ZZ/8\ZZ)^\times$ and $\eta_N: G \onto \{\pm 1\}.$
For $i \in (\ZZ/8\ZZ)^\times$, let $G_i := \eta_2^{-1}(i)$, so that $G = G_1 \cup G_3 \cup G_5 \cup G_7$. Further refine each $G_i$ as $G_i^+\cup G_i^-$, with $G^{\varepsilon}_i := \eta^{-1}\big((i, \eps)\big)$. Moreover, for $(i, \eps) \in (\ZZ/8\ZZ)^\times \times \{\pm 1\}$, we'll let $g_i$ be an arbitrary element of $G_i$ and $g_i^\eps$ an arbitrary element of $G_i^\eps$. Finally, for a subset $S$ of $G$, write $\barr S$ for its image in $G/G^2$. 

\begin{mylemma}\label{galoissplitDI}
For any prime $N$ we have 
$$(a) \mbox{ $c$ is in $G_7^+$;} \qquad (b) \mbox{ $\barr I_N = \barr G_1$}; \qquad (c) \mbox{ $\barr D_N = \langle \barr I_N, \barr G_N^+\rangle$.}$$ 
\end{mylemma}

\begin{proof} 
Let $K_1 = \QQ(\zeta_8)$. The first part follows from the fact that $c$ is $7$ in $\gal (K_1/ \QQ) = (\ZZ/8\ZZ)^\times$
and that~$\QQ(\sqrt{N})$ is a totally real field. For the other parts, since $\frob_N$ in $\gal (K_1/ \QQ) = (\ZZ/8\ZZ)^\times$ is the element $N$, the {decomposition group at $N$ in} $\gal(K_1/\QQ) =(\ZZ/8\ZZ)^\times$ is {$\langle N \rangle$}. Since every prime above $N$ of~$K_1$ totally ramifies in $K:=\QQ(\zeta_8, \sqrt{N})$, we 
{know that} $I_N$ is {contained in} the preimage $G_1$ of $\gal(K/K_1)$ in $G$ but not in $G_1^+$; the second part follows. {By the same token,} $D_N$ is {contained in} the preimage of ${\langle N \rangle} \subset (\ZZ/8\ZZ)^\times = \gal(K_1/\QQ)$ in $G$, and we can choose $\frob_N$ such that its image is in $G_N^+$. 
\end{proof}

\subsection{The structure of $\RRR(1)_\one$ and $A(1)$}\label{level1sec}

Before continuing with level $N$, we briefly describe the situation in level $1$. In this case, there is only one $\rhobar$, namely $\one$, so that $A(1) = A(1)_\one$ is a local ring.  Recall that the Galois group $G_{\QQ, 2}$ has $2$-Frattini quotient $\frat_2 = \Gal(\QQ(\zeta_8)/\QQ) = (\ZZ/8\ZZ)^\times$; for $i\in \frat_2$, let $G_{\QQ, 2}^i$ be the coset mapping to $i$ under the natural $2$-Frattini quotient map $G_{\QQ, 2} \onto (\ZZ/8\ZZ)^\times$. We've established a $(\ZZ/8\ZZ)^\times$-grading on $A(1)$ and $\tau$ (\cref{gradingthm}). We describe the structure of $A(1)$ and the isomorphism $\RRR(1)_\one \simeq A(1)$ following Nicolas, Serre, and \jb\ \cite{NS1, NS2, B}. 

Let $B = \FF_2\lb x, y\rb$, an abstract $\FF_2$-algebra that we endow with an $(\ZZ/8\ZZ)^\times$-grading in two different ways as follows. Fix $i = 3$ or $i = 5$ in $(\ZZ/8\ZZ)^\times$. Let $x$ have grading $i$ and $y$ have grading $-i$, so that $B^1 = \FF_2\lb x^2, y^2 \rb$, $B^i = xB^1$, $B^{-i} = y B^1$, and $B^7 = xy B^1$. 

\begin{mythm}[Nicolas, Serre, \jb] \label{level1} \
\begin{enumerate}[topsep = -5pt, itemsep = 5pt] \item For any choice of $g_i \in G^i_{2, \QQ}$ and $g_{-i} \in G^{-i}_{2, \QQ}$, the map 
$$\FF_2\lb x, y\rb \longrightarrow A(1) \qquad \mbox{given by} \quad x \mapsto \tau(g_i), \quad
y \mapsto \tau(g_{-i}),$$
is an isomorphism of $(\ZZ/8\ZZ)^\times$-graded $\FF_2$-algebras. 
In particular, if $i = 3$ then $x \mapsto T_{11}$ and $y \mapsto T_5$ is such an isomorphism; if $i = 5$, then $x \mapsto T_{13}$ and $y \mapsto T_3$ is such an isomorphism. 
\item \label{BR1} The map $\phi: \RRR(1)_\one \onto A(1)$ is an isomorphism; $\RRR(1)_\one$ and $\tau^\univ$ are $(\ZZ/8\ZZ)^\times$-graded.
\item Both $\tau^\univ$ and $\tau$ factor through $G_{\QQ, 2}^\protwo$, preserving the grading.
\end{enumerate}
\end{mythm}

The group $G_{\QQ, 2}^\protwo$ has been studied by Markshaitis and Serre; it has presentation $\langle g, c: c^2 = 1\rangle$ in the category of free pro-$2$ groups, where $g$ is any element that does not fix $\sqrt{2}$ \cite{markshaitis}. 

\cref{level1} begins to clarify why we eventually restrict to $N \equiv 3, 5 \cmod{8}$ for our main results. 

\subsection{The grading on $\hat \RRR(N)_\one$ and $\RRR(N)_\one$}\label{rgradesec}
By \cref{gradingthm} we know that the Hecke algebra $A(N)_\one$ and the modular pseudorepresentation $\tau^\hecke:=\tau^\hecke_{\one, N}$ are graded by $(\ZZ/8\ZZ)^\times$. \jb\ has described a richer grading by all of $G/G^2$ on $\hat \RRR(N)_\one$ and $\hat \tau^\univ:=\hat \tau^\univ_{\one, N}$ as well. The construction makes sense for any profinite group that is $2$-finite in the sense of Mazur, but we restrict to groups of the form $G_{\QQ, 2M}$. 

\begin{mythm}[\jb, unpublished]\label{bellgrade} Let $M \geq 1$ be any odd level, and 
let $\frat$ be the $2$-Frattini quotient of~$G_{\QQ, 2M}$. Then $\hat \RRR(M)_\one$ has a natural $\frat$-grading making $\hat \tau^\univ_{\one, M}$ into a graded $\frat$-pseudorepresentation. 
\end{mythm}

\jb's grading on $B= \hat \RRR(M)_\one$ takes the following shape: for $i \in \frat$, let $G_{\QQ, 2M}^i \subset G_{\QQ, 2M}$ be the corresponding coset and set $B^i$ to be the closed {$\FF_2$}-submodule of $B$ generated by ${\hat\tau_{\one,M}^\univ}(G_{\QQ, 2M}^i)$, along with~$1$ if $i = 1$. The trace-determinant identity \eqref{tracedet} implies that $B' :=\bigoplus_{i \in \Pi} B^i$ is a $\Pi$-graded algebra; universality shows that $B \cong B'$.
For level $M = 1$, \cref{bellgrade} recovers the grading on $\RRR(1)_\one$ from \cref{level1}\eqref{BR1}. For $M = N$ prime, \jb's construction in \cref{bellgrade} endows $\hat \RRR(N)_\one$ with a grading by $(\ZZ/8\ZZ)^\times \times {\FF_N^\times/ (\FF_N^\times)^2}$. The proof of \cref{rgrade} below suggests that the {grading by} {$\FF_N^\times/ (\FF_N^\times)^2$} is lost when we pass from $\hat \RRR(N)_\one$ to $\RRR(N)_\one$, leaving only a $(\ZZ/8\ZZ)^\times$-grading --- as befits a deformation ring approximating $A(N)_\one$.

\begin{mycor} \label{rgrade} The level-$N$-shape universal deformation ring $\RRR(N)_\one$ has a natural $(\ZZ/8\ZZ)^\times$-grading, so that $\tau^\hecke$ is a $(\ZZ/8\ZZ)^\times$-graded pseudorepresentation and $\phi: \RRR(N)_\one \onto A(N)_\one$ is a map of  $(\ZZ/8\ZZ)^\times$-graded $\FF_2$-algebras. 
\end{mycor} 

\begin{proof} 
By \cref{bellgrade}, $\hat \RRR(N)_\one$ has a grading by $(\ZZ/8\ZZ)^\times \times {\FF_N^\times/ (\FF_N^\times)^2}$, so it suffices to show that the kernel of the quotient map $\hat \RRR(N)_\one \onto \RRR(N)_\one$ is $(\ZZ/8\ZZ)^\times$-graded. This kernel is the ideal $\hat J_N$ defined in \cref{levelnshapej}, topologically generated by elements of the form $\hat \tau^\univ(di) - \hat \tau^\univ(d)$ for $d \in D_N$ and~$i \in I_N$; to see that $\hat J_N$ is graded, we want $\hat \tau^\univ(di)$ and $\hat \tau^\univ(d)$ {to be} in the same $(\ZZ/8\ZZ)^\times$-component of $\hat \RRR(N)_\one$. Since~$\hat \tau^\univ$ is graded, it suffices to know that the map $D_N \to \Gal(\QQ(\zeta_8)/\QQ) = (\ZZ/8\ZZ)^\times$ factors through~$D_N/I_N$. But this is clear as $\QQ(\zeta_8)$ is unramified at $N$.
\end{proof}

\subsection{The partially full Hecke algebra} \label{pf2sec} We specialize \cref{pfanal} to $p = 2$. In this setting, the polynomial satisfied by $U_N$ over $A(N)$ (\eqref{charpoly} and \cref{uellpoly}) has the form
\begin{equation}\label{Pmod2} \PPP_N(X) = X^2 + F_N X + 1.
\end{equation}
Here recall that $F_N = \tau^\hecke(\Frob_N)$. In particular, since $a_N(\one) = 0$, the polynomial $\PPP_N(X)$ has a double root residually modulo $\mm(N)_\one$, so that from the description of the maximal ideals of the partially full Hecke algebra in \cref{Apfsec}, it's clear that
\begin{equation}\label{Apflocal} 
\mbox{$A(N)_\one^\pf$ is a complete local noetherian ring with residue field $\FF_2$.}
\end{equation}
The same will be true for any $N$-old $\rhobar$ mod $2$ satisfying the level-raising condition of {\cref{newoldhecke}}.
\label{apfdisc}

Let $U_N' = U_N + 1$, so that $U_N' \in \mm(N)_\one^\pf$. Then a simple manipulation of \eqref{Pmod2} implies that 
\begin{equation}\label{fun2}
{F_N} = {F_N} U_N' - (U_N')^2 
\end{equation}
from which it's immediately clear that
\begin{equation}\label{fninmpf2}
{F_N} \in (\mm(N, \FF)^\pf_\rhobar)^2.
\end{equation}
Finally, we record \cref{fourlemma} in our setting: for $i \in I_N$ we have 
\begin{equation} \label{fourlemma2}
U_N \, \tau^\hecke(c\,i) = \tau^\hecke(c\, i \frob_N) - \tau^\hecke(c \frob_N).
\end{equation}

\subsection{New and \emph{very new} forms and their Hecke algebras} \label{verynewsec}
Specializing \cref{oldformsnewforms} to $p = 2$ tells us that the subspace of newforms inside $M(N)$ is $\ker (U_N^2 - N^{-2} S_N) = \ker (U_N')^2$. We similarly define $K(N)^\new := K(N) \cap M(N)^\new$, so that 
\begin{equation}\label{knewdef}
M(N)^\new = M(N)[(U_N')^2] \mbox{ and } K(N)^\new = K(N)[(U_N')^2].
\end{equation}
From this description and the duality in \eqref{dualJ} we deduce that 
\begin{align}
\label{anewpf}
& A(N)^\new = A(N)^\pf/(U_N')^2.
\end{align}
In particular, $A(N)^\new$ will have nilpotent elements. 
We therefore define, in an ad-hoc way for $p = 2$, the spaces of \emph{\verynew}\ modular forms: 
\begin{equation*}
\label{kvnewdef} M(N)^\elleng:= \ker U_N' \subseteq M(N) \mbox{ and } 
K(N)^\elleng := K(N)[U_N'].
\end{equation*}
With this definition, the newforms are generalized very new forms. 
Like the newforms, the very new forms break up into local components for various $\rhobar$; we in particular consider $$M(N)^\elleng_\one: = M(N)^\elleng \cap M(N)_\one \mbox{ and } 
K(N)_\one^\elleng := K(N)_{\one}[U_N'].$$

Let $A(N)^\elleng$ be the largest quotient of $A(N)$ acting faithfully on $M(N)^\elleng$; extending scalars as necessary, it too breaks up into local $\rhobar$-components, with $A(N)^\elleng_\rhobar$ the largest quotient of $A(N)$, or of $A(N)^\elleng$, acting faithfully on $M(N)_\rhobar^\elleng$. By \eqref{dualJ} again, 
\begin{equation}\label{avnewpf}
A(N)^\elleng = A(N)^\pf/(U_N') \mbox{ and } A(N)^\elleng_\one = A(N)_\one^\pf/(U_N').
\end{equation}
In particular, 
\begin{equation}\label{avnewfaith}
\mbox{$A(N)^\elleng$ acts faithfully on $K(N)^\elleng$.}
\end{equation}
\begin{remark}\label{Avnewgrade} One can show that $K(N)^\elleng$, $A(N)^\elleng$, and the modular $A(N)^\elleng$-valued pseudorepresentation $\tau^{\elleng}$,
 as well as all their corresponding $\rhobar$ component analogues, are naturally and compatibly graded by $(\ZZ/8\ZZ)^\times/\langle N \rangle$ in the sense of \cref{gradingthm}. 
\end{remark}

%%%%%%%%%%%%%%%%%%%%%%%%%%%%%%%%%%%%%%%%%%%%
%%%%%%%%%%%%%%%%%%%%%%%%%%%%%%%%%%%%%%%%%%%%
%%%%%%%%%%%%%%%%%%%%%%%%%%%%%%%%%%%%%%%%%%%%
\section{Infinitesimal deformations of the trivial $\rhobar$ mod $2$}
%%%%%%%%%%%%%%%%%%%%%%%%%%%%%%%%%%%%%%%%%%%%
%%%%%%%%%%%%%%%%%%%%%%%%%%%%%%%%%%%%%%%%%%%%
%%%%%%%%%%%%%%%%%%%%%%%%%%%%%%%%%%%%%%%%%%%%

We continue our notation from the previous section; in particular $N$ is an odd prime.
We analyze first-order deformations of the pseudorepresentation associated to the representation $\one$ of $G_{\QQ,2N}$ to compute tangent space dimensions of $\hat \RRR(N)_\one$ and $\RRR(N)_\one$ and their maximal reduced quotients. For $N \equiv 3$ or $5$ modulo~$8$ we also find the generators of the cotangent space of $\RRR(N)_\one^\red$.

\subsection{Tangent dimensions of $\hat \RRR(N)_\one$ and $\RRR(N)_\one$} 

\begin{mylemma}\label{chenny} Let $G^2$ be the closed subgroup of $G$ generated by the squares of elements of $G$. Then
\begin{enumerate}[itemsep = 5pt, topsep = -3pt]
\item \label{plaintan} $\tan \hat \RRR(N)_\one \cong \{\mbox{set maps}\ b: G/G^2 \to \FF_2 \mid b(1) = b(c) = 0\};$
\item \label{leventan}
$\tan \RRR(N)_\one \cong  \{\mbox{set maps}\ b: G/G^2 \to \FF_2 \mid b(1) = b(c) = 0,\ b(di)=b(d) \text{ for all } d \in \barr D_{N}, i \in \barr I_{N}\}$.
\end{enumerate}
\end{mylemma}

\begin{proof}
See also \cite[Lemma 5.3]{C}. We use equations \eqref{taneq} and \eqref{tanlevneq}, and identify a pseudodeformation $t = \eps b: G \to \FF_2[\eps]$ of the trivial mod-$2$ representation with the set-theoretic map $b: G \to \FF_2$. The trace-determinant identity for elements $g$ and $gh$ of $G$ on $t$ 
simplifies to $b(g^2 h) = b(h)$, so that $b$ factors through~$G/G^2$. 
Since the latter is abelian, the condition $t(gh)=t(hg)$ is automatically satisfied for all~$g,h \in G$. The condition $t(1) = 2 = 0$ forces $b(1) = 0$; oddness forces $b(c) = 0$. The 
additional requirements on $b$ in \eqref{leventan} come from the level-$N$ shape condition.
\end{proof} 

\begin{mycor}\label{tandim} 
$\dim \tan \hat \RRR(N)_\one = 6.$ 
\end{mycor}

\begin{proof}
From \eqref{frattinieq}, $G/G^2 = \gal(\QQ(\sqrt{-1}, \sqrt{2}, \sqrt{N})/\QQ)$. Now we use \cref{chenny}~\eqref{plaintan}.
\end{proof}

The following lemma will help us with $\tandim \RRR(N)_\one$ and follows immediately from \cref{galoissplitDI}. 

\begin{mylemma}\label{decomplem}\leavevmode
\begin{enumerate}[topsep = -5pt]
\item\label{n1l} If $N \equiv 1 \pmod{8}$, then $|\barr D_{N}|=|\barr I_{N}|=2$ and $c \not\in \barr D_{N}$. 
\item\label{n35l} If $N \equiv 3,5 \pmod{8}$, then $|\barr D_{N}|=4$, $|\barr I_{N}|=2$, and $c \not\in \barr D_{N}$. 
\item\label{n7l} If $N \equiv 7 \pmod{8}$, then $|\barr D_{N}|=4$, $|\barr I_{N}|=2$, and $c \in \barr D_{N} \setminus \barr I_{N}$. \end{enumerate}
\end{mylemma}

\begin{mycor}\label{tandim}
$\tandim {\RRR(N)_\one} = \begin{cases} 
5 & \mbox{if $N \equiv 1$ modulo $8$}\\
4 & \mbox{otherwise}.
\end{cases}$

In particular, suppose $N \equiv 3, 5$ modulo $8$. Choose any $g^{\pm}_N \in G_N^\pm,$ any $g_{-N}^+ \in G_{-N}^+$, any $g_{-N}^- \in G_{-N}^-$, and any $g_7^- \in G_7^-.$ The maximal ideal of $\RRR(N)_\one$ is generated by the images under $\tau^\univ$ of these four elements.
\end{mycor}

\begin{proof}
We use \cref{chenny}\eqref{leventan} and translate the conditions on $b$ via \cref{decomplem}, taking cases.
If \mbox{$N \equiv 1~\pmod{8}$} then the level-$N$ conditions on $b$ translate to $b(\barr D_{N} \cup \{c\})=\{0\}$. Hence we conclude that $\tandim {\RRR(N)_\one} =5$. If $N \equiv 7 \pmod{8}$, then the conditions give us $b(\barr D_{N})=\{0\}$, so we get that $\tandim \RRR(N)_\one =4.$ And if~$N \equiv 3,5 \pmod{8}$, then we get $b(\barr I_{N} \cup \{c\})=\{0\}$ and $b(g)=b(g')$ whenever $\{g,g'\}=\barr D_{N} \setminus \barr I_{N}$. In this case, we also get $\tandim \RRR(N)_\one = 4$. For the refinement for $N \equiv 3, 5 \cmod{8}$, use \cref{galoissplitDI}. 
\end{proof}

\subsection{Representation-theoretic lemmas}

In this subsection, we will study the properties of a pseudodeformation of the trivial mod-$2$ representation taking values in a domain. These results will be used to control the dimension of the tangent space of $\RRR(N)_\one^\red$. 

Before stating \cref{tracedecomp0}, we recall a result due to Chenevier that we will use in its proof. 
\begin{mylemma}{\cite[Lemma 3.8]{C}}
\label{factorthrough}
Suppose $B$ is an coefficient algebra and $t: G \to B$ 
is a pseudodeformation of the trivial mod-$2$ representation. Then $t$ factors through $\Gprotwo$. 
\end{mylemma}

\begin{myprop}\label{tracedecomp0}
Let $B$ be an integral domain coefficient algebra and $K$ its field of fractions with algebraic closure $\barr K$. Suppose that $t: G \to B$ is a pseudodeformation of the trivial mod-$2$ representation. Then $t$ is the trace of a semisimple representation $\rho: G \to \sl_2(\barr K)$. Moreover: 
\begin{enumerate}[topsep = -3pt, itemsep = 3pt]
\item \label{pro2} $\rho$ factors through $\Gprotwo$.
\item
\label{unip} $\left.\rho\right|_{I_N}$ is unipotent: that is, $\left.\rho\right|_{I_N} \sim \begin{pmatrix} 1 & \eta \\ 0 & 1 \end{pmatrix}$ for some additive character $\eta: I_{N} \to \barr K.$
\item \label{commy} $\left.\rho\right|_{D_N}$ has abelian image. 
\item \label{tracedecomp0item} Either $\rho(I_N)$ is trivial or 
  $\left.\rho\right|_{D_N}$ is unipotent.
\item \label{kerishi} $t$ has level-$N$ shape: $I_N \subseteq \ker \left.t \right|_{D_N}$.
\item \label{fourtermitem} For all $g \in G$, $d \in D_N$, $i \in I_N$, we have 
$t(gi)-t(g) = t(dgi) - t(dg).$
\end{enumerate}
\end{myprop}

\begin{proof}
The existence of $\rho$ with $\tr\rho=t$ is a theorem of Chenevier \cite[Theorem 2.12]{C}. 
\begin{enumerate}[topsep = -3pt, itemsep = 3pt]
\item Follows from \cref{factorthrough}, the semisimplicity of $\rho$, and the Brauer-Nesbitt theorem.
\item By \eqref{pro2}, the restriction $\left.\rho\right|_{I_N}$ factors through the pro-$2$ tame-inertia quotient of $I_N$, which is isomorphic to $\ZZ_2$. Therefore $\left.\rho\right|_{I_N}$ is reducible, and hence an extension of a character $\chi: \ZZ_2 \to {\barr K}^\times$ by $\chi^{-1}$. We claim that $\chi$ has finite order. Indeed, let $a$ be a lift of \mbox{a generator} of the $\ZZ_2$-quotient of $I_N$, and~$\sigma$ a lift of $\frob_N$ 
to {$D_N$}. Then $\sigma a \sigma^{-1} = a^N$  \cite[\href{https://stacks.math.columbia.edu/tag/0BU5}{Lemma~0BU5}]{stacks-project}. It follows that $\rho(a)$ and~$\rho(a)^N$ are conjugates of each other, so that if $\lambda$ is an eigenvalue of $\rho(a)$, then so is $\lambda^N$; in other words,~$\lambda^{N} = \lambda^{\pm 1}$. Therefore every eigenvalue of $\rho(a)$ has finite order; since $\rho(a)$ generates the image of~$\rho|_{I_N}$, the character $\chi$ does too.
Finally, 
any finite-order character from $\ZZ_2$ to a field of characteristic $2$ must be trivial.
\item From (\ref{unip}), we get that $\left.\rho\right|_{I_N} \sim \begin{pmatrix} 1 & \eta \\ 0 & 1 \end{pmatrix}$ for some additive character $\eta: I_{N} \to \barr K$ of order $1$ or $2$, depending on whether $\rho$ is ramified at $N$. Since $N$ is odd, we have $\rho(a)^N = \rho(a)$ in either case, so that $\rho(a)$ and $\rho(\sigma)$ commute.

\item Since $I_N$ is normal in $D_N$ with abelian quotient, $\left.\rho\right|_{D_N} = \begin{psmallmatrix} \alpha & \eta' \\ 0 & \alpha^{-1} \end{psmallmatrix}$ for some unramified character $\alpha: D_N \to {\barr K}^\times$ and extension class $\eta'$. Since $\rho(D_N)$ is abelian, we have, for every $i$ in $I_{N}$ and~$d$ in $D_{N}$, the matrix  $\rho(i) = \begin{psmallmatrix} 1 & \eta(i) \\ 0 & 1\end{psmallmatrix}$ commuting with $\rho(d) = \begin{psmallmatrix} \alpha(d) & \eta'(d) \\ 0 & \alpha^{-1}(d) \end{psmallmatrix}$. It follows that $\eta'(d) + \eta(i) \alpha^{-1}(d) = \eta(i) \alpha(d) + \eta'(d)$. If $\eta$ is nontrivial, then $\alpha = \alpha^{-1}$, so $\alpha = 1$ as we are in characteristic $2$. Otherwise $\rho(I_N) = 1$.
\item 
One checks that $t(di) = t(d)$ for all $d \in D_N$, $i \in I_N$ in both the possible settings from~\eqref{tracedecomp0item}.
\item Immediate if $\rho(I_N)$ is trivial; otherwise use \cref{matrixfact}.
\qedhere 
\end{enumerate}
\end{proof}
In summary,
the level-$N$ shape condition is automatic for pseudodeformations of $\one$ to a domain. In the next subsection we use this observation to show that the quotient map $\hat \RRR(N)_\one^\red {\onto} \RRR(N)_\one^\red$ induced from {the natural surjective map $\hat \RRR(N)_\one \onto \RRR(N)_\one$} (\cref{levelnshapej}) is an isomorphism. 

\subsection{The tangent space to $\RRR(N)_\one^\red$}\label{rredsec}

We now turn our attention to $\RRR(N)_\one^\red$, in particular its tangent dimension. For brevity, write $\RRR := \RRR(N)_\one$ and $\hat \RRR:=\hat \RRR(N)_\one$, and $\nn$ and $\hat \nn$, respectively, for their maximal ideals; keep the notation introduced above for $G = G_{\QQ, 2N}$, $\barr I_N := I_N(G/G^2)$ and $\barr D_N:= D_N(G/G^2)$. 

If $(B, \mm)$ is any integral-domain coefficient algebra carrying a pseudodeformation $t$ of $1 \oplus 1$, then \cref{tracedecomp0} implies that $t$  automatically has level-$N$ shape, so that the map $\hat \RRR \to B$ guaranteed by {the universal property} factors through the quotient map $\hat \RRR \to \RRR$. If~$\alpha :\RRR \to B$ satisfying $t = \alpha \circ \tau^\univ$ is surjective, then~$\alpha$ induces a surjection of finite-dimensional $\FF_2$-vector spaces $\hat\nn/\hat\nn^2 \onto \nn/\nn^2 \stackrel \alpha \onto \mm/\mm^2$, corresponding to an embedding of tangent spaces (see also \cref{chenny}): 
\begin{equation}\label{subtaneq}
\begin{tabular}{cccccccc}&$\tan B$& $=$ & 
$(\mm/\mm^2)^\ast$ &$\xrightarrow[\sim]{h \mapsto h \circ t}$& 
$T_B$\\
&$ \cap $& & $\cap$ && $\cap$ \\
&$\tan \RRR$ &=&
$ (\nn/\nn^2)^\ast $& $\xrightarrow[\sim]{h \mapsto h \circ \tau^\univ}$&
$T_{ \RRR} $&$: =$& $\{b \in T_{\hat \RRR}: b(di) = b(d)\ {\rm for\ all}\ d \in \barr D_N, i \in \barr I_N\}$\\
&$ \cap $& & $\cap$ && $\cap$ \\
&$\tan \hat \RRR$ &=&
$ (\hat \nn/\hat \nn^2)^\ast $& $\xrightarrow[\sim]{h \mapsto h \circ \hat\tau^\univ}$&
$T_{ \hat \RRR} $&$: =$& $\{\mbox{set maps}\ b: G/G^2 \to \FF_2 : b(1) = b(c) = 0\}.$
\end{tabular}
\end{equation}
To understand $\RRR^\red$, we begin to analyze $\RRR/\pp$ for prime ideals $\pp$.

\begin{myprop}\label{redlemma}\leavevmode
\begin{enumerate}[topsep = -3pt]
\item Every prime ideal $\hat \pp$ of $\hat \RRR$ is the preimage of a prime ideal $\pp$ of $\RRR$, so that $\hat \RRR/\hat \pp = \RRR/\pp$.
\item $\hat \RRR^\red = \RRR^\red$
\end{enumerate}
\end{myprop}

\begin{proof}
The second part follows from the first, so let $\hat \pp$ be a prime ideal of $\hat \RRR$. 
Let $t: G \to  \hat \RRR /\hat \pp$ be the pseudorepresentation obtained by composing $\hat\tau^\univ: G \to \hat\RRR$ with the quotient map $\hat\RRR \onto \hat\RRR/\hat \pp$. By \cref{tracedecomp0}\eqref{kerishi},~$t$ has level-$N$ shape so that the quotient map $\hat \RRR \to \hat \RRR /\hat \pp$ factors through $\RRR$; let $\alpha: \RRR \to \hat \RRR /\hat \pp$ be the corresponding map (that is, satisfying $t = \alpha \circ \tau^\univ$) and let $\pp = \ker \alpha$. Then $\RRR/\pp = \hat \RRR /\hat \pp$ since $\hat \RRR /\hat \pp$ is the trace algebra of $t$.
\end{proof}

\begin{myprop}
\label{tanlemma} Let $\pp$ be a prime ideal of $\RRR$. Then  
$$\tandim {\RRR/\pp} \leq \begin{cases} 
5 & \mbox{if $N \equiv 1$ modulo $8$;}\\
2 & \mbox{if $N \equiv 3, 5$ modulo $8$;} \\
3 & \mbox{if $N \equiv 7$ modulo $8$.}
\end{cases}$$ 
\end{myprop}

\begin{proof}
Write $K$ for the field of fractions of $B:=\RRR/\pp$. From \cref{tracedecomp0}, $t$ is the trace of a unique semisimple representation $\rho: G \to \sl_2(\barr K)$. If $\rho$ is unramified at $N$, then $\alpha$ factors through the quotient~$\RRR \onto \RRR(1)$: see  \eqref{kerdef} and \eqref{psiN1}. By \cref{level1}, $$\tandim B \leq \tandim \RRR(1) = 2.$$ Otherwise, \eqref{subtaneq} above and \cref{tracedecomp0}\eqref{tracedecomp0item},\eqref{fourtermitem} tell us that $t$ satisfies the following properties:
\begin{enumerate}[topsep = -3pt]
\item\label{zerocond} $t(c) = t(D_N) = 0$; 
\item\label{fourtermcount} $t(g) - t(gi) = {t(dg) - t(dgi)}$ for $g \in G$, $d, \in D_N$, $i \in I_N$. 
\end{enumerate}

Let $b\in T_{B}$. By following the maps in \eqref{subtaneq} we see that $b$ is a set map from $G/G^2 \cong (\ZZ/2\ZZ)^3$ to $\FF_2$ subject to the same conditions as $t$ projected to $G/G^2$. We use \cref{decomplem} repeatedly. If $N \equiv 1 \bmod{8}$, then~$\{c \} \cup \barr D_{N}$ has size $3$, so that $b$ has at most $8-3=5$ degrees of freedom and $\tandim B \leq 5$. Otherwise, we can choose~$i \in \barr I_N$ and $d \in \barr D_N$ so that $\{1, i, d, di\}$ are the distinct elements of $\barr D_N$. 
The condition~\eqref{fourtermcount} for~$g \not\in \barr D_N$ implies that the sum of the $b$-values on $G/G^2 - \barr D_N$ is zero. 
Since $\barr D_N$ has size $4$ itself, that is a total of $5$ independent conditions on~$b$, so that $\tandim B \leq 3$. Finally, for $N \equiv 3, 5  \cmod{8}$, the condition~$b(c) = 0$ is an additional independent condition, so that $\tandim B \leq 2$.
\end{proof}

In any case since $ \RRR(1) \simeq \FF_2\lb x, y \rb$ is a quotient of $\mathcal \RRR$ (\cref{level1}), the following corollary is immediate.

\begin{mycor}\label{krulldimR}
If $N \equiv 3, 5 \bmod{8}$, then $\dim \hat \RRR = \dim \RRR = 2$.
\end{mycor}

We now analyze the tangent dimension of $\RRR^\red$. 

\begin{myprop}
\label{propdef}
$\tandim \mathcal R^{\red} \leq \begin{cases} 5 & \mbox{ if $N \equiv 1 \bmod{8}$,} \\ 3& \mbox{ otherwise.} \end{cases}$ 
In particular, suppose  $N \equiv 3, 5 \cmod{8}$.
Choose $$g^{\pm}_N \in G_N^\pm,\quad g_{-N}^+ \in G_{-N}^+,\quad g_{-N}^- \in G_{-N}^-, \quad \mbox{and} \quad g_7^- \in G_7^-.$$ Then the maximal ideal of $\RRR^\red$ is generated by $\tau^\red (g_N^\pm)$ and any 
two 
of $\tau^\red(g_{-N}^+),$ $\tau^\red(g_{-N}^-),$ and $\tau^\red(g_7^-).$
\end{myprop}

\begin{proof}
 The proof is similar to that of \cref{tanlemma} and repeatedly uses \cref{decomplem}. By \cref{tracedecomp0}, the following hold modulo every prime ideal of $\RRR$, so that they hold in $\RRR^\red$:  
\begin{enumerate}
\item \label{1c} $\tau^\red(c) = \tau^\red(I_N) = 0.$
\item \label{2c} $\tau^\red(d i) = \tau^\red(d)$ for all $d \in D_N$ and $i \in I_N$.
\item \label{3c} $\tau^\red(gi) - \tau^\red(g) = {\tau^\red(dgi) - \tau^\red(dg)}$ for all $g \in G$, $d \in D_N$, and $i \in I_N$.
\end{enumerate} 
For $N \equiv 1 \cmod{8}$, then $\barr I_N = \barr D_N$, and the analysis and the conclusion are analogous to those of \cref{tanlemma}.
Otherwise, 
$\barr I_N$ has size $2$ and index $2$ in $\barr D_N$, which in turn has has index $2$ in $G/G^2$. For any $b \in T_{\RRR^\red}$, the first condition tells us that $b$ is zero on the $3$-element set $\bar I_N \cup \{c \}$. The second condition tells us that the $b$-values on the two elements of $\barr D_N$ that are not in $\barr I_N$ coincide. Now the third condition, as in the proof of \cref{tanlemma}, means that the sum of the $b$-values on $G/G^2 - \barr D_N$ is zero. {Thus we get} $5$ independent conditions, so that $\tandim \RRR^\red \leq 3$. For the refinement in the case $N \equiv 3, 5 \cmod{8}$, compare to \cref{tandim}. 
\qedhere
\end{proof}

\section{Hecke tangent dimensions at the trivial $\rhobar$ mod $2$}
Here we study $A(N)_\rhob := A(N, \FF_2)_\rhob$, the local component of trivial mod-$2$ representation of the big Hecke algebra acting on $M(N, \FF_2)$. 
Recall that $N$ is an odd prime.

\subsection{The shallow Hecke algebra and its reduced quotient}

As a consequence of the surjective map $\RRR(N)_\one \onto A(N)_\one$ (\cref{RNtoA}\eqref{Rnshape}) as well as \cref{tandim,propdef} we get the following:

\begin{mycor}\label{hecketandim}
$$(1)\ \tandim A(N)_\rhob \leq \begin{cases} 5 & \mbox{ if $N \equiv 1 \bmod{8}$} \\ 4& \mbox{ otherwise.} \end{cases} \qquad (2)\ \tandim A(N)^{\red}_\rhob \leq \begin{cases} 5 & \mbox{ if $N \equiv 1 \bmod{8}$} \\ 3& \mbox{ otherwise.} \end{cases}$$
\end{mycor}

\subsection{The partially full Hecke algebra of level $N$} 
From the discussion in \cref{pf2sec} $A(N)_\one^\pf$ is a complete local noetherian ring, an $A(N)_\one$-algebra whose maximal ideal $\mm^\pf := \mm(N)^\pf_\one$ is generated by $U_N' := U_N + 1$ together with $\mm:=\mm(N)_\one$. Moreover with $\tau: G \to A(N)_\one$ the modular pseudodeformation of $\one$, for any $\frob_N$ element in $D_N$, 
the element $F_N := \tau(\frob_N) \in A(N)_\one$ is well defined independent of the choice and satisfies $F_N \in (\mm^\pf)^2$ \eqref{fninmpf2}.  Finally, \eqref{fourlemma2} implies that for any $i \in I_N$ and any complex conjugation $c$, so that
$$U_N' \, \tau(c\,i) = \tau(c\,i) + \tau(c \frob_N) + \tau(c\, i \frob_N),$$
so that $\tau(c\,i) + \tau(c \frob_N) + \tau(c\, i \frob_N) \in (\mm^\pf)^2$ as well.

\begin{mylemma}
\label{pflem}
If $N \equiv 3, 5 \cmod{8}$, then $\tandim A(N)_\rhob^\pf \leq 3$, with the maximal ideal generated by $U_N'$, $\tau(ci)$, and $\tau(c \frob_N)$ --- or, more generally, by $U_N'$, $\tau(g_7^-)$, and $\tau(g_{-N}^\pm)$ for any $g_7^- \in G_7^-$ and any $g_{-N}^\pm \in G_{-N}^\pm$. 
\end{mylemma}
\begin{proof} 
Let $i \in I_N$ be a lift of a generator of its $\ZZ_2$ tame-inertia quotient. It follows from the proof of \cref{tandim} that the images of $\tau(c \,i)$, $\tau(c \frob_N)$,  $\tau(c i \frob_N)$ and $ \tau(\frob_N)$ are an $\FF_2$-basis for the cotangent space $\mm/\mm^2$. Therefore, so are 
$$\tau(c \,i),\ \tau(c\frob_N),\ \tau(\frob_N),\ \mbox{and}\ \tau(c\,i) + \tau(c \frob_N) + \tau(c\, i \frob_N).$$
So these four elements generate $\mm$ as in ideal. Therefore $\mm^\pf$ is generated by these four elements together with $U_N'$. Since both $\tau(\frob_N)$ and $\tau(c\,i) + \tau(c \frob_N) + \tau(c\, i \frob_N)$ are in $(\mm^\pf)^2$, the cotangent space of the partially full Hecke algebra is spanned by the images of $U_N'$, $\tau(c \,i)$, and $\tau(c\frob_N).$ The claim follows. 
\end{proof}

\subsection{The Hecke algebra on very new forms} 

As above, let $i \in I_N$ be a lift of the generator of its $\ZZ_2$ quotient, $\frob_N \in D_N$ any Frobenius-at-$N$ element, and $c \in G$ any complex conjugation.
\begin{mylemma} \label{fnto0}
The element $F_N = \tau^\hecke(\frob_N)$ of $A(N)_{\one}$ maps to zero in $A(N)_\one^\elleng$. 
\end{mylemma}

\begin{proof}
Since $U_N' = 0$ in $A(N)_\one^\elleng$ and $F_N = F_N U_N' - (U_N')^2$  \eqref{fun2},  we have $F_N = 0$ in $A(N)_\one^\elleng$.
\end{proof}
 
\begin{mylemma}\label{englem}
If $N \equiv 3, 5 \pmod{8}$, then $\tandim A(N)^\elleng_\rhob \leq 2$. Its maximal ideal is generated by {the images of} $\tau(ci)$ and either $\tau(c \frob_N)$ or $\tau(c i \frob_N)$. More generally still, it is generated by any $\tau(g_7^-)$ and $\tau({g^{\pm}_{-N}})$. 
\end{mylemma}

\begin{proof}
Let $\mm^\elleng$ be the maximal ideal of $A(N)_{\one}^\elleng$. By construction $U_N'$ kills $M(N)_\rhob^\elleng$. Therefore there is a surjective map $A(N)_\rhob^\pf \onto A(N)_{\one}^\elleng$ sending $U_N'$ to zero and restricting the action of the other Hecke operators. The proof of \cref{pflem} (and using the same notation) tells us that the images of $\tau^\hecke(ci)$ and $\tau^\hecke(c \frob_N)$ span {$\mm^\pf /\big((\mm^\pf)^2,U_N'\big)$} as an $\FF_2$-vector space. That space surjects onto $\mm^\elleng/(\mm^\elleng)^2$, proving the lemma.
\end{proof}

\section{Mod-2 modular forms of level $3$ and $5$ }\label{9}

In this section, we restrict ourselves to $N=3,5$ and study the properties of $A(N):=A(N, \FF_2)$ in these cases. We determine the structure of $M(N) := M(N, \FF_2)$ (\cref{struct}), prove that $A(N)$ is a local ring (\cref{loc}), and compute the exact dimensions of the tangent space of $A(N)$. 

\subsection{$M(3)$ and $M(5)$ are polynomial algebras} \label{struct} 
In level $1$, Swinnerton-Dyer showed that $M(1) = \FF_2 [\barr \Delta]$, where $\barr \Delta = q + q^9 + q^{25} +\cdots $ is the mod-$2$ $q$-expansion of $\Delta$, the unique normalized cuspform of level~$1$ and weight $12$ \cite[Theorem 3]{SwDy}. The fact that $M(1)$ is a polynomial algebra is a key ingredient for all the known proofs of the structure of $A(1)$ \cite{NS2,Mpaper,Mathilde,Monsky1}. In levels 3 and 5, we use an analogous structure result to allow us to apply the nilpotence method \cite{Mpaper} to prove that $\dim A(N)^\elleng \geq 2$ (\cref{nilpy}). The structure of $M(3)$ and $M(5)$ is certainly known to experts and we make no claim of originality.

\begin{mylemma}\label{lemma:polyalg}
If $N = 3, 5$, then $M(N) =  \FF_2[f_N]$ for some  $f_N \in M(N)$. More precisely, 
\begin{align*}
f_3 &= \barr E_4^{3\myhyphen{\rm crit}} = q + q^{2} + q^{3} + q^{4} + q^{6} + q^{8} + q^{9} + q^{12} + q^{16} + q^{18} + O(q^{20}) \in M_4(3, \FF_2),\\
f_5 &= \barr E_4^{5\myhyphen{\rm crit}} = q + q^{2} + q^{4} + q^{5} + q^{8} + q^{9} + q^{10} + q^{16} + q^{18} + O(q^{20})  \in M_4(5, \FF_2).
\end{align*}
\end{mylemma}

\begin{remark}\label{serresupersing}
For $p \geq 5$, the {subspace} of $M(N,\FF_p)$ coming from weights divisible by $p-1$ is the coordinate algebra {of} $X_0(N)_{\FF_p}$ with the supersingular points removed, 
  with the weight filtration corresponding to the maximal order at the poles \cite[Corollaire 2]{SerreBBK}; for $p = 2,3$ something similar is true with the filtration adjusted. For example, $M(7, \FF_2) =  \FF_2[f_7, f_7^{-1}]$ for 
$f_7 = \eta(q^7) \eta(q)^{-1}$ a Hauptmodul on $X_0(7)$. 
\end{remark}

\begin{proof}[Proof of \cref{lemma:polyalg}] In all cases, we determine $M(N, \ZZ)$ and reduce modulo $2$. 

\fbox{Case $N = 3$:} 
We claim that $M(3, \ZZ)$ is generated by $e_2 = E_{2, 3}$, the unique Eisenstein {series} in weight~$2$; $g_4 := E_4^{3\myhyphen{\rm crit}}$, the normalized semicuspidal Eisenstein eigenform in weight $4$; and $h_6$, which captures a congruence in weight $6$; with a single relation: $g_4^2 = e_2 h_6$.  More precisely, let
$$e_2 := E_{2, 3} = \frac{E_2 - 3 E_2(q^3)}{-2} =1 + 12 q + 36 q^2 + 12 q^3 + 84 q^4 + 72 q^5 + O(q^{6})\in M_2(3, \ZZ),$$
$$g_4 := E_4^{3\myhyphen{\rm crit}} = \frac{E_4 - E_4(q^3)}{240}=q + 9q^{2} + 27q^{3} + 73q^{4} + 126q^{5} + O(q^{6}) \in M_4(3, \ZZ),$$
and let $$c_6 := q - 6 q^2 + 9 q^3 + 4 q^4 + 6 q^5 + O(q^{6})  \in S_6(3, \ZZ)$$
be the unique normalized cuspform. Then $e_2 g_4$ and $c_6$ are congruent modulo $27$, so set 
$$h_6  := \frac{e_2 g_4 - c_6}{27} =q^{2} + 6q^{3} + 27q^{4} + 80q^{5} + O(q^{6}) \in M_6(3, \ZZ).$$
We claim that monomials in $e_2$, $g_4$, and $h_6$ give a Victor-Miller basis in each even weight $k \geq 0$, by which here we mean simply a set of forms $\{f_{k, 0}, \ldots, f_{k, d_k-1} \}$, where $d_k = \dim M_k(3, \CC) = 1 + \lfloor{\frac{k}{3}}\rfloor$, satisfying $f_{k, i} = q^i + O(q^{i + 1})$ for each $i$, $k$.  For $0 \leq i \leq \lfloor \frac{k}{6} \rfloor$ we define the even Victor-Miller elements: let $f_{k, 2i} := e_2^{k/2 - 3i} h_6^i$. For $0 \leq i  {\leq} \lfloor\frac{k-4}{6}\rfloor$ we define the odd Victor-Miller elements: let $f_{k, 2i +1 }  :=  e_2^{k/2 - 2 - 3i} g_4 h_6^{i}$. To see that these $\{f_{k, i}\}$ define a basis of $M_k(3, \CC)$, observe that {for even~$k$}, we have $\lfloor \frac{k}{6} \rfloor + {1}+ \lfloor\frac{k-4}{6}\rfloor + {1} = {1} + \lfloor{\frac{k}{3}}\rfloor$; the Victor-Miller shape of their $q$-expansions guarantees that they are also a basis for $M_k(3, \ZZ)$. 

The dimension formulas guarantee that the relation $g_4^2 = e_2 h_6$ is the only one. 
In other words, $$M(3, \ZZ) = \ZZ[e_2, g_4, h_6]/(g_4^2 -e_2 h_6)$$
as a graded algebra. 
Finally, modulo $2$ we have $e_2 \equiv 1$, so that our relation becomes ${h_6} = g_4^2$, so that we can conclude that $M(3) = M(3, \FF_2)$ is a polynomial algebra in $g_4$ over $\FF_2$. 
 
 \fbox{Case $N = 5$:} Very similar to $N = 3$, except that forms of weight $2$ and $4$ already generate. Again, let $e_2 := E_{2, 5} = 1 + 6q + 18q^{2} + 24q^{3} + 42q^{4} + 6q^{5} + O(q^{6})$ and $g_4 := E_4^{5\myhyphen{\rm crit}} = q + 9q^{2} + 28q^{3} + 73q^{4} + 125q^{5} + O(q^{6})$. Let $c_4 = q - 4q^{2} + 2q^{3} + 8q^{4} - 5q^{5} + O(q^{6})$ be the unique normalized cuspform of weight $4$. Since $c_4$ and~$g_4$ have a congruence modulo $13$, let $d_4: = \displaystyle \frac{g_4 - c_4}{13} = q^2 + 2 q^3 + 5 q^4 + 10q^5 + O(q^6) \in M_4(5, \ZZ)$. Using the standard dimension result $\dim M_k(5, \CC) = 1 + 2 \lfloor{\frac{k}{4}} \rfloor$, combined with a Victor-Miller--basis construction similar to above,
 one can prove that, as a graded algebra,
 $$M(5, \ZZ) = \ZZ[e_2, g_4, d_4]/(g_4^2- e_2^2 d_4 - 4 g_4 d_4 + 8 d_4^2).$$ Over $\FF_2$ we have $e_2 = 1$, so that the relation reduces to $g_4^2 = d_4$, so that $M(5, \FF_2) = \FF_2[g_4]$. 
\end{proof}

\subsection{$A(3)$ and $A(5)$ are local rings}\label{loc} By the discussion in \cref{heckealgmodpsec}, to show that $A(3)$ and $A(5)$ are local $\FF_2$-algebras, it suffices to prove the following lemma. Recall that $G = G_{\QQ, 2N}$.

\begin{mylemma}
\label{proprep}%\ \\
For $N =  3, 5$, if $\rhobar: G \to \gl_2(\barr\FF_2)$ is a $\Gamma_0(N)$-modular representation, then $\bar\rho = 1 \oplus 1$.
\end{mylemma}

\begin{proof}
It is clear that the only reducible semisimple such $\rhobar$ is $1 \oplus 1$, so it remains to prove that there are no such irreducible $\rhobar$. 
Since $M(N)$ is a polynomial algebra, one may use an elementary method of Serre, as sketched in \cite[footnote p.~398]{BK} in level $1$.  Alternatively, this fact follows from Serre reciprocity (formerly Serre's conjecture). We give a third argument that is in between these in terms of machinery involved. 

 Suppose $\rhobar: G \to \gl_2(\bar \FF_2)$ 
 $\Gamma_0(N)$-modular. We first show that the image of $\rhobar$ is a finite $2$-group. By a theorem of Tate, an irreducible $\rhobar$ must be ramified at $N$ \cite{tate}. Since the level is prime, $\rhobar|_{I_{N}} \simeq \begin{psmallmatrix} 1 & \ast \\ 0 & 1 \end{psmallmatrix}$, where $\ast$ is a nontrivial extension (\cref{levelelldecomp}). Thus $\ast$ is an additive character $I_N \to \bar \FF_2$, which must factor through the tame inertia {quotient} of $I_N$, an abelian group isomorphic to $\prod_{\ell \neq 2\ {\rm prime}} \ZZ_\ell$. The only nontrivial such $\ast$ factors through $\ZZ_2/2 \ZZ_2 = \FF_2$, so that $|\rhobar(I_{N})|=2$. 
{Following the proof of part~\eqref{tracedecomp0item} of \cref{tracedecomp0}, it follows that $\rhobar|_{D_N} \simeq \begin{psmallmatrix} 1 & \ast \\ 0 & 1 \end{psmallmatrix}$ which implies that $\rhobar(D_N)$ is an elementary abelian $2$-group.}
 Therefore, {it follows that} $\rhobar|_{G_{\QQ(\sqrt{-N})}}$ is unramified outside $\{2,\infty\}$. 
From \cite[Corollary to Theorem B]{Se} (see also \cite{MT}), it follows that $\bar\rho|_{G_{\R(\sqrt{-N})}}$ is reducible. From class field theory it follows that the maximal abelian extension of $\R(\sqrt{-N})$ unramified outside $\{2,\infty\}$ is a pro-$2$ extension of $\R(\sqrt{-N})$. 
Hence, the image $\rhobar(G)$ is a $2$-group, finite since~$\rhobar$ is continuous, as claimed. But any $2$-subgroup of $\gl_2(\bar\FF_2)$ is contained in a unipotent Borel. Therefore $\rhobar$ is reducible, hence trivial. 
\end{proof}

\begin{mycor}\label{Aloc}
$A(3)$ and $A(5)$ are pro-$2$ noetherian local rings with residue field $\FF_2$.
\end{mycor}

\subsection{Tangent spaces to $A(3,\FF_2)$ and $A(5, \FF_2)$}
For $N = 3, 5$, we can strengthen and refine \cref{hecketandim}. Let $\mm(N)$ be the maximal ideal of $A(N)$, and $\mm(N)^\red$ be the maximal ideal of $A(N)^\red$. We state the result for $N = 3$ and $N = 5$ separately, and then combine these to get a general result. 

\begin{myprop}\label{gen3}\
\begin{enumerate}[topsep=-5pt, itemsep = 3pt]
\item \label{gen3nil} $\tandim A(3) = 4$, {and} the maximal ideal {is} generated by $T_{11}$, $T_{5}$, $T_{13}$, and  $T_7$.
\item \label{gen3red} 
The maximal ideal of $A(3)^\red$ is generated by $T_{11}$ and any two of $T_5, T_7, T_{13}$. 
\item In \eqref{gen3nil} and \eqref{gen3red}, $T_{q'}$ may replace $T_{q}$ if $q'$ is a prime congruent to $q$ modulo $24$. 
\end{enumerate}
\end{myprop}

\begin{myprop}\label{gen5}\
\begin{enumerate}[topsep=-5pt, itemsep = 3pt]
\item \label{gen5nil} $\tandim A(5) = 4$, {and} the maximal ideal {is} generated by $T_{13}$, $T_{3}$, $T_{11}$ and $T_7$. 
\item \label{gen5red} 
The maximal ideal of $A(5)^\red$ is generated by $T_{13}$ and any two of $T_3, T_7, T_{11}$. 
\item In \eqref{gen5nil}, \eqref{gen5red}, $T_{q'}$ may replace $T_{q}$ if $\frob_q = \frob_{q'}$ in $\gal\big(\QQ(i, \sqrt{2}, \sqrt{5})/\QQ\big)$. 
\end{enumerate}
\end{myprop}

\begin{mycor}\label{genN} For $N = 3, 5$, we have
\begin{enumerate}[topsep=-5pt, itemsep = 3pt]
\item \label{genNnil} $\tandim A(N) = 4$, with the maximal ideal generated by $\tau(g_N^\pm)$, $\tau(g_{-N}^-)$, $\tau(g_{-N}^+)$, and $\tau(g_{7}^-)$. 
\item \label{genNred} $\mm(N)^\red$ is generated by {$\tau^{\red}(g_N^\pm)$} and any two of {$\tau^{\red}(g_{-N}^-)$, $\tau^\red(g_{-N}^+)$, and $\tau^\red(g_{7}^-)$}. 
\end{enumerate}
\end{mycor}

\cref{genN} follows from  \cref{tandim,propdef,gen3,gen5}.

\begin{proof}[Proof of \cref{gen3}]  By \cref{Aloc} we know that $A({3}) = A({3})_\rhob$ is already local. 
Write $\mm, \mm^\red$ for~$\mm(3), \mm(3)^\red$, respectively. The method of \cref{tanlemma} and \cref{propdef} will yield explicit spanning sets for $\mm/\mm^2$ and $\mm^\red/(\mm^\red)^2$; we then establish {their} linear independence in $\mm/\mm^2$ by exhibiting their action on forms explicitly. 

Recall that $G_{\QQ,{2\cdot3}}/G_{\QQ,{2\cdot3}}^2 = \gal(\QQ(\sqrt{-1}, \sqrt{2}, \sqrt{3})/\QQ) = \gal(\QQ(\mu_{24})/\QQ)$, so every element is determined by its action on the triple $v_3 = (\sqrt{-1}, \sqrt{2}, \sqrt{3})$ and may be represented by $\frob_q$ for $q$ in $\{73, 5, 7,11, 13, 17, 19, 23\}$ (or other prime representatives of their congruence classes modulo $24$). 

Let $i \in I_3$ be nontrivial, and let $d \in \barr D_3 - \bar I_3$ {be the element fixing} $\sqrt{3}$.
Then the correspondence is as follows. 
\begin{center}
\begin{tabular}{c|c|c|c|c}
\multirow{2}{*}{\small elt.} & 
\multirow{2}{*}{\small lift} & 
{\small action }&
{\small sample}&
{\small local} \\
& 
&
{\small on $v_3$} &
 {\small $\frob_q$} & 
 {\small at $3$}\\
\hline
1 & $g_1^+$&$(0,0, 0)$ & 73 & trivial\\
$i $ &$g_1^-$ & $(0, 0, 1)$ & 17 & in $\bar I_3$ \\
$d $ &$g_3^+$ & $(1, 1, 0)$ & 11 & in $\barr D_3$ \\
$id $ &$g_3^-$ & $(1, 1, 1)$ & 19 & in $\barr D_3$
\\
\end{tabular}\qquad\qquad
\begin{tabular}{c|c|c|c|c}
\multirow{2}{*}{\small elt.} & 
\multirow{2}{*}{\small lift} & 
{\small action }&
{\small sample}&
{\small local} \\
& 
&
{\small on $v_3$} &
 {\small $\frob_q$} & 
 {\small at $3$}\\
\hline
$c $ & $g_7^+$ & $(1, 0, 0)$ & 23 & $-$\\
$ci$ & $g_7^-$ &$(1, 0, 1)$ & 7& $-$\\
$cd$ & $g_5^+$ &$(0, 1, 0)$ & 13& $-$\\
$cid$ &$g_5^-$ & $(0, 1, 1) $& 5& $-$
\end{tabular}\end{center}
Comparing this chart to \cref{tandim} tells us that $\mm(3)$ is generated by, for example, $T_{11}$, $T_5$, $T_{13}$, and~$T_7$. For the generators of $\mm(3)^\red$, similarly use \cref{propdef}. 
It remains to see that $T_{11}$, $T_5$, $T_{13}$, and~$T_7$ are linearly independent in $\mm/\mm^2$. Let $\dd(q) = \sum_{2 \nmid n} q^{n^2}$ be the mod-$2$ $q$-expansion of the Ramanujan $\dd$-function, so that $\dd$ and $\dd': = \dd(q^3)$ are both forms in $M(3)$. It is straightforward to verify the following table, for example using SageMath \cite{sage}.
\begin{center}
\begin{tabular}{c|cccc}
$f$ & $T_5(f)$ & $T_7(f)$ & $T_{11}(f)$ & $T_{13}(f)$ \\
\hline
$\dd$ & $0$ & $0$ & $0$ & $0$\\
$\dd^3$ & $0$ & $0$ & $\dd$ & $0$\\
$\dd^5$ & $\dd$ & $0$ & $0$ & $\dd$ \\
$\dd'$ & $0$ & $0$ & $0$ & $0$ \\
$\dd^2 \dd'$ & $\dd$ & $\dd'$ & $0$ & $0$\\
\end{tabular}
\end{center}
From the table, it is clear that $\dd^3$, $\dd^5$ and $\dd^2 \dd'$ are three forms annihilated by $\mm^2$ but not by $\mm$. Now suppose $T = a T_{5} + b T_7 + c T_{11} + d T_{13}$ is in $\mm^2$ for some $a, b, c, d \in \FF_2$. Then $T(\dd^3) = c \dd$ and $T(\dd^5) = (a + d) \dd$ and $T(\dd^2 \dd') = a \dd + b \dd'$ are all three equal to zero, so $a = b = c = d = 0$.
\end{proof}

\begin{proof}[Proof of \cref{gen5}]
The proof for $N = 5$ is analogous; we highlight a few details. The Frattini field $K = \QQ(\sqrt{-1}, \sqrt{2}, \sqrt{5})$ has index $2$ in $\QQ(\mu_{40})$. So $\frob_q$ in $\gal(K/\QQ)$ depends on congruences modulo $40$. Table of elements in $\gal(K/\QQ)$ with local-at-$5$ subgroups and action on $v_5 =(\sqrt{-1}, \sqrt{2}, \sqrt{5})$:    

\begin{center}
\begin{tabular}{c|c|c|c|c}
\multirow{2}{*}{\small elt.} & 
\multirow{2}{*}{\small lift} & 
{\small action }&
{\small sample}&
{\small local} \\
& 
&
{\small on $v_5$} &
 {\small $\frob_q$} & 
 {\small at $5$}\\
\hline
1 & $g_1^+$& $(0, 0, 0)$ & 41, 89 & trivial \\
$i $ & $g_1^-$& $(0, 0, 1)$ &17, 73 & in $\bar I_5$ \\
$d $ & $g_5^+$& $(0, 1, 0)$ & 61, 29& in $\bar D_5$ \\
$id $ & $g_5^-$& $(0, 1, 1)$ & 13, 37 & in $\bar D_5$
\\
\end{tabular}\qquad\qquad
\begin{tabular}{c|c|c|c|c}
\multirow{2}{*}{\small elt.} & 
\multirow{2}{*}{\small lift} & 
{\small action}&
{\small sample}&
{\small local} \\
& 
&
{\small on $v_5$} &
 {\small $\frob_q$} & 
 {\small at $5$}\\
\hline
$c $ & $g_7^+$ & $(1, 0, 0)$ & 31, 79 & $-$\\
$ci$ & $g_7^-$ & $(1, 0, 1)$ & 7, 23& $-$\\
$cd$ & $g_3^+$ & $(1, 1, 0)$ &11, 19 & $-$\\
$cid$ & $g_3^-$ & $(1, 1, 1) $&3, 67 & $-$
\end{tabular}\end{center}

The forms $\dd^3$, $\dd^5$, and $\dd^2 \dd(q^5)$ again serve as witnesses for the linear independence of $\{T_3, T_7, T_{11}, T_{13}\}$ in~$\mm(5)/\mm(5)^2$. The details are left to the reader. 
\end{proof}

\section{Structure of $A(N, \FF_2)^\elleng$ for $N = 3, 5$} 
We now focus on $A(N)^\elleng$, the Hecke algebra acting faithfully on the \verynew\ forms $M(N)^{\elleng} = \ker (U_N + 1)$. In \cref{delicate} we use the nilpotence method (\cref{nilpthm}) to show that $\dim A(N)^\elleng \geq 2$. In \cref{mainsec10} we deduce that $A(N)^\elleng$ is a complete regular local $\FF_2$-algebra of dimension $2$ (\cref{Aellengred}) and deduce that  $\tandim A(N)^\red = 3$ (\cref{tandimAred}).  

Recall that $U_N' = U_N + 1$. Let $\mm^\new$, $\mm^\pf$ be the maximal ideals of $A(N)^\new$, $A(N)^\pf$, respectively.

\subsection{Lower bound on $\dim A(N)^\elleng$ by the nilpotence method} \label{delicate}
\begin{mythm}\label{nilpy}
For $N = 3, 5$, the Krull dimension of $A(N)^\elleng$ is at least $2$.
\end{mythm}

\begin{proof} 
Recall from \eqref{anewpf} and \eqref{avnewpf} the presentations of $A(N)^\new$ and $A(N)^\elleng$ as quotients of $A(N)^\pf$: to wit, $A(N)^\new = A(N)^\pf/(U_N')^2$ and $A(N)^\elleng = A(N)^\pf/(U_N')$. 
It is thus clear that~$U_N'$ is nilpotent in~$A(N)^\new$, so that $\dim A(N)^\new = \dim A(N)^\elleng$. It therefore suffices to show that $\dim A(N)^\new \geq 2$.
 We prove this with the nilpotence method, by showing that the four conditions of \cref{nilpthm} are satisfied for~$A := A(N)^\new$. 

\underline{Condition \eqref{power}:} By~\cref{lemma:polyalg}, $M(N) = \FF_2[f]$ for a form $f = f_N$. 

\underline{Condition \eqref{two}:} Set $K := K(N)^\new$ as in \eqref{knewdef}, by \eqref{dualJ} in duality with $A = A(N)^\pf/(U_N')^2$.

\underline{Condition \eqref{max}:} The Hecke operators $U_N'$ as well as every $T_\ell$ for $\ell\nmid 2N$ act locally nilpotently on $M(N)$ and are therefore in $\mm^\pf$. By \cite[Proposition 6.2]{M} for general level, generalizing \cite[{\thm}~3.1]{NS1}, given any prime $\ell \nmid 2N$, the sequence $\{T_\ell(f^n)\}_n$ satisfies a linear recurrence over $M(N)$ whose companion polynomial $P_{\ell, f} = X^{\ell + 1} + a_1 X^\ell + \cdots +a_{\ell + 1}$ has coefficients $a_j \in M(N)$ with $\deg_f a_j \leq j$ and $\deg_f a_{\ell + 1} = \ell + 1$. Equivalently, both the $f$-degree and the total degree (as a polynomial in $X$ and $f$) of $P_{\ell, f}$ coincide with its $X$-degree. By a similar argument, $\{U_N(f^n)\}_n$ satisfies an $M(N)$-linear recurrence whose polynomial $P_{N, f}$ also has its $f$-degree and total degree coincide with its $X$-degree~$N$. Therefore $\{U_N' (f^n)\} = \{U_N (f^n) + f^n\}$ will satisfy an $M(N)$-linear recurrence with characteristic polynomial $P_{N, f}' = (X - f)P_{N,f}$; the degree constraints of $P_{N, f}'$ follow from those of $P_{N, f}$. 

For completeness, we include below both $P_{N, f}$ and $P_{\ell, f}$ for $T_\ell$ that generate $\mm^\pf$ (\cref{pflem}) and hence its quotient $\mm^\new$.
For $N = 3$ we may take $\ell = 13$ and $\ell = 7$; for $N = 5$ we take $\ell = 11$ and $\ell = 7$: see the element tables in the proofs of \cref{gen3,gen5}.
\vspace{-10pt}
\renewcommand{\arraystretch}{1.2}
\begin{center}
\begin{tabular}{c|c|l}
$N = 3$ & $\ell$ & $P_{\ell, y}$ for $y = f_3$\\
\hline\hline
& 3 & $X^3 + yX^2 + (y^2 + y)X + y^3 + y$\\
\cline{2-3}
 &7 &  $X^8 + (y^2 + y)(X^4 + X^3) + (y^4 + y^3 + y^2 + y)(X^2  + X) + y^8$\\
\cline{2-3}
 &13 &  $X^{14} + y^2X^{12} + y^4X^{10} + y^6X^8 + (y^8 + y^4 + y^2)X^6 + (y^{10} + y^6 + y^2)X^4$\\ 
 &&\quad ${} + (y^{12} + y^6 + y^4 + y^2 + y)X^2 + (y^2 + y)X + y^{14}$\\
\hline\hline
$N =5$ & $\ell$ & $P_{\ell, y}$ for $y = f_5$\\
\hline\hline
& 5 & $X^5 + yX^4 + (y^2 + y)X^3 + (y^3 + y)X^2 + (y^4 + y^3 + y^2 + y)X + y^5 + y$
\\
\cline{2-3}
 &7 &  $X^8 + (y^2 + y)X^6 + (y^2 + y)X^5 + (y^6 + y^5 + y^2 + y)(X^2 + X) + y^8$
\\
\cline{2-3}
 &11 &  $X^{12} + y^2X^8 + (y^4 + y^2)X^6 + y^6X^4 + (y^8 + y^6 + y^2 + y)X^2 + (y^2 + y)X + y^{12}$
 \end{tabular}
\end{center}

\underline{Condition \eqref{four}:} We seek a linearly independent sequence of forms $\{g_n\}$ in $K(N)^\new$ with the property that~$\deg_f g_n$ grows no faster than linearly in $n$. 
By \eqref{weightfact} we may replace $\deg_f(g_n)$ in this estimate with its $\Gamma_1(N)$--weight filtration $w_1(g_n)$. 

Write $f = f_N$ and let $u = w_1(f_N)$; by \S\ref{weightsec} we know that $u= 3$ ($N = 3$) or $u = 2$ ($N = 5$). 

For $n$ odd, consider the $\FF_2$-vector space $W_n$ spanned by the forms $\theta(f)$, $\theta(f^3), \ldots,$ $\theta(f^n)$ inside $M(N) \subset \FF_2\lb q \rb$. On one hand, by \eqref{imtheta} we have $W_n \subseteq K(N)$. On the other hand, by \eqref{weightfact} for each $i$, we have  
$$w_1\big(\theta(f^i)\big) \leq w_1(f^i) + 3 = u i + 3.$$ 
Therefore \eqref{filt} implies that $W_n \subseteq K(N)_{u n + 6}$. Moreover, we claim that 
\begin{equation}\label{dimw} 
\dim W_n = \frac{n-1}{2}.
\end{equation} Indeed, $f = q + O(q^2)$, so that for odd $i$ we have $\theta(f^i) = q^i + O(q^{i+1})$. Thus $\theta(f), \theta(f^3), \ldots, \theta(f^n)$ are linearly independent, and there are $\frac{n-1}{2}$ of them. 

Now consider the image of $W_n$ under the operator $(U_N')^2$. By \cref{s2killsnew}, we have $(U_N')^2 W_n \subset M(N)^\old$. More precisely, since $\theta$ commutes with $U_N'$, we have  
 \begin{equation}\label{keyeq} 
 (U_N')^2 W_n \subseteq \theta \big(M_{u n + 6}(N)^\old\big),
 \end{equation} where we write 
 \begin{equation}\label{olddef} M_k(N)^\old := M_k(1) + W_N M_k(1).
 \end{equation} 

We now analyze $\dim \theta \big( M_k(N)^\old\big)$ for any even $k \geq 0$. Standard dimension formulas (for example, \cite[Theorem~3.5.1]{diamondshurman}) tell us that $\dim M_k(1) = \frac{k}{12} + O(1)$. Since $W_N$ is an involution and the sum in \eqref{olddef} is direct on cuspforms \cite[Proposition~5.4]{deomed}, we conclude that for any even $k \geq 0$, we have 
\begin{equation}\label{dimMold} 
\dim M_k(N)^\old =  \frac{k}{6} + O(1).
\end{equation}

On the other hand, by \eqref{kertheta}, \eqref{weightfact}, and \eqref{filt}, the kernel of $\theta$ on $M_k(N)^\old$ certainly contains the squares of forms in $M_{\lfloor \frac{k}{2} \rfloor}(N)^\old$. Thus by \eqref{dimMold} we obtain
\begin{equation}\label{dimkermold}
\dim \ker\big( \theta \big| M_k(N)^\old \big) \geq \frac{k}{12} + O(1),
\end{equation}
so that combining \eqref{dimMold} and \eqref{dimkermold} gives
\begin{equation}\label{dimimmold} 
\dim  \theta\big( M_k(N)^\old\big) \leq \frac{k}{12} + O(1).
\end{equation}

Finally, we return to \eqref{keyeq}. By \eqref{dimw}, the operator $(U_N')^2$ maps a space of dimension  $\frac{n-1}{2}$ to a space whose dimension, by \eqref{dimimmold}, grows no faster than $\frac{un}{12} + O(1)$. 
Therefore, 
$$\dim \ker\big( (U_N')^2 \big| W_n \big) \geq \frac{(6-u)n}{12} + O(1).$$ Since $u \leq 3$, 
we can certainly choose $g_n \in \ker\big( (U_N')^2 \big| W_n \big) \subset K(N)^\new$, at least for $n \gg 0$, so that the sequence $\{g_n\}$ is linearly independent. Moreover, $w_1(g_n) \leq u n$, as required.   

Finally, \cref{nilpthm} allows us to conclude that $\dim A(N)^\elleng = \dim A(N)^\new \geq 2$, as desired. 
\end{proof}

\subsection{Main result and corollaries}\label{mainsec10}

\begin{mycor}\label{Aellengred}
For $N = 3, 5$, $A(N)^{\elleng}$ is a complete regular local $\FF_2$-algebra of dimension $2$. 
\end{mycor}
The proof is immediate from \cref{nilpy,englem}. To give a more precise statement, we endow~$\FF_2\lb y, z \rb$ with a grading by $(\ZZ/8\ZZ)^\times/\langle N \rangle \simeq \{ \pm 1\}$ by giving $y$ and $z$ both the grading $-1$. Also write~$\tau^{\elleng}$ for the pseudorepresentation  $G_{\QQ,2N} \to A(N)^\elleng_{\one}$ coming from  $\tau^\hecke$.

\begin{mycor}\label{Avnewstruct} 
For $N = 3, 5$, the map $y \mapsto {\tau^\elleng}(g^\pm_{-N})$ and $z \mapsto {\tau^\elleng}(g_7^-)$ gives a graded isomorphism $\FF_2 \lb y, z \rb \simeq A(N)^\elleng$ of $(\ZZ/8\ZZ)^\times/\langle N \rangle$-graded $\FF_2$-algebras. In particular: 
\begin{itemize}[itemsep = 5pt, topsep = -5pt]
\item The map $\FF_2\lb y, z \rb \to A(3)^{\elleng}$ given by  $y \mapsto T_{13}$ and $z \mapsto T_7$ is an isomorphism of $\gal(\QQ(\sqrt{-2})/\QQ)$-graded algebras. 
\item The map $\FF_2\lb y, z\rb \to A(5)^{\elleng}$ given by $y \mapsto T_{11}$ and  $z \mapsto T_7$ is an isomorphism of $\gal(\QQ(i)/\QQ)$-graded algebras. 
\end{itemize}
\end{mycor}

\begin{proof}
\cref{Aellengred,englem}. For explicit generators, see \cref{gen3,gen5}.
\end{proof}

\begin{mycor}\label{tandimAred}
For $N = 3, 5$, we have $\tandim A(N)^\red = 3$. 
\end{mycor}

\begin{proof} 
On one hand, by  \cref{level1} the level-$1$ Hecke algebra $A(1)$ is a dimension-$2$ domain, so that the quotient map $A(N) \onto A(1)$  factors through $A(N)^\red$. \cref{Aellengred} and the same logic imply that the map~$A(N) \onto A(N)^\elleng$ factors through $A(N)^\red$ as well. On the other hand, \cref{gen3,gen5} tell us that $\tandim A(N)^\red$ is at most $3$. Therefore it suffices to prove, for example, that the kernel $J_1$ of the  induced quotient map $A(N)^\red \onto A(1)$ is nonzero. We show this by proving that $J_1 + J_\elleng \neq J_\elleng$, where~$J_\elleng$ is the kernel of the induced map $A(N)^\red \onto A(N)^\elleng$.

Since $A(N)^\elleng = A(N)^\pf/(U_N')$, the relation in \eqref{fun2} implies that $J_\elleng$, when projected to $A(1)$, contains~$T_N = \tau^\hecke_{2, 1}(\frob_N)$, one of the two generators of the maximal ideal of $A(1)$.  Therefore $A(N)^\red/(J_\elleng + J_1)$ is a quotient of $A(1)/(T_N)$, so has Krull dimension at most $1$. Since $A(N)^\red/J_\elleng = A(N)^\elleng$ has Krull dimension~$2$, we must have $J_\elleng + J_1 \neq J_\elleng$, so that $J_1 \neq 0$, as desired.  
\end{proof}  

\section{Structure of $\RRR(N, \FF_2)^\red$ and $A(N, \FF_2)^\red$ for $N=3,5$}
We are now ready to state and prove one of the main results of this article where for $N = 3, 5$ we describe the structure of $A(N)^\red := A(N, \FF_2)^\red$ precisely, by proving that it is isomorphic to $\RRR(N)^\red := \RRR(N, \FF_2)^\red$. In other words, we prove \cref{a1}.

\begin{myprop}\label{minprimes} For $N = 3, 5$, $\RRR(N)$ has two minimal prime ideals, and the quotient by each of them factors through $\phi: \RRR(N) \onto A(N)$.  
\end{myprop}
\begin{proof} 
By \cref{krulldimR}, the Krull dimension of $\RRR(N)$ is $2$. On one hand, its quotient $\RRR(1)$ is isomorphic to~$A(1)$, which is {a domain of Krull dimension $2$} (\cref{level1}). Therefore the kernel $\pp_\old: = J_{N, 1}$ of the map~$\psi_{N, 1}: \RRR(N) \onto \RRR(1)$ from \eqref{psiN1} is a minimal prime ideal of $\RRR(N)$. Moreover, $\psi_{N, 1}$ factors through~$A(N)$, since again its target is isomorphic to $A(1)$. On the other hand, $A(N)^{\elleng}$ is also a quotient, via $A(N)$, of~$\RRR(N)$; and for $N = 3, 5$ we know that $A(N)^{\elleng}$ is also {a domain of Krull dimension $2$} (\cref{Aellengred}). Therefore the kernel $\pp_\elleng$ of the corresponding surjection $\RRR(N) \onto A(N)^{\elleng}$ is also a minimal prime of $\RRR(N)$. 
\begin{equation}
\begin{tikzpicture}%[node distance=0.5cm]
\node(RN)	{$\RRR(N)$};
\node(AN)[right = 1.5cm of RN]{$A(N)$};
\node(A1)[above right = 0.3cm and 1.5cm  of AN]{$A(1)$};
\node(Anew)[below right =  0.3cm and 1.5cm of AN]{$A(N)^\elleng$};
\node(R1)[left =10pt of A1]{$\RRR(1)$};
\node(Anewring)[right = 10pt of Anew]{$\FF_2\lb y, z \rb$};
\node(A1ring)[right = 20pt of A1]{$\FF_2\lb x, y \rb$};
\draw(RN)[->>] -- (AN) node[near end, above] {$\scriptstyle \phi$};
\draw(AN)[->>] -- (A1); 
\draw(AN)[->>] -- (Anew);
\draw(RN)[->>] -- (Anew) node[midway, below] {$\scriptstyle \phi^\elleng$};
\draw(RN)[->>] -- (R1) node[midway, above] {$\scriptstyle \psi_{N, 1}$};
\draw(R1)[->] -- (A1) node[midway, above = -0.1cm]{$\sim$};
\draw(A1)[->] -- (A1ring) node[midway, above = -0.1cm]{$\sim$};
\draw(A1)[->] -- (A1ring) node[midway, above = -0.1cm]{$\sim$};
\draw(Anew)[->] -- (Anewring) node[midway, above = -0.1cm]{$\sim$};
\end{tikzpicture}
\end{equation}

We claim that $\pp_\old \neq \pp_\elleng$. Indeed, any element of $\RRR(N)$ of the form $\tau^\univ(\frob_N)$ must be in $\pp_\elleng$, since any such element maps to $F_N = \tau^\hecke(\frob_N)$ in $A(N)$ and then to zero in $A(N)^{\elleng}$ by \cref{fnto0}. At the same time the image of $F_N$ in $A(1)$ is $T_N$, which for $N = 3, 5$ modulo $8$ is a generator of the maximal ideal, and is in any case nonzero, so that none of the {elements of the form} $\tau^\univ(\frob_N)$ are in $\pp_\elleng$. Furthermore,~$\pp_\old \not\subset \pp_\elleng$, since the two have the same depth.

We now claim that every prime ideal of $\RRR(N)$ contains either $\pp_\old$ or $\pp_\elleng$, so that these are the only minimal primes. 
Let $J$ be the ideal of $\RRR(N)$ generated by elements of the form $\tau^\univ(\frob_N)$ and the nilradical of~$\RRR(N)$, so that $J \subseteq \pp_\elleng$. We show that $J = \pp_\elleng$. Since $J$ contains the nilradical, $\RRR(N)/J$ is a quotient of $\RRR(N)^\red$. Further, the image of $\tau^\univ(\frob_N)$ is a generator of the maximal ideal of $\RRR(N)^\red$ (in other words, {the image of} $\tau^\univ(\frob_N)$ is nonzero in $\mm(N)^\red/(\mm(N)^\red)^2$), because this is true in its quotient~$A(1)$. Therefore, 
$$\tandim \RRR(N)/J \leq \tandim \RRR(N)^\red - 1 \leq 2,$$
where the last inequality follows from \cref{propdef}. 
Therefore, $\RRR(N)/J$ is a quotient of a {complete regular local $\FF_2$-algebra of Krull dimension $2$}, which has $A(N)^\elleng \cong \FF_2\lb y, z\rb$ as a quotient. In other words, $\RRR(N)/J \cong A(N)^\elleng$ and $J = \pp_\elleng$. Finally, by \cref{tanlemma}, any prime ideal $\pp$ of $\RRR(N)$ that does not contain $\pp_\old$ contains the set $\tau^\univ(D_N)$ and hence, in particular $\tau^\univ(\frob_N)$. Therefore, any such $\pp$ contains $J = \pp_\elleng$, as claimed. 
\end{proof} 

\begin{mycor}\label{isomcor} The surjections $\hat \RRR(N) \onto \RRR(N) \stackrel{\phi}\onto A(N)$ induce isomorphisms
$$\hat\RRR(N)^\red \simeq \RRR(N)^\red \simeq A(N)^\red.$$
\end{mycor} 
\begin{proof}
By \cref{minprimes}, {all} the minimal primes of $\RRR(N)$ {contain} $\ker \phi$, so that they also correspond to the minimal primes of $A(N)$. {Hence, that $\ker(\phi)$ is nilpotent.} By \cref{redlemma}, the primes of $\hat \RRR(N)$ are all contractions 
of primes of $\RRR(N)$. 
\end{proof}

\begin{mythm}
\label{mainthm}
For $N = 3, 5$, the reduced Hecke algebra $A(N)^{\red}$ is isomorphic to $\displaystyle \frac{\FF_2\lb a, b, c\rb}{\big(ab\big)},$ with 
$$a \in {\tau^\red}(g_N^\pm) + (\mm(N)^\red)^2,\qquad b \in {\tau^\red}(g_7^-) + (\mm(N)^\red)^2, \qquad\mbox{and}\qquad c = {\tau^\red}(g_{-N}^\pm).$$
\begin{enumerate}[itemsep = 8pt]
\item For $N = 3$, let $f \in \FF_2\lb x, y\rb$ be the power series satisfying $f(T_{11}, T_{13}) = T_7$ in $A(1)$, and let $g \in \FF_2\lb y, z\rb$  be the power series satisfying $g(T_{13}, T_{7}) = T_{11}$ in $A(N)^\elleng$: that is, 
\begin{align*} 
f&= xy + x^3 y + x y^5 + x^7y + x^5y^3 + x^3 y^5 + x^9y + x^5y^5 + x^3y^7 + xy^9 + xyO\big((x^2, y^2)^5\big)\\
 \mbox{and} \quad g &=yz + z^2 + yz^3 + y^2 z^2 + y^3 z + y^3 z^3 + y^2 z^4 + O\big(y^4, (y,z)^8\big). \quad \mbox{Then the map}
\end{align*}
$$\frac{\FF_2\lb x, y, z\rb}{\big(z - f(x, y)\big)\big(x - g(y, z)\big)} \longrightarrow A(N)^\red \qquad \mbox{induced by}
\qquad {x \mapsto T_{11}, y \mapsto T_{13}, z \mapsto T_{7}}$$
is an isomorphism.  
\item For $N = 5$, let $f' \in \FF_2\lb x, y\rb$ be the power series satisfying $f'(T_{13}, T_{11}) = T_7$ in $A(1)$, and let $g' \in \FF_2\lb y, z\rb$  be the power series satisfying $g'(T_{11}, T_{7}) = T_{13}$ in $A(N)^\elleng$.
Then the map 
$$\displaystyle \frac{\FF_2\lb x, y, z\rb}{\big(z - f'(x, y)\big)\big(x - g'(y, z)\big)} \to A(N)^\red \quad\mbox{ induced by }\quad
x \mapsto T_{13}, y \mapsto T_{11}, z \mapsto T_{7}$$
is an isomorphism.  

\end{enumerate}
\end{mythm}

\begin{proof}

Following the notation of the proof of \cref{minprimes}, write $\pp_\old^\red$ and $\pp_\elleng^\red$ for the images of $\pp_\old$ and $\pp_\elleng$ respectively, in $A(N)^\red$. These are the two minimal primes of the reduced ring $A(N)^\red$, so that in particular $\pp_\old^\red \cap \pp_\elleng^\red = \big(0 \big).$ We know from \cref{tandimAred} that $\tandim A(N)^\red=3$; moreover {by} putting \cref{genN}, \cref{level1}, and \cref{Avnewstruct} together, we know that can choose an ordered triple of generators $(X, Y, Z)$ of $\mm(N)^\red$ so that the first two generate $\mm(1)$ and the second and third generate~$\mm(N)^\elleng$: namely $X = \tau(g_N^\pm)$, $Y = \tau(g_{-N}^\pm)$, and $Z = \tau(g_7^-)$. In particular for $N=3$ we can choose $(T_{11}, T_{13}, T_7)$; for $N = 5$ we choose $(T_{13}, T_{11}, T_7)$. Then we can find a unique power series $f \in \FF_2\lb x, y \rb$ so that $Z = f(X, Y)$ in $A(1)$; more precisely, $f$ is in $\mm_{\FF_2\lb x, y \rb}^2$ because $Z$ maps to $0$ in the cotangent space of~$A(1)$. Similarly, there is a unique power series $g \in \mm_{\FF_2\lb y, z \rb}^2$ with $X = g(Y,Z)$ in~$A(N)^\elleng$.

Now consider the map  
$\alpha: \FF_2 \lb x, y, z \rb \onto A(N)^\red$
given by $x \mapsto X, y \mapsto Y,  z \mapsto Z$. By construction, $z - f(x, y) \in \pp_\old^\red$ and $x - g(y, z) \in \pp_\elleng^\red$. Since furthermore, $$\frac{\FF_2\lb x, y, z \rb}{\big(z - f(x, y)\big)} \simeq \FF_2 \lb X, Y\rb = A(1),$$ we have $\alpha^{-1}(\pp_\old^\red) = \big(z - f(x, y)\big)$. Similarly, $\alpha^{-1}(\pp_\elleng^\red) = \big(x - g(y, z)\big)$. 
By considering degrees, it is clear that the elements $z - f(x, y)$ and $x - g(y, z)$ of the UFD $\FF_2\lb x, y, z \rb$ have no common factors --- they are nonassociate irreducibles --- so that the intersection and the product of the ideals they generate agree: 
$$\ker \alpha = \alpha^{-1}(\pp_\old^\red \cap \pp_\elleng^\red) = \alpha^{-1}(\pp_\old^\red) \cap \alpha^{-1}(\pp_\elleng^\red) = \big(z - f(x, y)\big)\big(x - g(y, z)\big).$$
Since {the images of} $x - g(y, z)$, $y$, and $z - f(x, y)$ form a basis {of} the cotangent space of $\FF_2 \lb x, y ,z \rb$, the reduced Hecke algebra has the structure $\FF_2\lb a, b, c \rb/(ab)$ as claimed. The computation of $f$ and $g$ in the case~$N = 3$ and $(X, Y, Z) = (T_{11}, T_{13}, T_7)$ proceeds by linear algebra by constructing a basis of $K(1)$ (respectively, $K(3)^\elleng$) dual to the monomials in $x$ and $y$ of $A(1)$ (respectively, $y$ and $z$ of $A(3)^\elleng$) 
\end{proof}

\begin{remark} \label{thing} The obstacle to extending \cref{mainthm} from $N = 3, 5$ to all primes \mbox{$N \equiv 3,5 \pmod{8}$} is the lower bound on the {Krull dimension of} $A(N)_\one^\elleng$ from \cref{nilpy}, which implies that $A(N)_\one^\elleng \simeq \FF_2 \lb y, z \rb$ (\cref{Aellengred}). The limitation is twofold.
\begin{itemize}[topsep = -3pt, itemsep = 5pt]
\item Condition~\eqref{power} of the nilpotence method (\cref{nilpthm}) requires that the ambient space of modular forms be a polynomial algebra over a finite field. Using  Riemann-Hurwitz and the ideas described in \cref{serresupersing}, one can show that 
this happens if and only if $X_0(2N)$ has genus $0$. This a technical, not a conceptual, limitation to the nilpotence method, but we do not know of any workarounds at the moment.
\item Condition~\eqref{four} of the nilpotence method requires as input an infinite sequence of forms in  $K(N)_\one^\new$ whose filtration grows at most linearly. In the proof of \cref{nilpy}, we essentially obtain such a sequence from dimension formulas for $M_k(1) = M_k(1)_\one$ and $M_k(N) = M_k(N)_\one$, from which we can deduce information about $\dim M_k(N)_\one^\new$. But in general, standard dimension formulas are insufficient: they must be refined for various~$\rhobar$.

There are several  known $\rhobar$-dimension formula techniques: Bergdall-Pollack \cite[Proposition 6.9(a)]{bergdallpollack} uses Ash-Stevens, which requires $p \geq 5$; {Jochnowitz \cite[Lemma 6.4]{JochCong}, uses $\theta$ twists  for $p \geq 5$ and the Eichler-Selberg trace formula, but is limited by dimension $< p$ in characteristic $p$;} Anni-Ghitza-Medvedovsky (forthcoming, currently also only for $p \geq 5$) refines Jochnowitz's trace formula work with deeper congruences; see also \cite{AGM}. All of these rely on propagation from low weight and none of them yet allow for $p=2$. The low-weight piece in our context is done: since the level-raising condition for $\one$ modulo~$2$ is satisfied at every odd $N$, \cite[Theorem 2(2b)]{deomed} guarantees that $K_k(N)_\one^\new$ is nonempty for \emph{some} weight $k$. But the propagation part for $p =2$ awaits further development. 
\end{itemize}
Despite these limitations, one would not be surprised to discover that $\RRR(N)_\one^\red \cong A(N)_\one^\red \cong \FF_2\lb a, b, c\rb/(ab)$ for all primes $N \equiv 3, 5 \cmod{8}$. 
\end{remark}

\section{Structure of $A(N, \FF_2)^\pf$ for $N=3,5$}
We now determine the structure of $A(N)^\pf$ for $N=3,5$. 
{Let $\tau^\pf : G_{\QQ,2N} \to A(N)_{\one}^\pf$ be the pseudorepresentation obtained by composing $\tau^\hecke$ with the natural injection $A(N)_{\one} \to A(N)_{\one}^\pf$.} In this section we write $\tau$ for $\tau^\pf$. Recall that $c \in G_{\QQ, 2N}$ is a complex conjugation. 
\begin{mythm}
\label{pfprop}
Let $N = 3, 5$, and let $Y = \tau(g_{-N}^\pm)$ and {$Z = \tau(ci)$} for any choice of $g_{-N}^\pm$, $c$ and $i \in I_N$. 
Then the map 
$$\frac{\FF_2\lb y, z, u \rb}{(zu^2)} \onto A(N)^\pf \qquad \mbox{induced by} \qquad { u \mapsto U_N', \qquad y \mapsto Y \qquad z \mapsto Z}$$
is an isomorphism. {If $N = 3$ then $(Z, Y)$ may be $(T_{13}, T_7)$; if $N = 5$  then $(Z, Y)$ may be 
$(T_{11}, T_7)$.} 
\end{mythm}

The proof has a number of steps. We first determine the structure of $A(N)^{\pf,\old}$. Recall from \cref{pfoldnew} that $M(N)^\old = M(1) +W_NM(1)$ is a subspace of $M(N)$  stable under the action of both $A(N)$ and $U_N$, and $A(N)^{\pf,\old}$ is the largest quotient of $A(N)^\pf$ {acting faithfully on $M(N)^\old$}. 
\begin{mylemma}
\label{powerlem}
For any prime $N \equiv 3,5 \cmod{8}$ we have $A(N)_\one^{\pf,\old} \simeq \FF_2 \lb U_N' , Y \rb$.\\
 More precisely, given any $g_{-N}^\pm$, the map  $\FF_2 \lb u,y \rb \to A(N)_\one^{\pf, \old}$ given by $u \mapsto U_N'$ and $y \mapsto \tau(g_{-N}^\pm)$ is an isomorphism. In particular, 
 $$A(3)^{\pf, \old} \cong \FF_2\lb U_3', T_{13}\rb \qquad \mbox{ and } \qquad A(5)^{\pf, \old} \cong \FF_2 \lb U_5', T_{11} \rb.$$
\end{mylemma}
\begin{proof}
Recall that the Hecke algebra $A(N)$ sits inside $A(N)^\pf$ as a closed subalgebra, with the latter finite over the former (\cref{Apfsec,apffinite}). Moreover, the image of {$A(N)_{\one}$} under the natural surjection $A(N)_\one^\pf \onto A(N)_\one^{\pf, \old}$ is $A(N)_\one^\old \cong A(1)$. Therefore $A(N)_\one^{\pf, \old}$ is finite over $A(1)$, so $\dim A(N)_\one^{\pf, \old} = \dim A(1) = 2$ (\cref{level1}). We now look at generators closely: the maximal ideal of~$A(1)$ is generated by the images $F_N$ and $Y$ of $\tau(\frob_N)$ and $\tau(g_{-N}^\pm)$, respectively.
Therefore, the finiteness of~$A(N)_\one^{\pf, \old}$ over $A(1)$ implies that the maximal ideal of $A(N)_\one^{\pf,\old}$ is generated by $F_N$, $Y$, and~$U_N' = U_N +1$. Since, by \eqref{fun2}, the image of $F_N$ in the cotangent space of $A(N)_\one^{\pf,\old}$ is $0$, the maximal ideal of~$A(N)_\one^{\pf, \old}$ is generated by $Y$ and $U_N'$. As $\dim A(N)_\one^{\pf, \old} = 2$, this is a minimal generating set, and the map
$$\FF_2 \lb y, u \rb \to A(N)^{\pf, \old}_\one$$
given by $y\mapsto Y = {\text{image of }\tau(g^{\pm}_{-N})}$ and $u \mapsto U_N'$ is an isomorphism. 
\end{proof}

\begin{mylemma}
\label{partlem}\label{genlem}
For $N = 3,5$ we have $\tandim A(N)^\pf =3$. The maximal ideal is generated by $U_N'$ and any choice of ${\tau}(g_{-N}^{\pm})$ and ${\tau}(g_7^-)$. 
For example, the maximal ideal of $A(3)^\pf$ is generated by $T_{13}$, $T_7$,  and $U_3'$, and the maximal ideal of $A(5)^\pf$ is generated by $T_{11}$, $T_7$, and $U_5'$.
\end{mylemma}

\begin{proof}
On one hand, $\tandim A(N)^\pf \leq 3$ (\cref{pflem}). On the other hand, \cref{powerlem} implies that $\dim A(N)^\pf \geq 2$, so that $\tandim A(N)^\pf \geq 2$ as well. If $\tandim A(N)^\pf =2$, then $A(N)^\pf$ is a regular local ring of dimension $2$ and hence isomorphic to its quotient $A(N)^{\pf, \old}$. But that would imply that the annihilator of $M(N)^\old$ in $A(N)^\pf$ is trivial, which would mean that the same is true for the annihilator $\pp_\old$ of $M(N)^\old$ in $A(N)$, which in turn would mean that $A(N)\simeq A(1)$. But that is absurd --- for example, \cref{minprimes} shows that $A(N)$ has two {distinct} minimal primes, so that $\pp_\old \neq (0)$. 
The precise {statement about} generators follows from the analysis in \cref{pflem}.
\end{proof}

We are now ready to prove the structure theorem for $A(N)^\pf$ (\cref{pfprop}).

\begin{proof}[Proof of \cref{pfprop}]
From \cref{genlem} we know that $Y = {\tau}(g_{-N}^\pm)$, $Z = {\tau(ci)}$ and $U_N'$ generate the maximal ideal of $A(N)^\pf$. Let $\beta: \FF_2 \lb y, z, u \rb \onto A(N)^\pf$ be the map sending $y \mapsto Y$, $z \mapsto Z$, and $u \mapsto U_N'$.

\underline{{\bf Claim 1:} $\ker \beta \subseteq (zu)$:} {Note that $\tau^\hecke_{2, 1}(ci) = \tau^\hecke_{2, 1}(c)=0$ for any complex conjugation $c$ and $i \in I_N$. Here $\tau^\hecke_{2,1} : G_{\QQ,2} \to A(1)$ is the level-one modular pseudorepresentation.}\label{there}
Hence, $\beta(z)$ annihilates $M(N)^\old$.
Consider the natural surjections $A(N)^\pf \onto A(N)^{\pf, \old}$ and $A(N)^\pf \onto A(N)^\elleng$. Precomposed with $\beta$ and the isomorphism from \cref{powerlem}, the former becomes the quotient map $\FF_2\lb y, z, u \rb \onto \FF_2 \lb y, u\rb$, with kernel~$(z)$.
Indeed, from the observation above, we see that $z$ is in the kernel and \cref{powerlem} then implies that the kernel is~$(z)$. 
Similarly, precomposed with $\beta$ and the isomorphism from \cref{Aellengred} the latter becomes the quotient map $\FF_2\lb y, z, u \rb \onto \FF_2 \lb y, z\rb$, with kernel $(u)$. 
\begin{equation}\label{Apfstructure}
\begin{tikzpicture}%[node distance=0.5cm]
\node(POL)	{$\FF_2 \lb y, z, u \rb$};
\node(Apf)[right = 2cm of POL]{$A(N)^\pf$};
\node(Apf1)[above right = 0.3cm and 1.5cm  of Apf]{$A(N)^{\pf, \old} $};
\node(Anew)[below right =  0.3cm and 1.5cm of Apf]{$A(N)^\elleng$};
\node(POLApf1)[left =10pt of Apf1]{$\FF_2\lb y, u \rb$};
\node(POLAnew)[left = 10pt of Anew]{$\FF_2\lb y, z \rb$};
\draw(POL)[->>] -- (Apf) node[midway, above = -0.1cm] {$\scriptstyle \beta$};
\draw(Apf)[->>] -- (Apf1);
\draw(Apf)[->>] -- (Anew) node[pos = 0.9, above] {$\scriptstyle \ker = (U_N')$};
\draw(POL)[->>] -- (POLApf1) node[midway, above] {$\scriptstyle \ker = (z)$};
\draw(POL)[->>] -- (POLAnew) node[midway, below] {$\scriptstyle \ker = (u)$};
\draw(POLApf1)[->] -- (Apf1) node[midway, above = -0.1cm]{$\sim$};
\draw(POLAnew)[->] -- (Anew) node[midway, above = -0.1cm]{$\sim$};
\end{tikzpicture}
\end{equation}
Since both the maps $\FF_2\lb y,z,u\rb\onto A(N)^{\pf,\old}$ and $\FF_2\lb y,z,u\rb \onto A(N)^{\elleng}$ factor through $\beta$, its kernel $\ker(\beta)$ is contained in $(z) \cap (u) = (zu).$

\underline{{\bf Claim 2:} $\ker \beta \supseteq (zu^2)$:} Since $\beta(u^2) M(N) = (U_N')^2 M(N) \subseteq M(N)^\old$ (\cref{s2killsnew}) and $\beta(z)$ annihilates $M(N)^\old$, the operator $\beta(zu^2)$ annihilates $M(N)$. 

\underline{{\bf Claim 3:} $\ker \beta = (zu^2)$:} By Claim 1 above, any element of $\ker \beta$ is of the form $zut$ for some $t \in \FF_2 \lb y, z, u\rb$. Suppose $t \not\in (u)$. Then $\beta(zt)$ does not annihilate $M(N)^\elleng$, so that {by \eqref{avnewfaith} there is an $f \in K(N)^\elleng$ with $\beta(zt) (f) \neq 0$.} By \cref{shaunakveryclever} below, $f$ is in the image of~$U_N'$. This means that if $U_N'(g) = f$ for some~$g \in K(N)$, then $\beta(zut) (g) \neq 0$. Hence $zut \not\in \ker \beta$. 
\end{proof}

To complete the argument, we show that every very new form in $\ker U_2$ is in the image of $U_N'$:
\begin{mylemma}\label{shaunakveryclever}\ \\
Let $f$ be in $K(N)^\elleng$ for a prime $N \not \equiv 1 \cmod{8}$. There exists $g \in K(N)$ with $U_N' (g) = f$. 
\end{mylemma}
\begin{proof}
Write $f = f_1 + f_3 + f_5 + f_7$, with $f_i \in K(N)$ as in \cref{gradingthm}. Since $U_N$ is an $N$-graded operator and $f \in K(N)^\elleng = \ker (U_N + 1)$, we have $U_N f = f$, so that $U_N f_i$ must equal $f_{Ni}$.
For $i \not\equiv 1, N \bmod 8$, let~$f_{1, i} := f_1 + f_i$. Then $U_N' (f_{1, i}) = f$. 
\end{proof}

We deduce the structure of $A(N)^\new$. Write $\tau^\new$ for the composition $G_{\QQ, 2N} \stackrel{\tau^\hecke }{\longrightarrow} A(N) \onto A(N)^\new$. 
\begin{mycor}\label{anewstructure}
Let $N = 3, 5$, and let $Y = \tau^\new(g_{-N}^\pm)$ and {$Z = \tau^\new(ci)$} for any choice of $g_{-N}^\pm$, {complex conjugation} 
$c$, and $i \in I_N$. 
Then the following map is an isomorphism: 
$$\frac{\FF_2\lb y, z, u \rb}{(u^2)} \onto A(N)^\new \qquad \mbox{defined by} \qquad { u \mapsto {U_N'}, \qquad y \mapsto Y \qquad z \mapsto Z}.$$
{If $N = 3$ then we may take  $(Z, Y)=(T_{13}, T_7)$; if $N = 5$  then $(Z, Y)$ may be 
$(T_{11}, T_7)$.} 
\end{mycor}
\begin{proof}
\cref{pfprop} and \eqref{anewpf}. 
\end{proof}

\section{Structure of $A(N, \FF_2)$ for $N = 3, 5$}
Finally, we determine the structure of $A(N)$ for $N = 3, 5$. In other words, we prove \cref{B}.

\begin{mythm}
\label{Astructurethm}
Let $N = 3$ or $5$. Choose any $Y = {\tau^\hecke}(g_{-N}^\pm)$, 
and $Z = \tau^\hecke(ci)$ for any choice of 
$c$ and any $i \in I_N$. Recall that $F_N = {\tau^\hecke}(\frob_N)$. Then the map 
$$\frac{\FF_2 \lb x, y, z, w \rb} {(xz, xw, (z + w)^2 )} \onto A(N) \qquad \mbox{induced by} \quad { x \mapsto F_N, \quad  y \mapsto Y, \quad  z \mapsto Z,\quad  w \mapsto U_N Z}$$
is an isomorphism of $(\ZZ/8\ZZ)^\times$-graded algebras. Here the grading on $\frac{\FF_2 \lb x, y, z, w \rb} {(xz, xw, (z + w)^2)}$ is defined by $x$ having grading $N$, $y$ and $w$ having grading $-N$, and $z$ having grading $7$. 
\end{mythm}

\begin{proof} {Denote $\tau^\hecke$ by $\tau$ throughout the proof.}
From \cref{genN}, $A(N)$ is topologically generated by~$F_N$,~$Y$,~$Z$, and any 
$Y' = \tau(g_{-N}^{-\eps})$ chosen so that $Y = \tau(g_{-N}^\eps)$. 
First we adjust these generators slightly. Recall that \eqref{funcp} gives us that, for any $i \in I_N$, 
$$U_NZ=U_N \,\tau(c i) = \tau(c\,i\, \frob_N) - \tau(c\,\frob_N);$$ 
moreover, the image of the set $\{Y,Y'\}$ in $\mm(N)/\mm(N)^2$ is the same as the image of the set 
$\{\tau(c\,i\, \frob_N), \tau(c\,\frob_N)\}$.
 Therefore we can replace $Y'$ by $U_N Z$ in the generating set. 

Now let ${\FF_2 \lb x, y, z, w \rb}$ be an abstract power series ring endowed with the grading as described in the statement of \cref{Astructurethm}. By the discussion above, the map $\gamma$ defined by sending $x \mapsto F_N$, $y \mapsto Y$, $z \mapsto Z$, and~$w \mapsto {U_N Z}$ gives us a surjection of $(\ZZ/8\ZZ)^\times$-graded $\FF_2$-algebras $\gamma: {\FF_2 \lb x, y, z, w \rb} \onto A(N)$. We now view $A(N)$ as a subalgebra of $A(N)^\pf \simeq \FF_2 \lb y_1, z_1, u_1 \rb/(z_1u_1^2)$ via $\delta$, the inverse of the map $\beta$ from \cref{pfprop}, in order to understand $\ker \gamma$: 
$$\FF_2 \lb x, y, z, w \rb \stackrel{\gamma}\onto A(N) \subset A(N)^\pf \stackrel{\delta}\toiso \FF_2 \lb y_1, z_1, u_1 \rb/(z_1u_1^2).$$
Here the $y$ and the $z$ of the first ring 
 map to the image of {$y_1$ and $z_1$}, respectively, of 
 the last.
 
We show that $(w + z)^2$, $xz$ and $xw$ are all in $\ker \gamma = \ker \delta \circ \gamma$. The first is simple: $\delta\circ\gamma(w + z) = \delta(U_N Z + Z) = {{u_1}{z_1}}$, so that $\delta\circ\gamma\big((w + z)^2\big) = 0$ and $(w + z)^2 \in \ker \gamma$. For the second and third, recall that $F_N = F_N U_N' + (U_N')^2$~\eqref{fun2}, so that 
$\delta\circ\gamma(x(w + z)) = \delta(F_N U_N' Z) = 0$, as $F_N$ is a multiple of $U_N'$ and $(U_N')^2 Z$ is in $\ker \delta$. Moreover~$\delta \circ \gamma (xz) = \delta(F_N Z) = \delta( F_N U_N' Z + (U_N')^2 Z) = 0.$ Therefore both $xw$ and $xz$ are in $\ker \gamma$. 

Finally, we claim that $\ker \gamma = \big((w + z)^2, xw, xz\big)$. For simplicity, reparametrize by letting $w_0 = w + z$ (note that this is not a graded parameter!) and {rephrase the claim in the following way: consider the map} $$\eta : \FF_2 \lb x, y, z, w_0 \rb \to \FF_2\lb y_1, z_1, u_1 \rb$$ sending $x \mapsto u_1^2/(1 + u_1)$, $y \mapsto y_1$, $z \mapsto z_1$, and $w_0 \mapsto z_1 u_1$. Then our claim $\ker \gamma = \big((w + z)^2, xw, xz\big)$ is equivalent to the claim $\eta\big((w_0^2, xw_0, xz)\big) = (z_1 u_1^2)$.
 Note that we have shown that $\eta\big((w_0^2, xw_0, xz)\big) \subset (z_1 u_1^2)$; our goal is to show that there is nothing else in $\eta^{-1}\big( (z_1 u_1^2)\big)$. Observe that any $f$ in $\FF_2\lb x, y, z, w_0\rb$ is equivalent modulo $\big(w_0^2, xw_0, xz\big)$ to a power series of the form $g = a(x, y) + w_0 b(y, z) + c(y, z)$ for some $a \in \FF_2 \lb x, y \rb$ and~$b, c \in \FF_2 \lb y, z \rb$. Then $\eta(g) = a\big({u_1^2}/{(1 + u_1)}, y_1\big) + z_1 u_1 b(y_1, z_1) + c(y_1, z_1)$. By inspection it is clear that~$\eta(g)$ is a multiple of $z_1 u_1^2$ only if $g = 0$. Therefore $\eta^{-1}\big( (z_1 u^2)\big)= (w_0^2, xw_0, xz)$, as claimed. Returning to our original graded parametrization, we have shown that $\ker \gamma = \big((w + z)^2, xw, xz\big)$. Note that this is a graded ideal, as expected. 
\end{proof} 

Note that the element $z + w$ in $\frac{\FF_2 \lb x, y, z, w \rb} {(xz, xw, (z + w)^2 )}$ is nilpotent, and the reduced quotient $$\left(\frac{\FF_2 \lb x, y, z, w \rb} {\big(xz, xw, (z + w)^2 \big)}\right)^{\!\!\red} \simeq \frac{\FF_2 \lb x, y, z\rb}{(xz)}$$ matches the results of \cref{mainthm}. 

\section{An $R=\mathbb{T}$ theorem for $A(N, \FF_2)_\one^\pf$}\label{rtpfsec}
In this section we prove an $R=\mathbb{T}$ theorem for the partially full Hecke algebra $A(N)^{\pf} = A(N, \FF_2)_\one^\pf$ for~$N=3,5$. To construct a deformation ring surjecting onto $A(N)^\pf$ we interpolate the deformation conditions of Wake--Wang-Erickson \cite{WWE} and Calegari--Specter \cite[arxiv source file]{CalegariSpecter}.

Let $U$ be a formal variable, and let $\RRR(N)_\one^\pf := \RRR(N,\FF_2)_{\one}^\pf$ be the quotient of the polynomial ring~$\RRR(N)_{\one}[U]$ by the closed ideal $J$ generated by the following elements:
\begin{enumerate}[itemsep = 5pt, listparindent = 0pt, parsep = 2pt]
\item\label{cond1} 
$U^2-\tau^{\univ}(\Frob_N)\,U+1$,
\item\label{cond2} 
$\tau^{\univ}(g\Frob_N i)-\tau^{\univ}(g \Frob_N) - U\big(\tau^{\univ}(gi)-\tau^{\univ}(g)\big),
\qquad  \mbox{for $g \in G_{\QQ, Np}$ and $i \in I_N$};$
\item\label{cond3} 
$\tau^{\univ}(g\Frob_N i)-\tau^{\univ}(g \Frob_N) - U^{-1}\big(\tau^{\univ}(gi)-\tau^{\univ}(g)\big), 
\qquad \mbox{for $g \in G_{\QQ, Np}$ and $i \in I_N$}.$

Here $U^{-1} := \tau^{\univ}(\Frob_N) - U$, as suggested by \eqref{cond1}. 
\end{enumerate}

Then 
$\RRR(N)^\pf_\one$ is a complete Noetherian local $\FF_2$-algebra with residue field $\FF_2$.
Moreover, given a profinite local $\FF_2$-algebra $B$, the morphisms $\RRR(N)^\pf_{\one} \to B$ correspond to tuples $(t,\alpha)$, where $t$ is a $B$-valued pseudorepresentation with constant determinant and level-$N$ shape deforming $\one$ and $\alpha \in B$ is an element satisfying the relations \eqref{cond1}--\eqref{cond3} given above. 

We are now ready to state and prove the main result of this section.
\begin{mythm}\label{rt}
For $N= 3, 5$, the tuple $(\tau^{\pf}, U_N)$ induces an isomorphism $$\RRR(N)^\pf_{\one} \simeq A(N)^{\pf}.$$
\end{mythm}

\begin{proof}
Write $\tau^\pf: G_{\QQ, 2N} \to A(N)^\pf$ for the modular pseudorepresentation $\tau^\hecke: G_{\QQ, 2N} \to A(N)$ composed with the natural inclusion $A(N) \into A(N)^\pf$. By \eqref{charpoly}, \cref{fourlemma}, and \cref{wwerk}, the element $U_N$ of $A(N)^\pf$ satisfies \eqref{cond1}--\eqref{cond3}, inducing the surjective morphism 
$$\varphi^\pf: \RRR(N)^\pf_\one \onto A(N)^\pf.$$
Conditions \eqref{cond1}-\eqref{cond3}, 
together with the proof of \cref{pflem}, implies that the cotangent space of $\RRR(N)^\pf_{\one}$ is spanned by the images of $U+1$, $\tau^{\univ}(ci)$ and $\tau^{\univ}(c\Frob_N)$.
Note that $$\varphi^\pf(X+1) =U_N',\quad  \varphi^\pf\big(\tau^{\univ}(ci)\big) = \tau^{\hecke}(ci), \quad \mbox{and} \quad \varphi^\pf\big(\tau^{\univ}(c\Frob_N)\big) = \tau^{\hecke}(c\Frob_N).$$
From \cref{pfprop}, we know that there exists a surjective map $$\beta : \FF_2\lb y,z,u \rb \to A(N)^{\pf}  \quad\text{ sending }\quad y \mapsto \tau^{\hecke}(c\Frob_N), \quad z \mapsto \tau^{\hecke}(ci), \quad u \mapsto U_N'.$$
Therefore, the surjective map
$$\beta' : \FF_2\lb y,z, u \rb \to \RRR(N)^\pf_{\one} \quad \text{ sending } \quad y \mapsto \tau^{\univ}(c\Frob_N), \quad z \mapsto \tau^{\univ}(ci), \quad u \mapsto U+1.$$
is a lift of $\beta$: that is, $\varphi^\pf \circ \beta' = \beta$. Combining conditions \eqref{cond2} and \eqref{cond3} 
for $g =c$, we obtain $U\tau^{\univ}(ci)=U^{-1}\tau^{\univ}(ci)$: i.e., $(U+1)^2\,\tau^{\univ}(ci)=0$.
Thus $zu^2 \in \ker \beta'$.
On the other hand, by \cref{pfprop}, $\ker \beta=(zu^2)$.
Since~$\varphi^\pf \circ \beta' = \beta$ and $\beta'$ is surjective, it follows that $\ker \beta = \ker \beta'$ so that $\varphi^\pf$ is injective. Since $\varphi^\pf$ is surjective by construction, our claim is proved.
\end{proof}

Note that, for $N=3,5$, the trace algebras of $\tau^{\hecke}$ and $\tau^{\univ}$ in $A(N)^\pf$ and $\RRR(N)^\pf_\one$ are $A(N)$ and $\RRR(N)_\one$, respectively.
So we obtain the following immediate corollary.
\begin{mycor}\label{ANinRNpf}
For $N=3,5$, under the isomorphism obtained in \cref{rt}, $A(N)$ is isomorphic to the image of $\RRR(N)_{\one}$ in $\RRR(N)^\pf_{\one}$.
\end{mycor}

\section{Complements and questions}
In this last section we gather some easy-to-deduce information about Hecke algebras and deformation rings of levels closely related to $\Gamma_0(3)$ and $\Gamma_0(5)$. We also include some unanswered questions.

\subsection{Hecke algebras for $\Gamma_0(9)$ and $\Gamma_0(25)$ modulo 2}
The analysis we have already done allows us to determine the structure of the reduced quotients $A(9, \FF_2)^\red$ and $A(25, \FF_2)^\red_\one$ with minimal additional work. 

\begin{myprop} If $N = 3$ or $5$, then we have an 
isomorphism 
$$\hat \RRR(N, \FF_2)_\one^\red \simeq  A(N^2, \FF_2)_\one^\red \simeq A(N, \FF_2)_\one^\red$$
\end{myprop}

Note that $A(9, \FF_2) = A(9, \FF_2)_\one$ is a local ring, whereas $A(25, \FF_2)$ has two local components.
\begin{proof} 
On the deformation side, the rings $\hat \RRR(N^2, \FF_2)_\one$ and $\hat \RRR(N, \FF_2)_\one$ coincide by definition. On the Hecke side, restriction to modular forms of level $N$ induces a surjective morphism $A(N^2, \FF_2)_\one \onto A({N}, \FF_2)_\one$. Combining these with the map $\hat \phi$ from \eqref{RtoA} gives us 
\begin{equation}\label{thang}
\hat \RRR(N, \FF_2)_\one = \hat \RRR(N^2, \FF_2)_\one \stackrel{\hat \phi} \onto A(N^2, \FF_2)_\one \onto A({N}, \FF_2)_\one.
\end{equation}
Taking reduced quotients, along with \cref{isomcor},
completes the proof.
\end{proof}

Since characteristic-zero Galois representations attached to a form of level $N^2$ --- for example, to the twist of a level-one eigenform by a Dirichlet character modulo $N$ --- need not have level-$N$ shape, we do not expect~$\hat \phi$ to factor through $\RRR(N, \FF_2)_\one$.

\subsection{Hecke algebras of for $\Gamma_1(3)$  and $\Gamma_1(5)$ modulo 2}
We can achieve similar results for $\Gamma_1(3)$ and~$\Gamma_1(5)$. The construction of the spaces of mod-$2$ modular forms and the Hecke algebra are similar to those for~$\Gamma_0(N)$ described in \cref{modformspace,heckealgmodpsec}. For odd $N$ we define spaces $M_k\big(\Gamma_1(N), \FF_2\big)$ of \mbox{mod-$2$} modular forms of level $\Gamma_1(N)$ as reductions of $q$-expansions of characteristic-zero $\Gamma_1(N)$-modular forms of weight $k$ whose $q$-expansion at the infinity cusp is integral. We also let $A_k\big(\Gamma_1(N), \FF_2\big)$ be the Hecke algebra generated by the action of the Hecke operators prime to $2N$ acting on $M_k\big(\Gamma_1(N), \FF_2)$. As described in \cref{weightsec}, we have embeddings $M_k\big(\Gamma_1(N), \FF_2\big) \into M_{k + 1}\big(\Gamma_1(N), \FF_2\big)$ induced by multiplication by~$E_{1, \chi_N}$ lifting the Hasse invariant, so that restriction defines  projections $A_{k+1}\big(\Gamma_1(N), \FF_2\big) \onto A_k\big(\Gamma_1(N), \FF_2\big)$, and we can form the Hecke algebra 
$$A\big(\Gamma_1(N), \FF_2\big) := \varprojlim_k A_k\big(\Gamma_1(N), \FF_2\big) \quad\mbox{acting on} \quad 
M\big(\Gamma_1(N), \FF_2\big) := \bigcup_k M_k\big(\Gamma_1(N), \FF_2\big).$$
Both the Hecke algebra and the space of forms break up into $\rhobar$ components as in \eqref{Asplit}, \eqref{Msplit}. 

\begin{myprop} For $N = 3, 5$ restriction to forms of level $\Gamma_0(N)$ induces the isomorphism 
$$A\big(\Gamma_1(N), \FF_2\big)_\one^\red \simeq A(N, \FF_2)_\one^\red$$

For $N = 3$ we further have the isomorphism $A\big(\Gamma_1(N), \FF_2\big)_\one \simeq A(N, \FF_2)_\one$.
\end{myprop}

\begin{proof} 
Although the modular pseudorepresentation taking values in $A\big(\Gamma_1(N), \FF_2\big)_\one$ does not have constant determinant, and hence does not factor through $\hat \RRR(N, \FF_2)_\one$, its image in $A\big(\Gamma_1(N), \FF_2\big)_\one^\red$ does \cite[Lemma~5~and~ff.]{D}, so that universality gives us a surjective morphism 
$$\hat \RRR(N, \FF_2)_\one \onto A\big(\Gamma_1(N), \FF_2\big)_\one^\red.$$
On the other hand, restriction from forms of level $\Gamma_1(N)$ to forms of level $\Gamma_0(N)$ gives a surjective map $A\big(\Gamma_1(N), \FF_2\big)_\one\onto A(N, \FF_2)_\one$. {Now take reduced quotients and apply \cref{isomcor}.}

The second statement comes from equality on the space of forms. Indeed, for $k$ even, we already have $M_k(3, \ZZ_2) = M_k\big(\Gamma_1(3), \ZZ_2\big)$ since the only even Dirichlet character modulo $3$ is the trivial one. And for~$k$ odd, we have $M_k\big(\Gamma_1(3), \FF_2\big) \subseteq M_{k+1}(3, \FF_2)$ with the embedding induced by multiplying by the Hasse invariant lifted by $E_{1, \chi}$, where $\chi$ is the nontrivial mod-$3$ Dirichlet character. Thus $M(\Gamma_1(3), \FF_2) = M(3, \FF_2)$, and the claim follows.
\end{proof}

\subsection{Remaining questions}\label{questions}
We close with a number of questions not answered in this text.
\begin{enumerate}[itemsep = 5pt]
\item Can \cref{a1,B} be generalized to all primes $N \equiv 3, 5 \mod{8}$? See \cref{thing} for a discussion of the limitations of our methods.
\end{enumerate}
Let $\RRR'(N, \FF_2)_\one$ be the deformation ring with {conditions \eqref{wwe} and \eqref{wwe2} imposed}. 
Recall that $\RRR(N, \FF_2)_\one$ is the quotient of $\hat \RRR(N, \FF_2)_\one$ subject to the purely local level-$N$ shape condition \eqref{levelNintro}, whereas $\RRR'(N, \FF_2)_\one$ a priori satisfies a global condition (see also the discussion in \cref{levelnshapeintro}). One can show that the surjection from $\RRR(N, \FF_2)_\one$ to $A(N, \FF_2)_\one$ factors through $\RRR'(N, \FF_2)_\one$, so that there are natural surjective~maps 
\begin{equation}\label{wwemany} \RRR(N, \FF_2)_\one \onto \RRR'(N, \FF_2)_\one \onto A(N, \FF_2)_\one.
\end{equation} 
Furthermore, \cref{isomcor} implies that all three rings have the same reduced quotient. 
\begin{enumerate}[resume, itemsep = 5pt]
\item\label{wwehope} Is the second map in \eqref{wwemany} an isomorphism? 
\item\label{wwehope2}  Does the first surjection in \eqref{wwemany} have a nontrivial kernel? Note that by \cref{ANinRNpf} this map is an isomorphism if and only if the structure map $\RRR(N, \FF_2)_\one \to  \RRR(N, \FF_2)_\one^\pf$ is injective. 
\item For $N = 3, 5$, is the map $\hat \RRR(N, \FF_2)_\one \onto A(N^2, \FF_2)_\one$ from \eqref{thang} an isomorphism? 
\item In general, given a modular representation $\bar\rho : G_{\QQ,Np} \to \gl_2(\FF)$ and a prime $\ell \nmid Np$, can one determine the structure of the level $N\ell$ Hecke algebra $A(N\ell,\FF)_{\bar\rho}$ from the structure of the level $N$ Hecke algebra $A(N,\FF)_{\bar\rho}$?
\end{enumerate}

%%%%%%%%%%%%%%%%%%%%%%%%%%%%%%%%%%%%%%%%%%%%%%%%%%%%%
%%%%%%%%%%%%%%%%%%%%%%%%%%%%%%%%%%%%%%%%%%%%%%%%%%%%%
%%%%%%%%%%%%%%%%%%%%%%%%%%%%%%%%%%%%%%%%%%%%%%%%%%%%%
%\bibliographystyle{alphanum}

%%%%%%%%%%%%%%%%%%%%%%%%%%%%%%%%%%%%%%%%%%%%%%%%%%%%%
%%%%%%%%%%%%%%%%%%%%%%%%%%%%%%%%%%%%%%%%%%%%%%%%%%%%%
%%%%%%%%%%%%%%%%%%%%%%%%%%%%%%%%%%%%%%%%%%%%%%%%%%%%%

\end{document}